\documentclass[reqno, 11pt]{amsart}
\title[Decay estimates for discrete bi-Schr\"{o}dinger operators on the lattice $\Z$]{Decay estimates for discrete bi-Schr\"{o}dinger operators on the lattice $\Z$ }
\author{ Sisi Huang and Xiaohua Yao }
\address{Sisi Huang, Department of Mathematics, Central China Normal University, Wuhan, 430079, P.R. China}
\email{hss@mails.ccnu.edu.cn}

\address{Xiaohua Yao, Department of Mathematics and  Key Laboratory of Nonlinear Analysis and Applications(Ministry of Education), Central China Normal University, Wuhan, 430079, P.R. China}
\email{yaoxiaohua@ccnu.edu.cn}

\thanks{The work is partially supported by NSFC No.12171182}
\usepackage{color}
\usepackage{amsmath}
\usepackage{amssymb}
\usepackage[mathscr]{euscript}
\usepackage{paralist}
\usepackage{amsthm}
\usepackage{ulem}
\usepackage{bm}
\usepackage{extarrows}%ÊäÈë³¤µÈºÅ
\usepackage{diagbox}
\usepackage{multirow}%ºÏ²¢ÐÐ
\usepackage{makecell}%ºÏ²¢ÁÐ
\usepackage{threeparttable}%±í¸ñ½Å×¢
\usepackage{tikz}
\usetikzlibrary{patterns}
\usetikzlibrary{decorations.markings}
\usepackage{caption}
\usepackage{verbatim}

\usepackage{cite}
\usepackage{hyperref}
\newtheorem{definition}{Definition}[section]
\newtheorem{theorem}[definition]{Theorem}
\newtheorem{lemma}[definition]{Lemma}
\newtheorem{remark}[definition]{Remark}
\newtheorem{proposition}[definition]{Proposition}
\newtheorem{claim}[definition]{Claim}
\newtheorem{corollary}[definition]{Corollary}
\newcommand\R{\mathbb{R}}
\newcommand\Z{\mathbb{Z}}
\newcommand\B{\mathbb{B}}
\newcommand\C{\mathbb{C}}
\newcommand\N{\mathbb{N}}
\newcommand\T{\mathbb{T}}
\newcommand\mscH{\mathcal{H}}
\newcommand\mcaF{\mathcal{F}}
\newcommand\mcaI{\mathcal{I}}
\newcommand\mcaJ{\mathcal{J}}

\newcommand\mcaD{\mathcal{D}}

\newcommand\mcaN{\mathcal{N}}
\newcommand\mcaP{\mathcal{P}}
\newcommand\mcaE{\mathcal{E}}
\numberwithin{equation}{section}
\begin{document}

\begin{abstract}
It was known that the discrete Laplace operator $\Delta$ on the lattice $\mathbb{Z}$ satisfies the following sharp time decay estimate:
$$\left\|e^{it\Delta}\right\|_{\ell^1\rightarrow\ell^{\infty}}\lesssim|t|^{-\frac{1}{3}},\quad t\neq0,$$
which is slower than the usual $|t|^{-\frac{1}{2}}$ decay in the continuous case on $\mathbb{R}$. However in this paper,  we have showed that the discrete bi-Laplacian $\Delta^2$ on $\mathbb{Z}$ actually exhibits  the same sharp decay estimate $|t|^{-\frac{1}{4}}$ as its continuous counterpart.

In view of these free decay estimates,  this paper further investigates  the discrete bi-Schr\"{o}dinger operators of the form $H=\Delta^2+V$ on the lattice space $\ell^2(\mathbb{Z})$, where $V(n)$ is a real valued potential of $\mathbb{Z}$. Under suitable decay conditions on $V$ and assuming that both 0 and 16 are regular spectral points of $H$,  we establish the following sharp $\ell^1-\ell^{\infty}$ dispersive estimates:
$$\left\|e^{-itH}P_{ac}(H)\right\|_{\ell^1\rightarrow\ell^{\infty}}\lesssim|t|^{-\frac{1}{4}},\quad t\neq0,$$
where $P_{ac}(H)$ denotes the spectral projection onto the absolutely continuous spectrum space of $H$. Additionally, the following decay estimates for   beam equation  are also derived:
$$\|{\cos}(t\sqrt H)P_{ac}(H)\|_{\ell^1\rightarrow\ell^{\infty}}+\left\|\frac{{\sin}(t\sqrt H)}{t\sqrt H}P_{ac}(H)\right\|_{\ell^1\rightarrow\ell^{\infty}}\lesssim|t|^{-\frac{1}{3}},\quad t\neq0.$$
\end{abstract}

\maketitle

%Éú³ÉÄ¿Â¼
\tableofcontents

%ÉèÖÃ¼ä¾à
%\baselineskip=14pt
\section{Introduction and main results}
\subsection{Introduction}
Let $\ell^2(\Z)$ denote the complex Hilbert space  consisting of square summable sequences $\{\phi(n)\}_{n\in\Z}$. The non-negative discrete Laplacian $-\Delta$ is defined as
\begin{equation*}
\left((-\Delta)\phi\right)(n):=-\phi(n+1)-\phi(n-1)+2\phi(n),\quad n\in\Z, \quad\forall\ \phi\in\ell^2(\Z).
\end{equation*}
In this paper, we are devoted to considering the time decay estimates of the solutions for the following fourth order Schr\"{o}dinger equation on the lattice $\Z$:
\begin{equation}\label{Bi-Schrodinger equation}
\left\{\begin{aligned}&i(\partial_tu)(t,n)-(\Delta^2u+Vu)(t,n)=0,\ \ (t,n)\in\R\times\Z,\\
&u(0,n)=\varphi_0(n),\end{aligned}\right.
\end{equation}
and the discrete beam equation:
\begin{equation}\label{Beam equation}
\left\{\begin{aligned}&(\partial_{tt}v)(t,n)+(\Delta^2v+Vv)(t,n)=0,\ \ (t,n)\in\R\times\Z,\\
&v(0,n)=\varphi_1(n),\ \left(\partial_{t}v\right)(0,n)=\varphi_2(n),\end{aligned}\right.
\end{equation}
where $\varphi_j\in\ell^2(\Z)$ for $j=0,1,2$, and $V$ is a real-valued decay potential satisfying $|V(n)|\lesssim \left<n\right>^{-\beta}$ for some $\beta>0$ with $\left<n\right>=(1+|n|^2)^{\frac{1}{2}}$.

 The discrete bi-Laplace operator $\Delta^2$ on the lattice $\Z$ is the discrete analogue of the fourth-order differential operator $\frac{d^4}{dx^4}$ on the real line. The equations \eqref{Bi-Schrodinger equation} and \eqref{Beam equation} are discretizations of classical continuous models studied in \cite{SWY22} and \cite{CLSY25}, respectively. These discretizations not only serve as numerical tools in computational mathematics but also hold profound significance in mathematics physics, particularly in quantum physics. For instance, discrete Schr\"{o}dinger equations are standard
models for random media dynamics, as discussed in Aizenman-Warzel \cite{AW15}, while the discrete beam equation describes the deformation of elastic beam under  certain force (cf. \"{O}chsner \cite{Och21}).

Denote $H:=\Delta^2+V$. Then both $\Delta^2$ and $H$ are bounded self-adjoint operators on $\ell^2(\Z)$, generating the associated unitary groups $e^{-it\Delta^2}$ and $e^{-itH}$, respectively. The solutions to equations \eqref{Bi-Schrodinger equation} and \eqref{Beam equation} are given as follows:
\begin{equation}\label{solutions for Bi-Schrodinger}
u(t,n)=e^{-itH}\varphi_0(n),
\end{equation}
\begin{equation}\label{solutions for Beam-equation}
v(t,n)={\rm cos}(t\sqrt H)\varphi_1(n)+\frac{{\rm sin}(t\sqrt H)}{\sqrt H}\varphi_2(n).
\end{equation}
The expression \eqref{solutions for Beam-equation} above depends on the branch chosen of $\sqrt z$ with $\Im z\geq0$, so the solution $v(t,n)$ is well-defined even if $H$ is not positive. In the sequel, we are devoted to establishing the time decay estimates of the propagator operators $e^{-itH}$, ${\rm cos}(t\sqrt H)$ and $\frac{{\rm sin}(t\sqrt H)}{\sqrt H}$.

\begin{comment}
one obtains
\begin{equation}\label{unitary equivalent}
(\mcaF\Delta^2\phi)(x)=(2-2{\rm cos}x)^2(\mcaF\phi)(x):=M(x)(\mcaF\phi)(x),\quad  x\in\T=[-\pi,\pi],
\end{equation}
which implies that the spectrum $\sigma(\Delta^2)$ of $\Delta^2$ is $[0,16]$. And by Wely's theorem, the essential spectrum $\sigma_{ess}(H)$ of $H$ is
$$\sigma_{ess}(H)=\sigma_{ess}(\Delta^2)=\sigma(\Delta^2)=[0,16].$$
\end{comment}

For the free case,~i.e., $V\equiv0$, then $\sqrt{H}=-\Delta$. By virtue of  Fourier transform,
 %$e^{-it\Delta}$ and $e^{-it\Delta^2}$ is the multiplier $e^{-it(2-2{\rm cos}\theta)}$,
 it is well-known that the following sharp decay estimates hold (cf. \cite{SK05}):
\begin{equation}\label{eitlaplacian}
\left\|e^{it\Delta}\right\|_{\ell^1\rightarrow\ell^{\infty}}\lesssim|t|^{-\frac{1}{3}},\quad t\neq0.
\end{equation}
As a consequence of \eqref{eitlaplacian}, one has
\begin{equation}\label{cos-sin lapacian}
\|{\rm cos}(t\Delta )\|_{\ell^1\rightarrow\ell^{\infty}}+\left\|\frac{{\rm sin}(t\Delta )}{t\Delta }\right\|_{\ell^1\rightarrow\ell^{\infty}}\lesssim|t|^{-\frac{1}{3}}.
\end{equation}
It is worth noting that the decay estimates \eqref{eitlaplacian} and \eqref{cos-sin lapacian} are slower than the usual $|t|^{-\frac{1}{2}}$ in the continuous case, cf.\cite{GS04}. However, interestingly, for the discrete bi-Laplacian on $\Z$, we can prove that
\begin{equation}\label{eitlap2}
\left\|e^{-it\Delta^2}\right\|_{\ell^1\rightarrow\ell^{\infty}}\lesssim|t|^{-\frac{1}{4}},
\end{equation}
which is sharp and the same as the continuous case on the line \cite{SWY22}, for details see Section \ref{Sec of decay for free}.

When $V\not\equiv0$, the decay estimates for the solution operators of equation \eqref{Beam equation} are affected by the spectrum of $H$, which in turn depends on the conditions of potential $V$.
\begin{comment}
Under the aforementioned decay condition on potential $V$, one can obtain that the essential spectrum $\sigma_{ess}(H)$ of $H$ is
$$\sigma_{ess}(H)=\sigma_{ess}(\Delta^2)=\sigma(\Delta^2)=[0,16],$$
which follows from Weyl's theorem and the fact
\begin{equation}\label{unitary equivalent}
(\mcaF\Delta^2\phi)(x)=(2-2{\rm cos}x)^2(\mcaF\phi)(x):=M(x)(\mcaF\phi)(x),\quad  x\in\T=[-\pi,\pi].
\end{equation}
Furthermore, if $\beta>1$, we will prove in Theorem \ref{LAP-theorem} that the singular spectrum $\sigma_{sc}(H)\cap(0,16)=\varnothing$ and the point spectrum $\sigma_p(H)\cap(0,16)$ is discrete with each eigenvalue having finite multiplicities.
\end{comment}
In this paper, we assume that the potential $V$ has fast decay and $H$ has no embedded positive eigenvalues in the continuous spectrum interval $(0,16)$. Under such assumptions, let $\lambda_j$ be the discrete eigenvalues of $H$ and $H\phi_j=\lambda_j\phi_j$ for $\phi_j\in\ell^2(\Z)$, $P_{ac}(H)$ denote the spectral projection onto the absolutely continuous spectrum of $H$ and $P_j$ be the projection on the eigenspace corresponding to the discrete eigenvalue $\lambda_j$. Then the solutions of the equations \eqref{Bi-Schrodinger equation} and \eqref{Beam equation} can be respectively further written as
\begin{equation}\label{decomp of eitH}
u(t,n)=\sum\limits_{j}e^{-it\lambda_j}P_j\varphi_{0}(n)+e^{-itH}P_{ac}(H)\varphi_0(n):=u_{d}(t,n)+u_{c}(t,n),
\end{equation}
\begin{equation}\label{decom of v}
v(t,n)=v_{d}(t,n)+v_{c}(t,n),
\end{equation}
where
\begin{align*}
v_{d}(t,n)&=\sum\limits_{j}^{}{\rm cosh}(t\sqrt{-\lambda_j})\left<\varphi_1,\phi_j\right>\phi_j(n)+\frac{{\rm sinh}(t\sqrt{-\lambda_j})}{\sqrt{-\lambda_j}}\left<\varphi_2,\phi_j\right>\phi_j(n),\\%\label{expre of u2d}\\
v_{c}(t,n)&={\rm cos}(t\sqrt{H})P_{ac}(H)\varphi_1(n)+\frac{{\rm sin}(t\sqrt H)}{\sqrt H}P_{ac}(H)\varphi_2(n).%\label{expre of u2c}
\end{align*}
Observe that the discrete part $u_{d}(t,n)$ of $u$ has no any time decay estimates. Similarly,  the existence of discrete negative/positive eigenvalues of $H$ will lead to the exponential growth/dissipation of $v_{d}(t,n)$ as $t$ becomes large.
%In particular, we notice that the absence of positive eigenvalue of $H$ in $(0,16)$ has been an indispensable assumption in deriving of dispersive estimates.
Therefore, the main goal of this paper is to investigate the time decay estimates for the continuous components $u_c(t,n)$ and $v_c(t,n)$ in \eqref{decomp of eitH} and \eqref{decom of v} under certain decay conditions on potential $V$ and  assuming that thresholds $0$ and $16$ are regular points of $H$ (see Definition \ref{defin of regular point 1} below).

 To achieve this, we will make use of Stone's formula. We first establish the limiting absorption principle for the operator $H$ and then study the asymptotic expansions of resolvent $R^{\pm}_V(\lambda)$ near $\lambda=0$ and $\lambda=16$ under the regular conditions.   Finally, we employ the Van der Corput Lemma to derive the desired estimates.
\subsection{Main results}
For $a,b\in\R^{+}$, $a\lesssim b$ means $a\leq cb$ with some constant $c>0$. Let $\sigma\in\R$, denote by $W_{\sigma}(\Z)$:=$\underset{s>\sigma}{\bigcap}\ell^{2,-s}(\Z)$ the intersection space, where
$$\ell^{2,s}(\Z)=\Big\{\phi=\left\{\phi(n)\right\}_{n\in\Z}:\|\phi\|^2_{\ell^{2,s}}=\sum_{n\in\Z}^{}\left<n\right>^{2s}|\phi(n)|^2<\infty\Big\}.$$
Note that $W_{\sigma_2}(\Z)\subseteq W_{\sigma_1}(\Z)$ if $\sigma_2<\sigma_1$ and $\ell^2(\Z)\subseteq W_0(\Z)$.
\begin{definition}\label{defin of regular point 1}
{\rm Let $H=\Delta^2+V$ be defined on the lattice $\Z$ and $|V(n)|\lesssim \left<n\right>^{-\beta}$ for some $\beta>0$. Then
\begin{comment}
\begin{itemize}
\item [(i)] We say that $0$ is a regular point of $H$ if no solution $\phi$ of equation $H\phi=0$ exists in $W_{\frac{3}{2}}(\Z)$.
\item [(ii)] We say that $16$ is a regular point of $H$ if no solution $\phi$ of equation $H\phi=16\phi$ exists in $W_{\frac{1}{2}}(\Z)$.
\end{itemize}
\end{comment}
\begin{itemize}
\item [(i)] We say that {\bf{$0$ is a regular point of}} $H$ if the discrete equation $H\phi=0$ has no solution in $W_{\frac{3}{2}}(\Z)$.
%\vskip0.1cm
\item [(ii)] We say that {\bf{$16$ is a regular point of}} $H$ if the discrete equation $H\phi=16\phi$ has no solution in $W_{\frac{1}{2}}(\Z)$.
\end{itemize}
}
\end{definition}
\begin{remark}
{\rm We remark that the regular condition introduced above shares similarity with the concept of a generic potential as discussed in \cite{PS08} for the discrete Schr\"{o}dinger operator $-\Delta+V$ on $\Z$. In the continuous analogue studied in \cite{SWY22}, the statement that zero is a regular point means that zero is neither an eigenvalue nor a resonance.
}
\end{remark}
The main results  are summarized as follows.
\begin{theorem}\label{main-theorem}
Let $H=\Delta^2+V$ with $|V(n)|\lesssim \left<n\right>^{-\beta}$ for $\beta>15$. Suppose that $H$ has no positive eigenvalues in the interval $\mcaI=\rm{(}0,16\rm{)}$, and let $P_{ac}(H)$ denote the spectral projection onto the absolutely continuous spectrum space of $H$. If both endpoints $0$ and $16$ of $\mcaI$ are regular points of $H$, then the following decay estimates hold:
\begin{equation}\label{eitH decay-estimate}
\|e^{-itH}P_{ac}(H)\|_{\ell^1\rightarrow\ell^{\infty}}\lesssim|t|^{-\frac{1}{4}},\quad t\neq0,
\end{equation}
and
\begin{equation}\label{cos-sin decay-estimate}
\|{\rm cos}(t\sqrt H)P_{ac}(H)\|_{\ell^1\rightarrow\ell^{\infty}}+\left\|\frac{{\rm sin}(t\sqrt H)}{t\sqrt H}P_{ac}(H)\right\|_{\ell^1\rightarrow\ell^{\infty}}\lesssim|t|^{-\frac{1}{3}},\quad t\neq0.
\end{equation}
\end{theorem}
%\vskip0.1cm
\begin{remark}\label{Rem of M}{\rm Some remarks on Theorem \ref{main-theorem} are given as follows:}
{\rm \begin{itemize}
\begin{comment}
    \end{comment}
\item[(i)] When $V\equiv0$, the estimates \eqref{eitH decay-estimate} and  \eqref{cos-sin decay-estimate}  are sharp, as shown in Theorems \ref{D-E for free case} and \ref{theorem of strichartz estimate}. 
  \vskip0.1cm
\item[(ii)] Note that $\Delta^2_{\Z}$ has two thresholds:~$0$~(degenerate) and $16$~(non-degenerate). As shown in  \cite{SWY22}  for the continuous case, the classification of resonances at the degenerate point $0$ is complicated, suggesting that there may be additional technique challenges arise in the discrete setting, which will be discussed elsewhere.   %Given that 
    \vskip0.1cm
\item[(iii)] We notice that the absence of positive eigenvalues  has been an indispensable assumption in deriving all kinds of dispersive estimates. For $H=\Delta^2+V$ on the lattice $\Z$, Horishima and L\H{o}rinczi have demonstrated in \cite{HL14} that $H$ has no eigenvalues in the interval $\mcaI$ for $V(n)=c\delta_0(n)$~$(c\neq 0)$, the $\delta$-potential with mass $c$ concentrated on $n=0$. On the other hand, in contrast to the extensive results on the eigenvalue problems for discrete Schr\"{o}dinger operators $-\Delta+V$, cf.\cite{HMO11,HSSS12,Kru12,BS12,IM14,HHNO16,Man19,Liu19,HMK22,Liu22,LMT24}, more studies are needed to establish the absence of positive eigenvalue for higher order cases.
\end{itemize}
}
\end{remark}
\subsection{The idea of proof}\label{subsec of idea of proof}
In this subsection, we outline the main ideas behind the proof of Theorem \ref{main-theorem}.
Throughout this paper, we denote by $K$ the operator with kernel $K(n,m)$, i.e.,
 $$(Kf)(n):=\sum\limits_{m\in\Z}^{}K(n,m)f(m).$$

To derive Theorem \ref{main-theorem}, based on the following two formulas:
 \begin{equation*}
{\rm cos}(t\sqrt H)=\frac{e^{-it\sqrt H}+e^{it\sqrt H}}{2},\quad \frac{{\rm sin}(t\sqrt H)}{t\sqrt H}=\frac{1}{2t}\int_{-t}^{t}{\rm cos}\left(s\sqrt H\right)ds,
\end{equation*}
 it suffices to show that the estimates \eqref{eitH decay-estimate} and \eqref{cos-sin decay-estimate} hold for $e^{-itH}P_{ac}(H)$ and $e^{-it\sqrt{H}}P_{ac}(H)$, respectively. Using Stone's formula, their kernels are expressed as follows:
\begin{equation}\label{kernel of eitHPacH}
\left(e^{-itH}P_{ac}(H)\right)(n,m)=\frac{2}{\pi i}\int_{0}^{2}e^{-it\mu^4}\mu^3\left[R^{+}_V\left(\mu^4\right)-R^{-}_V\left(\mu^4\right)\right](n,m)d\mu,
\end{equation}
\begin{equation}\label{kernel of eitsqrtHPacH}
\left(e^{-it\sqrt H}P_{ac}(H)\right)(n,m)=\frac{2}{\pi i}\int_{0}^{2}e^{-it\mu^2}\mu^3\left[R^{+}_V\left(\mu^4\right)-R^{-}_V\left(\mu^4\right)\right](n,m)d\mu.
\end{equation}
Notice that the  difference between \eqref{kernel of eitHPacH} and \eqref{kernel of eitsqrtHPacH} lies in the power of $\mu$ in the exponent. This change affects the decay rate, which shifts from $\frac{1}{4}$ to $\frac{1}{3}$.
%And the change in decay rate is demonstrated in the proof of Theorem \ref{D-E for free case} dealing with the estimates for free case, with more details provided in Section \ref{Sec of decay for free}.
 In the following discussion, due to the similarity, we only address three fundamental problems that arise in the estimate of \eqref{kernel of eitHPacH}.
\subsubsection{Limiting absorption principle}\label{subsubsection of LAP}
According to \eqref{kernel of eitHPacH}, the first difficulty is to show the existence of boundary value $R^{\pm}_{V}(\mu^4)$ for any $\mu\in(0,2)$.

It was well-known that the limiting absorption principle (LAP) generally states that the resolvent $R_V(z)$ may converge in a suitable way as $z$ approaches spectrum points, which plays a fundamental role in spectral and scattering theory. For instance, see Agmon's work \cite{Agm75} for the Schr\"{o}dinger operator $-\Delta+V$ in $\R^{d}$. In the discrete setting, the LAP for discrete Schr\"{o}dinger operators $-\Delta+V$ on $\Z^d$ has received much attention (cf.\cite{Esk67,SV01,BS98,BS99,KKK06,KKV08,IK12,Man17,PS08} and references therein). 

However, to the best of our knowledge, it seems that LAP is  open for higher-order Schr\"{o}dinger operators on the lattice $\Z^d$. Hence, based on the commutator estimates \cite{JMP84} and Mourre theory~(cf. \cite{ABG96,Mou81,Mou83}),  we will first demonstrate that under appropriate conditions on $V$, $R^{\pm}_{V}(\mu^4)$ for $H=\Delta^2+V$ exist as  bounded operators from $\ell^{2,s}(\Z)$ to $\ell^{2,-s}(\Z)$ for  $s>1/2$~(see Theorem \ref{LAP-theorem}).

\subsubsection{Asymptotic expansions of $R^{\pm}_{V}(\mu^4)$}\label{subsubsection of asy}
As indicated in Theorem \ref{LAP-theorem}, the second challenge lies in deriving the asymptotic behaviors of $R^{\pm}_{V}(\mu^4)$ near $\mu=0$ and $\mu=2$.

To this end, let $R^{\pm}_{0}\left(\mu^4\right)$ be the boundary value of the free resolvent $R_{0}(z):=(\Delta^2-z)^{-1}$, and define $$M^{\pm}\left(\mu\right)=U+vR^{\pm}_{0}\left(\mu^4\right)v,\quad v(n)=\sqrt{|V(n)|},\ U={\rm sign}\left(V(n)\right),\ \mu\in(0,2),$$
which is invertible on $\ell^2(\Z)$ by the assumption of absence of positive eigenvalues in $\mcaI$ and Theorem \ref{LAP-theorem}. Then
 \begin{equation}\label{reso identity 1}
R^{\pm}_V\left(\mu^4\right)=R^{\pm}_{0}\left(\mu^4\right)-R^{\pm}_0\left(\mu^4\right)v\left(M^{\pm}\left(\mu\right)\right)^{-1}vR^{\pm}_0\left(\mu^4\right),
\end{equation}
from which we turn to study the asymptotic expansions of $\left(M^{\pm}\left(\mu\right)\right)^{-1}$ near $\mu=0$ and $\mu=2$.

The basic idea behind the expansion of $\left(M^{\pm}\left(\mu\right)\right)^{-1}$ is the Neumann  expansion, which in turn depends on the expansion of $R^{\pm}_{0}(\mu^4)$. In this respect, Jensen and Kato initiated their seminal work in \cite{JK79} for Schr\"odinger operator $-\Delta_{\R^3}+V$ on $\R^3$. Since then, the method has been widely applied (cf. \cite{JN01, SWY22}). 
When considering the discrete bi-Laplacian $\Delta^2$ on the lattice $\Z$, we will face two distinct difficulties. 
Firstly, compared with Laplacian $-\Delta_{\Z}$ on the lattice, the threshold $0$ now is a {\bf degenerate critical value}~( i.e., $M(0)=M^{'}(0)=M^{''}(0)=0$, where the symbol $M(x)=(2-2{\rm cos}x)^2$ is defined in \eqref{unitary equivalent}). This degeneracy leads to additional steps to expand the $\left(M^{\pm}\left(\mu\right)\right)^{-1}$. Secondly, in contrast to the continuous analogue \cite{SWY22}, we encounter another threshold 16~(i.e., corresponding to $\mu=2$). 

The kernels of boundary values $R^{\pm}_0(\mu^4)$, as presented in \eqref{kernel of R0 boundary}, are given by
 \begin{align*}
  \ R^{\pm}_0(\mu^4,n,m)=\frac{1}{4\mu^3}\left(\frac{\pm ie^{-i\theta_{\pm}|n-m|}}{\sqrt{1-\frac{\mu^2}{4}}}-\frac{e^{b(\mu)|n-m|}}{\sqrt{1+\frac{\mu^2}{4}}}\right),
  \end{align*}
  where $\theta_{\pm}$ satisfies ${\rm cos}\theta_{\pm}=1-\frac{\mu^2}{2}$ and $b(\mu)={\rm ln} \big(1+\frac{\mu^2}{2}-\mu(1+\frac{\mu^2}{4})^{\frac{1}{2}}\big)$. This can be formally expanded near $\mu=0$ and $2$ respectively as follows:
   \begin{align*}
  R^{\pm}_0\left(\mu^4,n,m\right)&\thicksim \sum\limits_{j=-3}^{+\infty}\mu^{j}G^{\pm}_j(n,m),\quad \mu\rightarrow0,\\
R^{\pm}_0\left((2-\mu)^4,n,m\right)&\thicksim\sum\limits_{j=-1}^{+\infty}\mu^{\frac{j}{2}}\widetilde{G}^{\pm}_j(n,m),\quad \mu\rightarrow0,
 \end{align*}  
where $G^{\pm}_j(n,m)$, $\widetilde{G}^{\pm}_j(n,m)$ are specific  kernels defined in \eqref{expan coeffie}.  
Given these expansions of $R_0^\pm$ above, the asymptotic expansions of $(M^{\pm}(\mu))^{-1}$ can be derived near $\mu=0$ and $\mu=2$ under the assumptions that 0 and 16 are regular thresholds of $H$.

The asymptotic expansions of $\left(M^{\pm}\left(\mu\right)\right)^{-1}$ is presented in Theorem \ref{asy theor of Mpm mu}. Their proofs will be given in Section \ref{proof of asy}.
%\end{comment}

\subsubsection{Treatment of oscillatory integral}
Equipped with the two tools mentioned above, the final step is to handle the oscillatory integral \eqref{kernel of eitHPacH} by Van der Corput Lemma \cite[P. ${332-334}$]{Ste93}.
Specifically,  we decompose \eqref{kernel of eitHPacH} into three parts:
\begin{equation}\label{kernel of eitHPacH(3 sec)}
\left(e^{-itH}P_{ac}(H)\right)=\frac{2}{\pi i}\Big(\int_{0}^{\mu_0}+\int_{\mu_0}^{2-\mu_0}+\int_{2-\mu_0}^{2}\Big)e^{-it\mu^4}\mu^3\left[R^{+}_V\left(\mu^4\right)-R^{-}_V\left(\mu^4\right)\right]d\mu.
\end{equation}
where $\mu_0$ is a sufficient small fixed positive constant.
%Therefore, the final problem is how to deal with the oscillatory integral \eqref{kernel of eitHPacH(3 sec)}. We point out that, different from the continuous case, the basic tool here we will adopt is the Van der Corput Lemma \cite[P$_{332-334}$]{Ste93}. Moreover, in many cases, the oscillatory integral will be handled on the frequency space through suitable variable substitution.
 Substituting \eqref{reso identity 1} into the first and third integrals,  and the following \eqref{reso identity 2} into the second integral,
%For the integral \eqref{kernel of eitHPacH(3 sec)}, we replace its first and third terms with the \eqref{reso identity 1} and its second term with the following \eqref{reso identity 2}
\begin{align}\label{reso identity 2}
R^{\pm}_V\left(\mu^4\right)&=R^{\pm}_{0}\left(\mu^4\right)-R^{\pm}_{0}\left(\mu^4\right)VR^{\pm}_{0}\left(\mu^4\right)+R^{\pm}_{0}\left(\mu^4\right)VR^{\pm}_V\left(\mu^4\right)VR^{\pm}_{0}\left(\mu^4\right),
\end{align}
then we obtain %we can further express \eqref{kernel of eitHPacH(3 sec)} as the sum of the following four integral kernels:
\begin{equation}\label{kernel of eitHPacH(4 section)}
\left(e^{-itH}P_{ac}(H)\right)(n,m)=-\frac{2}{\pi i}\sum\limits_{j=0}^{3}(K^{+}_{j}-K^{-}_{j})(t,n,m),
\end{equation}
where
\begin{align}\label{kernels of Ki}
\begin{split}
K^{\pm}_{0}(t,n,m)&=\int_{0}^{2}e^{-it\mu^4}\mu^3R^{\mp}_0\left(\mu^4,n,m\right)d\mu,\\
K^{\pm}_{1}(t,n,m)&=\int_{0}^{\mu_0}e^{-it\mu^4}\mu^3\left[R^{\pm}_0\left(\mu^4\right)v\left(M^{\pm}\left(\mu\right)\right)^{-1}vR^{\pm}_0\left(\mu^4\right)\right](n,m)d\mu,\\
K^{\pm}_{2}(t,n,m)&=\int_{\mu_0}^{2-\mu_0}e^{-it\mu^4}\mu^3\left[R^{\pm}_0\left(\mu^4\right)VR^{\pm}_0\left(\mu^4\right)-R^{\pm}_0\left(\mu^4\right)VR^{\pm}_V\left(\mu^4\right)VR^{\pm}_0\left(\mu^4\right)\right](n,m)d\mu,\\
K^{\pm}_{3}(t,n,m)&=\int_{2-\mu_0}^{2}e^{-it\mu^4}\mu^3\left[R^{\pm}_0\left(\mu^4\right)v\left(M^{\pm}\left(\mu\right)\right)^{-1}vR^{\pm}_0\left(\mu^4\right)\right](n,m)d\mu.
\end{split}
\end{align}
Thus, it suffices to show the decay estimate \eqref{eitH decay-estimate} holds for each component $K^{+}_{j}-K^{-}_j$, which will be dealt with in Section \ref{Sec of decay for free} and Section \ref{Sec of proof}.

\subsection{The organization of paper}
In Section \ref{sec of asy expa}, we prepare some preliminary materials including the basics about free resolvent,
%and present 
the limiting absorption principle (Theorem \ref{LAP-theorem})
and the asymptotic expansions of $(M^{\pm}(\mu))^{-1}$ (Theorem \ref{asy theor of Mpm mu}). Detailed proof of these two theorems are presented in Section \ref{proof of LAP} and Section  \ref{proof of asy}, respectively. 

In Section \ref{Sec of decay for free}, we prove the decay estimate for the free case and demonstrate its sharpness. Section \ref{Sec of proof} focuses on estimating the kernels $(K^{+}_j-K^{-}_{j})(t,n,m)$ defined in \eqref{kernels of Ki} for $j=1,2,3$. Finally, we give a short review of commutator estimates and Mourre theory in Appendix \ref{section of Appendix}.
\section{Asymptotic expansions of $R^{\pm}_{V}(\mu^4)$}\label{sec of asy expa}
\subsection{Free resolvent}\label{Subsec of free resol}
%As mentioned in Subsubsection \ref{subsubsection of asy}, the asymptotic expansions of $R^{\pm}_{V}\left(\mu^4\right)$ near $\mu=0$ and $\mu=2$ can be reduced to those of $\left(M^{\pm}\left(\mu^4\right)\right)^{-1}$ on $\mscH$. Therefore, we initiate this section with this premise.unless otherwise noted, we usually choose the branch of $\sqrt z$ for $z\in\C\setminus[0,\infty)$ with ${\rm Im}\sqrt z>0$.
In this subsection, we will give some preliminaries about $\Delta^2$ on $\Z$. %For any $s,s'\in\R$, denote by $\B(s,s')$  the space of the bounded linear operators from $\ell^{2,s}(\Z)$ to $\ell^{2,s'}(\Z)$.
%To begin with, using the following conventions for the Fourier transform $\mcaF:\ell^2(\Z)\rightarrow L^2(\T), \T=\R/2\pi\Z$,
Define the following Fourier transform $\mcaF$: $\ell^2(\Z)\rightarrow L^2(\T), \T=\R/2\pi\Z$,
\begin{equation}\label{fourier transform}
(\mcaF\phi)(x):=\sum_{n\in\Z}^{}(2\pi)^{-{\frac{1}{2}}}e^{-inx}\phi(n), \quad\forall\ \phi\in\ell^2(\Z),
\end{equation}
then we have
\begin{equation}\label{unitary equivalent}
(\mcaF\Delta^2\phi)(x)=(2-2{\rm cos}x)^2(\mcaF\phi)(x):=M(x)(\mcaF\phi)(x),\quad  x\in\T=[-\pi,\pi],
\end{equation}
which implies that the spectrum of $\Delta^2$ is purely absolutely continuous and equals $[0,16]$.
Let
  $$R_0(z):=(\Delta^2-z)^{-1},\quad z\in\C\setminus[0,16],$$
  be the resolvent of $\Delta^2$ and denote by $R^{\pm}_0(\lambda)$ its boundary value on $(0,16)$, namely,
  \begin{align*}
  R^{\pm}_{0}(\lambda)=\lim\limits_{\varepsilon\downarrow0}R_{0}(\lambda\pm i\varepsilon ),\quad\lambda\in(0,16).
  \end{align*}
   Denote by $\B(s,s')$ the space of all bounded linear operators from $\ell^{2,s}(\Z)$ to $\ell^{2,s'}(\Z)$. Then the existence of $R^{\pm}_0(\lambda)$ as an element of $\B(s,-s)$ for $s>\frac{1}{2}$ follows from the following limiting absorption principle for $-\Delta$~(cf.\cite{KKK06}):
  $$R^{\pm}_{-\Delta}(\mu):=\lim\limits_{\varepsilon\downarrow0}R_{-\Delta}(\mu\pm i\varepsilon )\quad{\rm exists}\ {\rm in\ the\ norm\ of }\ \B(s,-s)\ {\rm for}\ s>\frac{1}{2},\ \mu\in(0,4),$$
 and the resolvent formula:
  \begin{equation}\label{unity partition}
 R_0(z)=\frac{1}{2\sqrt z}\left(R_{-\Delta}(\sqrt z)-R_{-\Delta}(-\sqrt z)\right),\quad\sqrt{z}=\sqrt{|z|}e^{i\frac{arg z}{2}},\  0<argz<2\pi,
 \end{equation}
where $R_{-\Delta}(\omega)=(-\Delta-\omega)^{-1}$ is the resolvent of $-\Delta$.

%Furthermore, based on the following elementary fact for $-\Delta$, we can obtain the kernel of $R^{\pm}_0(\lambda)$.
\begin{lemma}\label{kernel of lapla}
{\cite[Lemma 2.1]{KKK06} For $\omega\in\C\setminus[0,4]$, the kernel of resolvent $R_{-\Delta}(\omega)$ is given by
\begin{equation}\label{kernel of lapl resolvent}
R_{-\Delta}(\omega,n,m)=\frac{-ie^{-i\theta(\omega)|n-m|}}{2{\rm sin}\theta(\omega)},\quad n,m\in\Z,
\end{equation}
where $\theta(\omega)$ is the solution of the equation
\begin{equation}\label{map}
2-2{\rm cos}\theta=\omega
\end{equation}
in the domain $\mcaD:=\left\{\theta(\omega)=a+ib:-\pi\leq a\leq\pi,b<0\right\}$.
}
\end{lemma}
\begin{remark}
{\rm Precisely, let
$C^{\pm}=\{\omega=x\pm iy:y>0\}$ and $\mcaD_{\mp}=\{\theta(\omega)=a+ ib\in\mcaD:\pm a<0\}$. Define directed lines and line segments $\ell_{i},\ell'_{i},\tilde{\ell}_{i}$ as follows:
 \begin{align*}
 &\ell_1=\{x:x\in(-\infty,0)\},\quad\ell_2=\{x:x\in(0,4)\},\quad\ell_3=\{x:x\in(4,\infty)\},\\
 &\ell'_1=\{ib:-\infty<b<0\}, \quad\ell'_2=\{a:a\in(0,\pi)\}, \quad\tilde{\ell}_2=\{a:a\in(-\pi,0)\},\\ &\ell'_{3}=\{\pi+ib:b\in(0,-\infty)\},\quad\tilde{\ell}_3=\{-\pi+ib:b\in(-\infty,0)\},
 \end{align*}
Denote by $\ell^{-}_{i}$ the line with opposite direction of $\ell_{i}$, then the map $\theta(\omega)$ defined in \eqref{map} between $\C\setminus[0,4]\longrightarrow \mcaD~(\omega\mapsto\theta(\omega))$ has the following corresponding relation(see Figure \ref{myplot} below).
%$\ell_2=\{x:0<x<4\}$ $\ell'_2=\{\}$
\begin{figure}[htbp!]\label{figure}
\centering
 \captionsetup{labelformat=empty}
\begin{minipage}{0.4\textwidth}
\centering
\begin{tikzpicture}[>=stealth]
\draw[->] (-2.5,0)--(2.5,0)node[below left]{$x$};%\draw ±íÊ¾»­Ò»ÌõÏß¶Î£¬×ø±êÔ­µãÊÇ(0,0),ÕâÌõÂ·¾¶±íÊ¾´Ó(-4,0) Õâ¸öµãµ½(4,0) Õâ¸öµãÐÎ³ÉµÄÏß¶Î, Ä¬ÈÏµ¥Î»³¤ÊÇ1cm.
\draw[->]  (0,-2.5)--(0,2.5)node[below right]{$y$};
\draw[green,thick, fill=none] (0,0)circle (1.5pt)node[below left, black]{0};
\draw[blue, fill=none] (1,0)circle (1.5pt)node[below, black]{4};
%\definecolor{darkerred}{rgb}{0.8,0,0};
\draw[red,thick](0.06,0)--(0.94,0);
\draw[->,red,thick](0.06,0)--(0.6,0);
%\draw[-{Stealth[length=3pt,width=3pt]}]
%\draw[red,very thick,dashed,line width=0.5pt](0.06,0)--(0.94,0);
\draw[green,thick] (-2.5,0)--(-0.06,0);
\draw[->,green,thick] (-2.5,0)--(-1,0);
\draw[blue,thick] (1.06,0)--(2.5,0);
\draw[->,blue,thick] (1.06,0)--(1.7,0);
%\fill[green](-2.5,0)rectangle(2.5,2.5);
\fill[pattern=north west lines, pattern color=magenta, opacity=0.5] (-2.5,0.06) rectangle (2.5,2.5);
%\node at (-0.2,1.3){$C^{+}=\{x+iy:y>0\}$};
\node at (-1.25,1.25){$C^{+}$};
\fill[pattern=north west lines, pattern color=cyan, opacity=0.9] (-2.5,-2.5) rectangle (2.5,-0.06);
\node at (-1.25,-1.25){$C^{-}$};
\node at (-1.3,0.2){$\textcolor{green}{\ell_1}$};
\node at (0.5,-0.3){$\textcolor{red}{\ell_2}$};
\node at (1.8,-0.3){$\textcolor{blue}{\ell_3}$};
\node[circle, draw, inner sep=1pt] at (2,2) {$\omega$};
\end{tikzpicture}
%\caption{$\omega$\ {\rm plane}}
\end{minipage}
\begin{tikzpicture}[overlay, remember picture]
        \draw[->, thick] (0.2,1)--(2,1); % »æÖÆ¼ýÍ·£¬¸ù¾ÝÐèÒªµ÷Õû³¤¶ÈºÍ·½Ïò
        \node[below, font=\bfseries] at (1,0.9){$\theta(C^{\pm})=\mcaD_{\mp}$};
        \node[below, font=\bfseries] at (1,0.4){$\theta(\ell_i)=\ell'_{i}$};
        \node[below, font=\bfseries] at (1,-0.1){$\theta(\ell^{-}_j)=\tilde{\ell}_{j}$};
        \node[below, font=\bfseries] at (1,-0.7){($i=1,2,3,\ j=2,3$)};
        \node[above, font=\bfseries] at (1,1.06) {${\rm \ \ cos}\theta(\omega)=1-\frac{\omega}{2}$}; % ÔÚ¼ýÍ·ÉÏ·½±ê×¢ÊýÑ§¹ØÏµ
    \end{tikzpicture}
\hspace{0.1\textwidth}
\begin{minipage}{0.4\textwidth}
\centering
\begin{tikzpicture}[>=stealth]
\draw[->] (-2.5,0)--(2.5,0) node[below left]{$a$};%\draw ±íÊ¾»­Ò»ÌõÏß¶Î£¬×ø±êÔ­µãÊÇ(0,0),ÕâÌõÂ·¾¶±íÊ¾´Ó(-4,0) Õâ¸öµãµ½(4,0) Õâ¸öµãÐÎ³ÉµÄÏß¶Î, Ä¬ÈÏµ¥Î»³¤ÊÇ1cm.
\draw[->]  (0,-2.5)--(0,2.5)node[below right]{$b$};
%\node at (2,2){\textcircled{$\theta$}};
\draw[green,thick, fill=none] (0,0)circle (1.5pt)node[below left, black]{0};
\draw[blue, thick,fill=none] (1,0)circle (1.5pt)node[below right, black]{$\pi$};
\draw[blue, thick,fill=none] (-1,0)circle (1.5pt)node[below left, black]{$-\pi$};
\draw[red,thick](0.06,0)--(0.94,0);
\draw[->,red,thick](0.06,0)--(0.56,0);
%\draw[red,thick,dashed,line width=0.5pt](0,0)--(1,0);
%\draw[red,thick,dashed,line width=0.5pt](-1,0)--(0,0);
\draw[red,thick](-0.94,0)--(-0.06,0);
\draw[->,red,thick](-0.94,0)--(-0.4,0);
\draw[green,thick](0,-2.5)--(0,-0.06);
\draw[->,green,thick](0,-2.5)--(0,-1);
\draw[blue,thick](-1,-2.5)--(-1,-0.06);
\draw[->,blue,thick](-1,-2.5)--(-1,-1);
\draw[blue,thick](1,-2.5)--(1,-0.06);
\draw[-<,blue,thick](1,-2.5)--(1,-1);
\fill[pattern=north west lines, pattern color=magenta, opacity=0.5] (-0.95,-2.5) rectangle (-0.06,-0.06);
\node at (-0.5,-1.25){$\mcaD_-$};
%\fill[pattern=north west lines, pattern color=green, opacity=0.5] (-0.95,-2.5) rectangle (-0.06,-0.06);
%\fill[pattern=north west lines, pattern color=fuchsia, opacity=0.5] (-0.95,-2.5) rectangle (-0.06,-0.06);
\fill[pattern=north west lines, pattern color=cyan, opacity=0.9] (0.05,-2.5) rectangle (0.95,-0.06);
\node at (0.5,-1.25){$\mcaD_+$};
\node at (0,-2.75){$\textcolor{green}{\ell'_1}$};
\node at (-1,-2.75){$\textcolor{blue}{\tilde{\ell}_{3}}$};
\node at (-0.6,0.3){$\textcolor{red}{\tilde{\ell}_{2}}$};
\node at (0.6,0.3){$\textcolor{red}{\ell'_{2}}$};
\node at (1,-2.75){$\textcolor{blue}{\ell'_{3}}$};
\node[circle, draw, inner sep=1pt] at (2,2) {$\theta$};
\end{tikzpicture}
%\caption{$\theta$\ {\rm plane}}
\end{minipage}
\caption{Figure 1: The map $\theta(\omega)$ from $\C\setminus[0,4]$ to $\mcaD$.}
\label{myplot}
\end{figure}
}
\end{remark}
Therefore, for any $n,m\in\Z$, it concludes that
\begin{itemize}
\item [(i)] If $\lambda\in (0,4)$, one obtains that
\begin{equation}\label{kernel of lapa boundary}
R^{\pm}_{-\Delta}(\lambda,n,m)=\frac{-ie^{-i\theta_\pm(\lambda)|n-m|}}{2{\rm sin}\theta_\pm(\lambda)},
\end{equation}
\end{itemize}
where $\theta_{\pm}(\lambda)$ satisfies the equation $2-2{\rm cos}\theta=\lambda$ with $\theta_{+}(\lambda)\in(-\pi,0)$, $\theta_{-}(\lambda)\in(0,\pi)$ and $\theta_{+}=-\theta_{-}$.
\begin{itemize}
\item [(ii)] If $\lambda\in(-\infty,0)$, then% $\theta(\lambda)=i\alpha(\lambda)$ with $\alpha(\lambda)<0$, and
    \begin{equation}\label{expre of sin theta}
    {\rm sin}\theta(\lambda)=-i\sqrt{-\lambda+\frac{\lambda^2}{4}}=i\frac{e^{-i\theta(\lambda)}-e^{i\theta (\lambda)}}{2}.
    \end{equation}
\end{itemize}
\begin{lemma}
{ For $\mu\in(0,2)$, the kernel of $R^{\pm}_0(\mu^4)$ is as follows:
\begin{equation}\label{kernel of R0 boundary}
R^{\pm}_{0}(\mu^4,n,m)=\frac{-i}{4\mu^2}\left(\frac{e^{-i\theta_\pm|n-m|}}{{\rm sin}\theta_\pm}-\frac{e^{-i\theta|n-m|}}{{\rm sin}\theta}\right)=\frac{1}{4\mu^3}\left(\frac{\pm ie^{-i\theta_{\pm}|n-m|}}{\sqrt{1-\frac{\mu^2}{4}}}-\frac{e^{b(\mu)|n-m|}}{\sqrt{1+\frac{\mu^2}{4}}}\right),
\end{equation}
where $\theta_{\pm}:=\theta_\pm(\mu^2)$, $\theta:=\theta(-\mu^2)$ and $b(\mu)={\rm ln} \big(1+\frac{\mu^2}{2}-\mu(1+\frac{\mu^2}{4})^{\frac{1}{2}}\big)$.
}
\end{lemma}
  \begin{proof}
 Firstly, for any $\mu\in(0,2)$ and $\varepsilon>0$, let $z=\mu^4\pm i\varepsilon$ in \eqref{unity partition} and take limit $\varepsilon\rightarrow0$, one obtains that
 \begin{equation}\label{R0mu4 and Rdeltamu2}
R^{\pm}_{0}(\mu^4)=\frac{1}{2\mu^2}\left(R^{\pm}_{-\Delta}(\mu^2)-R_{-\Delta}(-\mu^2)\right),
 \end{equation}
 and then based on \eqref{kernel of lapl resolvent}, \eqref{kernel of lapa boundary} and \eqref{expre of sin theta}, the desired \eqref{kernel of R0 boundary} is proved.
  \end{proof}

Roughly speaking, from the second equality in \eqref{kernel of R0 boundary}, one can observe that $R^{\pm}_0(\mu^4)$ exhibits singularity of $\mu^{-3}$ near $\mu=0$ and $(2-\mu)^{-\frac{1}{2}}$ near $\mu=2$. By means of Taylor's expansion and Euler's formula, we can get the formal expansions:
 \begin{align}\label{Taylor expansion}
 R^{\pm}_0\left(\mu^4,n,m\right)\thicksim \sum\limits_{j=-3}^{+\infty}\mu^{j}G^{\pm}_j(n,m),\quad R^{\pm}_0\left((2-\mu)^4,n,m\right)\thicksim\sum\limits_{j=-1}^{+\infty}\mu^{\frac{j}{2}}\widetilde{G}^{\pm}_j(n,m),\quad \mu\rightarrow0,
 \end{align}
 where $G^{\pm}_j(n,m)$, $\widetilde{G}^{\pm}_j(n,m)$ are as follows:
 \begin{itemize}
 \item $G^{\pm}_{-3}(n,m)=\frac{-1\pm i}{4}$,\quad $G^{\pm}_{-2}(n,m)=0$,\quad $G^{\pm}_{-1}(n,m)=\frac{1\pm i}{4}\left(\frac{1}{8}-\frac{1}{2}|n-m|^2\right)$,
 \item $G^{\pm}_{0}(n,m)=\frac{1}{12}\left(|n-m|^3-|n-m|\right)$,
 \item $\widetilde{G}^{\pm}_{-1}(n,m)=\frac{\pm i}{32}(-1)^{|n-m|}$,\quad $\widetilde{G}^{\pm}_{0}(n,m)=\frac{(-1)^{|n-m|}}{32\sqrt2}\left(2\sqrt2 |n-m|-\left(2\sqrt2-3\right)^{|n-m|}\right)$,
 \item for $j\in\N_+$,
     \begin{equation}\label{expan coeffie}
G^{\pm}_{j}(n,m)=\sum\limits_{k=0}^{j+3}c^{\pm}_{k,j}|n-m|^k,\quad c_{k,j}\in\C,
\end{equation}
\begin{equation*}
\widetilde{G}^{\pm}_{j}(n,m)=\sum\limits_{k=0}^{j+1}d^{\pm}_{k,j}(n,m)|n-m|^{k},\ d^{\pm}_{j+1,j}(n,m)=\frac{(\mp2i)^{j}(-1)^{|n-m|}}{16(j+1)!}.
\end{equation*}
 \end{itemize}
 Indeed, we further claim that the expansions \eqref{Taylor expansion} hold in the space $\B(s,-s)$ for suitable $s$. %$R^{\pm}_0(\mu^4)$ near $\mu=0$ and $\mu=2$
\begin{lemma}\label{Puiseux expan of R0}
{Let $N$ be an integer and $\mu\in(0,2)$,
\begin{itemize}
\item [(i)] Suppose that $N\geq-3$ and $s>\frac{1}{2}+N+4$, then
\begin{equation}\label{Puiseux expan of R0 0}
R^{\pm}_0\left(\mu^4\right)=\sum\limits_{j=-3}^{N}\mu^{j}G^{\pm}_j+r^{\pm}_{N}\left(\mu\right),\quad\mu\rightarrow0\ {\rm in}\ \B(s,-s),
\end{equation}
\end{itemize}
where $\left\|r^{\pm}_{N}\left(\mu\right)\right\|_{\B(s,-s)}=O\left(\mu^{N+1}\right)$ as $\mu\rightarrow0$ and $G^{\pm}_j$ are integral operators with kernels given by \eqref{expan coeffie}.
Moreover, in the same sense, the \eqref{Puiseux expan of R0 0} can be differentiated $N+4$ times in $\mu$.
\begin{itemize}
\item[(ii)] Suppose that $N\geq-1$ and $s>\frac{1}{2}+N+2$, then
\begin{equation}\label{Puiseux expan of R0 2}
R^{\pm}_0\left((2-\mu)^4\right)=\sum\limits_{j=-1}^{N}\mu^{\frac{j}{2}}\widetilde{G}^{\pm}_j+\mcaE^{\pm}_{N}\left(\mu\right),\quad\mu\rightarrow0\ {\rm in}\ \B(s,-s),
\end{equation}
\end{itemize}
where $\left\|\mcaE^{\pm}_{N}\left(\mu\right)\right\|_{\B(s,-s)}=O\left(\mu^{\frac{N+1}{2}}\right)$ as $\mu\rightarrow0$, and $\widetilde{G}^{\pm}_j$ are integral operators with kernels given by \eqref{expan coeffie}.
Furthermore, in the same sense, the \eqref{Puiseux expan of R0 2} can be differentiated $N+2$ times in $\mu$.}
\end{lemma}
\begin{proof}
We only deal with (i) since (ii) can follow in a similar way. Given that $N\geq-3$ and $s>\frac{1}{2}+N+4$. Firstly, by Taylor's expansion with remainders, one obtains that
$$r^{\pm}_{N}(\mu,n,m)=\mu^{N+1}\sum\limits_{k=0}^{N+4}a^{\pm}_{k}(\mu)|n-m|^{k},$$
where $a^{\pm}_{k}(\mu)=O(1)$ as $\mu\rightarrow0$. Since that $s>\frac{1}{2}+N+4$ and $|n-m|^{2k}\lesssim\left<n\right>^{2k}\left<m\right>^{2k}$ for any $k\in\N$, we have
$$\sum\limits_{n\in\Z}^{}\sum\limits_{m\in\Z}\left<n\right>^{-{2s}}|n-m|^{2(N+4)}\left<m\right>^{-{2s}}<\infty,$$
then it follows that $\left\|r^{\pm}_{N}\left(\mu\right)\right\|_{\B(s,-s)}=O\left(\mu^{N+1}\right)$. As for the differentiability, note that for each differentiation of \eqref{kernel of R0 boundary}, we just obtain a power of $|n-m|$. Therefore, repeating the process above, we can get the desired conclusion.
\end{proof}
\subsection{Asymptotic expansions of $\left(M^{\pm}\left(\mu\right)\right)^{-1}$}\label{subsection of asympotic}
In the previous Subsection \ref{Subsec of free resol}, we obtained the limiting absorption principle (LAP) for the free case. At the beginning of this subsection, we will establish the LAP under a certain perturbation $V$.
\begin{theorem}\label{LAP-theorem}
{ Let $H=\Delta^2+V$ with $|V(n)|\lesssim \left<n\right>^{-\beta}$ for $\beta>1$ and $\mcaI=(0,16)$. Denote by $[\beta]$ the biggest integer no more than $\beta$, then}
\begin{itemize}
{
\item [(i)] The point spectrum $\sigma_p(H)\cap\mcaI$ is discrete, each eigenvalue has a finite multiplicity and the singular continuous spectrum $\sigma_{sc}(H)=\varnothing$.
\vskip0.2cm
\item [(ii)] Let $j\in\left\{0,\cdots,[\beta]-1\right\}$ and $j+\frac{1}{2}<s\leq[\beta]$, then %for any $\lambda\in\mcaI\setminus\sigma_{p}(H)$, 
the following norm limits
    \begin{equation*}
   \frac{d^{j}}{d\lambda^j}(R^{\pm}_V(\lambda))=\lim\limits_{\varepsilon\downarrow0}R^{(j)}_V(\lambda\pm i\varepsilon) \quad {\rm in} \quad \B(s,-s)
    \end{equation*}
are norm continuous from $\mcaI\setminus\sigma_p(H)$ to $\B(s,-s)$,
    }
\end{itemize}
{ where $R_{V}(z)=(H-z)^{-1}$ is the resolvent of $H$ and $R^{(j)}_{V}(z)$ denotes the jth derivative of $R_{V}(z)$.}
\end{theorem}
%\vskip0.2cm
%\begin{remark}\label{remark of LAP theorem}
% {\rm Some comments on Theorem \ref{LAP-theorem} are provided as below:
 %\begin{itemize}
%\item [(i)] When $V\equiv0$ and $j=0$, our result coincides with that established in Subsection \ref{Subsec of free resol}.
The derivation of this LAP is based on the commutator estimates and Mourre theory~(refer to Appendix \ref{section of Appendix}), with a detailed proof provided in Section \ref{proof of LAP}. The upper bound of $s$  is closely related to the regularity of $H$ (as defined in Definition \ref{def of regularity}).

\vskip0.2cm
Throughout this paper, we assume that $H$ has no positive eigenvalues in $\mcaI$. As a consequence of Theorem \ref{LAP-theorem}, $R^{\pm}_V(\mu^4)$ exists in $\B(s,-s)$ with $\frac{1}{2}<s\leq[\beta]$ for any $\mu\in(0,2)$. In what follows, we will further investigate the asymptotic behaviors of $R^{\pm}_V(\mu^4)$ near $\mu=0$ and $\mu=2$. 
To this end, we introduce
\begin{equation}\label{expr of M mu}
M^{\pm}(\mu):=U+vR^{\pm}_{0}\left(\mu^4\right)v,\quad\mu\in(0,2),\ v(n)=\sqrt{|V(n)|},\ U={\rm sign}\left(V(n)\right),\ n\in\Z,
\end{equation}
and denote by $\left(M^{\pm}\left(\mu\right)\right)^{-1}$ the inverse of $M^{\pm}(\mu)$ as long as it exists.
\vskip0.3cm
\begin{lemma}\label{lemma of inverse of M pm mu}
{  Let $H,V,\mcaI$ be as in Theorem \ref{LAP-theorem}. %and assume that $H$ has no eigenvalues in $\mcaI$.
Then for any $\mu\in(0,2)$, $M^{\pm}\left(\mu\right)$ is invertible on $\ell^2(\Z)$ and satisfies the relation below in $\B(s,-s)$ with $\frac{1}{2}<s\leq\frac{\beta}{2}$,
\begin{equation}\label{inverti relation}
R^{\pm}_V\left(\mu^4\right)=R^{\pm}_{0}\left(\mu^4\right)-R^{\pm}_0\left(\mu^4\right)v\left(M^{\pm}\left(\mu\right)\right)^{-1}vR^{\pm}_0\left(\mu^4\right).
\end{equation}
}
\end{lemma}
\begin{proof}
Firstly, for any $\mu\in(0,2)$, the invertibility of $M^{\pm}(\mu)$ follows from the absence of eigenvalues in $\mcaI$ and Theorem \ref{LAP-theorem}. Then based on the following resolvent identity:
\begin{equation}
R_{V}(z)=R_0(z)-R_0(z)v\left(U+vR_0(z)v\right)^{-1}vR_0(z),
\end{equation}
the relation \eqref{inverti relation} can be deduced from Theorem \ref{LAP-theorem} and the fact that $\left\{v(n)\left<n\right>^{-s}\right\}_{n\in\Z}\in\ell^{\infty}$ by $\frac{1}{2}<s\leq\frac{\beta}{2}$.
\end{proof}
Lemma \ref{lemma of inverse of M pm mu} indicates that one can reduce the asymptotic behaviors of $R^{\pm}_V(\mu^4)$ near $\mu=0$ and $\mu=2$ to those of $\left(M^{\pm}\left(\mu\right)\right)^{-1}$. For this purpose, let
$$\left<f,g\right>:=\sum\limits_{m\in\Z}^{}f(m)\overline{g(m)},\quad f,g\in\ell^2(\Z),$$
and 
\begin{equation}\label{P}
P:=\left\|V\right\|^{-1}_{\ell^1}\left<\cdot,v\right>v,\quad \widetilde{P}:=JPJ^{-1}=\left\|V\right\|^{-1}_{\ell^1}\left<\cdot,\tilde{v}\right>\tilde{v},
\end{equation}
where $\tilde{v}=Jv$ and $J$ is a unitary operator on $\ell^2(\Z)$ given by
\begin{equation}\label{J}
(J\phi)(n)=(-1)^{n}\phi(n), \quad n\in\Z,\ \phi\in\ell^2(\Z).
\end{equation}
We see that $P$ and $\widetilde{P}$ are orthogonal projections onto the span of $v$, $\tilde{v}$ in $\ell^2(\Z)$, respectively, i.e., $P\ell^2(\Z)$ =span$\{v\}$ and $\widetilde{P}\ell^2(\Z)$=span$\{\tilde{v}\}$.

Define $Q,S_0$ and $\widetilde{Q}$ as the orthogonal projections onto the following spaces  respectively:
\begin{align}\label{definition of Q,S0,Qtuta}
\begin{split}
Q\ell^2(\Z):&=\left\{f\in\ell^2(\Z):\left<f,v\right>=0\right\}=\left({\rm span}\{v\}\right)^{\bot},\\
S_0\ell^2(\Z):&=\big\{f\in\ell^2(\Z):\left<f,v_{k}\right>=0,\ v_{k}(n)=n^kv(n),\ k=0,1\big\}=\left({\rm span}\{v,v_1\}\right)^{\bot},\\
\widetilde{Q}\ell^2(\Z):&=\left\{f\in\ell^2(\Z):\left<f,\tilde{v}\right>=0\right\}=\left({\rm span}\{\tilde{v}\}\right)^{\bot}.
\end{split}
\end{align}
Then by definition, it follows that  for any $f\in\ell^2(\Z)$,
\begin{align}
\left<Qf,v\right>&=0,\quad Qv=0,\quad Q=I-P,\label{cancel of Q}\\
\big<\widetilde{Q}f,\tilde{v}\big>&=0,\quad \widetilde{Q}\tilde{v}=0,\quad \widetilde{Q}=I-\widetilde{P},\label{cancel of Qtuta}\\
\left<S_0f,v_k\right>&=0,\quad\ S_0(v_k)=0,\quad k=0,1.\label{cancel of s0}
\end{align}
%We remark that cancellations \eqref{cancel of Q}$\sim$\eqref{cancel of s0} can eliminate some singularity near $\mu=0$ and $\mu=2$ in the estimates of the kernels $K^{\pm}_{i}(t,n,m)(i=1,3)$ defined in \eqref{kernels of Ki}.

Finally, for any $k>0$,  denote by $\Gamma_k(\mu)$  a $\mu$-dependent operator which satisfies
\begin{equation}\label{estimate of Gamma}
\|\Gamma_{k}(\mu)\|_{\B(0,0)}+\mu\left\|\frac{\partial}{\partial\mu}(\Gamma_{k}(\mu))\right\|_{\B(0,0)}\lesssim\mu^{k},\quad \mu>0.
\end{equation}
\begin{comment}
\begin{definition}\label{def of regular points}
{\rm Let $G_0=G^{\pm}_{0}$ and $\widetilde{G}_0=\widetilde{G}^{\pm}_{0}$, where $G^{\pm}_{0}$ and $\widetilde{G}^{\pm}_{0}$ are the integral operators with kernels in \eqref{expan coeffie}, and denote $T_{0}=U+vG_{0}v,\widetilde{T}_{0}=U+v\widetilde{G}_{0}v$. Let $H=\Delta^2+V$, $|V(n)|\lesssim \left<n\right>^{-\beta}$ with some $\beta>0$.
\begin{itemize}
\item [(i)] We say that $0$ is a regular point of $H$ if $S_0T_0S_0$ is invertible in $S_0\ell^2(\Z)$.
\item [(ii)] We say that $16$ is a regular point of $H$ if $\widetilde{Q}\widetilde{T}_{0}\widetilde{Q}$ is invertible in $\widetilde{Q}\ell^2(\Z)$.
\end{itemize}
}
\end{definition}
\end{comment}
%Now we give the main results of this section.
\begin{theorem}\label{asy theor of Mpm mu}
{ Let $H=\Delta^2+V$ with $|V(n)|\lesssim \left<n\right>^{-\beta}$ for some $\beta>0$. %and assume that $H$ has no eigenvalues in $\mcaI$,
Then there exists $\mu_0>0$ small enough, such that $\left(M^{\pm}\left(\mu\right)\right)^{-1}$ satisfy the following asymptotic expansions on $\ell^2(\Z)$ for any $0<\mu<\mu_0$:
\begin{itemize}
\item [(i)] If 0 is a regular point of $H$ and $\beta>15$, then
\begin{align}\label{asy expan on 0}
\left(M^{\pm}\left(\mu\right)\right)^{-1}=S_0A_{01}S_0+\mu QA^{\pm}_{11}Q+\mu^2\left(QA^{\pm}_{21}Q+S_0A^{\pm}_{22}+A^{\pm}_{23}S_0\right)+\mu^3A^{\pm}_{31}+\Gamma_{4}(\mu).
%\left(M^{\pm}\left(\mu^4\right)\right)^{-1}=S_0A_{01}S_0+\mu QA^{\pm}_{11}Q+\mu^2\left(QA^{\pm}_{21}Q+S_0A^{\pm}_{22}+A^{\pm}_{23}S_0\right)+\mu^3A^{\pm}_{31}+\Gamma_{4}(\mu).
\end{align}
\vskip0.2cm
\item [(ii)] If 16 is a regular point of $H$ and $\beta>7$, then
\begin{align}\label{asy expan on 2}
\left(M^{\pm}\left(2-\mu\right)\right)^{-1}=\widetilde{Q}B_{01}\widetilde{Q}+\mu^{\frac{1}{2}}B^{\pm}_{11} +\Gamma_{1}(\mu),
\end{align}
\end{itemize}
where
  $A_{01},A^{\pm}_{kj},B_{01},B^{\pm}_{11}$ are $\mu$-independent bounded operators on $\ell^2(\Z)$ defined in \eqref{asy expan on 0} and \eqref{asy expan on 2}.
}
\end{theorem}
%\begin{remark}
%{\rm Some remarks are given as follows:
%\begin{itemize}
%\item[(i)] 
The proof of Theorem \ref{asy theor of Mpm mu} will be presented in Section \ref{proof of asy}. We point out that under the assumptions in Theorem \ref{asy theor of Mpm mu}, the regular conditions given in Definition \ref{defin of regular point 1} can be characterized by the invertibility of  the operators $S$ and  $\widetilde{S}$ defined in \eqref{S,Stuta}, i.e., 
    \begin{align}\label{charc in remark}
    \begin{split}
0\ {\rm is\  a\  regular\  point\  of}\  H\ &\Leftrightarrow\  S=\{0\},\\
 16\ {\rm is\  a\  regular\  point\  of} \ H\ &\Leftrightarrow\ \widetilde{S}=\{0\}.
 \end{split}
 \end{align}
%where $S$ and $\widetilde{S}$ are . 

%\vskip0.1cm
%\item[(ii)] When 0 is a regular point of $H$, formally, the asymptotic expansions of $\left(M^{\pm}\left(\mu\right)\right)^{-1}$ are the same as in the continuous case \cite{SWY22}. However, in the discrete scenario, we need expand more to meet the requirements in the Van der Corput Lemma.
%\end{itemize}

%{\rm }
%}
%\end{remark}

\section{Decay estimates for the free case and sharpness}\label{Sec of decay for free}

When $V\equiv0$, the decay estimate \eqref{cos-sin decay-estimate} follows directly from the estimate $|t|^{-\frac{1}{3}}$ of $e^{it\Delta}$. Hence in this section, we will establish the decay estimate \eqref{eitH decay-estimate} for the free group $e^{-it\Delta^2}$.
\begin{theorem}\label{D-E for free case}
{ For $t\neq0$, one has the following decay estimate:
\begin{equation}\label{eitH0 decay estimate}
\left\|e^{-it\Delta^2}\right\|_{\ell^1(\Z )\rightarrow\ell^{\infty}(\Z )}\lesssim|t|^{-\frac{1}{4}}.
\end{equation}
%\begin{equation}\label{cos-sin H0 decay-estimate}
%\|{\rm cos}(\Delta t)\|_{\ell^1\rightarrow\ell^{\infty}}+\left\|\frac{{\rm sin}(\Delta t)}{\Delta t}\right\|_{\ell^1\rightarrow\ell^{\infty}}\lesssim|t|^{-\frac{1}{3}}.
%\end{equation}
}
\end{theorem}
\begin{remark}
{\rm It is well-known that decay estimate for $e^{it\Delta}$ is derived using the Fourier
 transform, whose kernel is given by:
  \begin{equation*}
\left(e^{it\Delta}\right)(n,m)=(2\pi)^{-\frac{1}{2}}\int_{-\pi}^{\pi}e^{-it(2-2{\rm cos}x)-i(n-m)x}dx.
 \end{equation*}
Hence,
$$\left\|e^{it\Delta}\right\|_{\ell^1(\Z )\rightarrow\ell^{\infty}(\Z )}\lesssim \sup\limits_{s\in\R}\left|\int_{-\pi}^{\pi}e^{-it\left(2-2{\rm cos}x-sx\right)}dx\right|\lesssim |t|^{-\frac{1}{3}}.$$}
\end{remark}

Note  that  Fourier method may establish the decay estimate for $e^{-it\Delta^2}$. 
However, we would like to use Stone's formula to derive the free estimate \eqref{eitH0 decay estimate} in Theorem \ref{D-E for free case}
since it offers key insights for studying the perturbation case. 
To this end, we first establish the following lemma.
\begin{lemma}\label{Von der}
{ For $t\neq0$, the following estimate holds:}
\begin{equation}\label{von der esti}
\sup\limits_{s\in\R}\left|\int_{-\pi}^{0}e^{-it\left[(2\pm2{\rm cos}x)^2-sx\right]}dx\right|\lesssim |t|^{-\frac{1}{4}}.
\end{equation}
\begin{comment}
\begin{itemize}
\item[(1)]
%with $\alpha_1=\frac{1}{3}$ and $\alpha_2=\frac{1}{4}$.
\item[(2)] For any interval $[a,b]\subseteq[-\pi,0]$, the estimates \eqref{von der esti} still hold on $[a,b]$.
\item [(3)] Let $\phi(x)$ be a complex-valued function on $[a,b]\subseteq[-\pi,0]$, differentiable on $(a,b)$ with $\phi'(x)\in L^{1}([a,b])$. Then
\begin{equation}\label{von der esti complex}
\sup\limits_{s\in\R}\left|\int_{a}^{b}e^{-it\left[(2\pm2{\rm cos}x)^2-sx\right]}\phi(x)dx\right|\lesssim |t|^{-\frac{1}{4}}\left(|\phi(b)|+\int_{a}^{b}|\phi'(x)|dx\right).
\end{equation}
%with the same $\alpha_k$ in (1).
\end{itemize}
\end{comment}
\end{lemma}
\begin{remark}
{\rm %Some remarks on Lemma \ref{Von der} are given as follows:
%\begin{itemize}
%\item [(i)]
Note that the estimate above also holds on the interval $[0,\pi]$ by the variable substitution $x\rightarrow-x$. %Consequently, based on the kernel with the form \eqref{another proof method for eitH0 decay estimate}, the decay estimate \eqref{D-E for free case} is immediately obtained.
%\item [(ii)] For $-\Delta$, it was shown in \cite{SK05} that
%\end{itemize}
    }%Furthermore, we remark that all the estimates for kernels \eqref{kernels of Ki} can ultimately be reduced to the form presented in this lemma.}
\end{remark}
\begin{proof}
This estimate is a concrete application of the Van der Corput Lemma (see e.g. \cite[P. ${332-334}$]{Ste93}). For any $s\in\R$, observe that by substituting $x=-\pi-y$, we obtain
$$\int_{-\pi}^{0}e^{-it[(2+2{\rm cos}x)^2-sx]}dx=e^{-its\pi}\int_{-\pi}^{0}e^{-it[(2-2{\rm cos}y)^2+sy]}dy.$$
Hence it suffices to prove that
\begin{equation}\label{Oci int 1}
\sup\limits_{s\in\R}\left|\int_{-\pi}^{0}e^{-it\Phi_s(x)dx}\right|\lesssim |t|^{-\frac{1}{4}},
\end{equation}
where $$\Phi_{s}(x):=(2-2{\rm cos}x)^2-sx,\quad x\in[-\pi,0].$$
%then it suffices to prove that (1) holds for the case with the integrand being $e^{-it\Phi^{-}_{k,s}}$.
%In view that the estimate \eqref{von der esti} for $k=1$ has been proved in \cite[Section 4.1]{SK05}, so in what follows we focus on the case of $k=2$, and for simplicity, we drop the subscripts $``-"$ and $``~2~"$, simply denote by $\Phi_{s}$.

First, a direct calculation yields that
$$\Phi'_{s}(x)=8(1-{\rm cos}x){\rm sin}x-s,\quad \Phi''_{s}(x)=8(1-{\rm cos}x)(2{\rm cos}x+1).$$
Note that $\Phi'_s(0)=\Phi'_s(-\pi)=-s$, and
$$\Phi''_{s}(x)=0,\ x\in[-\pi,0]\Longleftrightarrow x=-\frac{2\pi}{3}\ {\rm or}\ x=0,$$
it follows that $\Phi'_{s}(x)$ is monotonically increasing on $[-\frac{2\pi}{3},0]$ and decreasing on $[-\pi,-\frac{2\pi}{3})$. Consequently, for any $s\in\R$, the equation $\Phi'_{s}(x)=0$ has at most two solutions on $[-\pi,0]$. By the Van der Corput Lemma, slower decay rates of the oscillatory integral \eqref{Oci int 1} occur in the cases of $s=0$ and $s=-6\sqrt 3$. For the other values of $s$, the rate is either $|t|^{-1}$ or $|t|^{-\frac{1}{2}}$.

If $s=0$, then $$\Phi'_{0}(x)=0,\ x\in[-\pi,0]\Longleftrightarrow x=0\ {\rm or}\ x=-\pi.$$
Moreover, we can compute:
$$\Phi''_{0}(-\pi)=-16\neq0\ {\rm and}\ \Phi''_{0}(0)=\Phi^{(3)}_{0}(0)=0, \Phi^{(4)}_{0}(0)=24\neq0,$$
Thus, by the Van der Corput Lemma, \eqref{Oci int 1} is controlled by $|t|^{-\frac{1}{4}}$.

If $s=-6\sqrt 3$, then
$$\Phi'_{-6\sqrt 3}(x)=0,\ x\in[-\pi,0]\Longleftrightarrow x=-\frac{2\pi}{3},$$
and
$$\Phi''_{-6\sqrt 3}\left(-\frac{2\pi}{3}\right)=0\ {\rm but}\ \Phi^{(3)}_{-6\sqrt 3}\left(-\frac{2\pi}{3}\right)\neq0,$$
Similarly, this implies that the decay rate is $|t|^{-\frac{1}{3}}$. In summary, we obtain the desired estimate of $|t|^{-\frac{1}{4}}$.
\end{proof}
From the above proof, we  can  immediately obtain the following corollary, which plays a key role in the estimates for the kernels $K^{\pm}_{j}(t,n,m)$ defined in \eqref{kernels of Ki}.
\begin{corollary}\label{corollary} { The following conclusions hold:

 (i)~For any interval $[a,b]\subseteq[-\pi,0]$, the estimate \eqref{von der esti} still holds on $[a,b]$.

 (ii)~Let $[a,b]\subseteq[-\pi,0]$. Suppose that $\phi(x)$ is a continuously differentiable function on $(a,b)$ with $\phi'(x)\in L^{1}((a,b))$. Moreover, $\lim\limits_{x\rightarrow a^+}\phi(x)$ and $\lim\limits_{x\rightarrow b^-}\phi(x)$ exist. Then
\begin{equation}\label{von der esti complex}
\sup\limits_{s\in\R}\left|\int_{a}^{b}e^{-it\left[(2\pm2{\rm cos}x)^2-sx\right]}\phi(x)dx\right|\lesssim |t|^{-\frac{1}{4}}\left(\big|\lim\limits_{x\rightarrow b^-}\phi(x)\big|+\int_{a}^{b}|\phi'(x)|dx\right).
\end{equation}
}
\end{corollary}
Now we return to the proof of Theorem \ref{D-E for free case}.
\begin{proof}[Proof of Theorem \ref{D-E for free case}]
%\underline{\textbf{(1)}~Proof of \eqref{eitH0 decay estimate}.}
First, using Stone's formula, the kernel of $e^{-it\Delta^2}$ is given by:
\begin{equation}\label{kernel of 4order}
\left(e^{-it\Delta^2}\right)(n,m)=\frac{2}{\pi i}\int_{0}^{2}e^{-it\mu^4}\mu^3(R^{+}_0-R^{-}_0)(\mu^4,n,m)d\mu=-\frac{2}{\pi i}(K^{+}_{0}-K^{-}_{0})(t,n,m),
\end{equation}
where
\begin{equation}\label{K0}
K^{\pm}_{0}(t,n,m)=\int_{0}^{2}e^{-it\mu^4}\mu^3R^{\mp}_0(\mu^4,n,m)d\mu.
\end{equation}
Thus, it suffices to show that
$$\left|K^{\pm}_{0}(t,n,m)\right|\lesssim|t|^{-\frac{1}{4}},\quad t\neq0,$$
uniformly in $n,m\in\Z$.

Next, we focus on the case of $K^{-}_0(t,n,m)$, and the case of $K^{+}_0(t,n,m)$ can be handled similarly. Taking formula \eqref{kernel of R0 boundary} into \eqref{K0}, we obtain
\begin{align}\label{decompose of K0+}
\begin{split}
K^{-}_0(t,n,m)&=\frac{-i}{4}\left(\int_{0}^{2}e^{-it\mu^4}\frac{\mu e^{-i\theta_+|n-m|}}{{\rm sin}\theta_+}d\mu-\int_{0}^{2}e^{-it\mu^4}\frac{\mu e^{b(\mu)|n-m|}}{{\rm sin}\theta}d\mu\right)\\
&:=\frac{-i}{4}(K^{-}_{0,1}-K_{0,2})(t,n,m).
\end{split}
\end{align}

 On one hand, we have
\begin{equation}\label{decay esti of K01+}
\sup\limits_{n,m\in\Z}|K^{-}_{0,1}(t,n,m)|\lesssim|t|^{-\frac{1}{4}},\quad t\neq0.
\end{equation}
To see this, we consider the following variable substitution:
 \begin{equation}\label{varible substi1}
 {\rm cos}\theta_{+}=1-\frac{\mu^2}{2} \Longrightarrow\  \frac{d\mu}{d\theta_+}=\frac{{\rm sin}\theta_+}{\mu},
\end{equation}
where $\theta_+\rightarrow0$ as $\mu\rightarrow0$ and $\theta_+\rightarrow-\pi$ as $\mu\rightarrow2$.
 Then $K^{-}_{0,1}(t,n,m)$ can be rewritten as:
 \begin{equation}\label{K01+}
 K^{-}_{0,1}(t,n,m)=-\int_{-\pi}^{0}e^{-it\left[(2-2{\rm cos}\theta_+)^2-\theta_{+}\left(-\frac{|n-m|}{t}\right)\right]}d\theta_+,\quad t\neq0.
\end{equation}
Thus, \eqref{decay esti of K01+} follows from Lemma \ref{Von der} since that
$$|\eqref{K01+}|\leq\left|\sup\limits_{s\in\R}\int_{-\pi}^{0}e^{-it\left[(2-2{\rm cos}\theta_+)^2-s\theta_{+}\right]}d\theta_+\right|\lesssim |t|^{-\frac{1}{4}},\quad t\neq0.$$
%uniformly in $n,m\in\Z$.

 On the other hand, we have
\begin{equation}\label{decay estim of K02}
\sup\limits_{n,m\in\Z}|K_{0,2}(t,n,m)|\lesssim|t|^{-\frac{1}{4}},\quad t\neq0.
\end{equation}
By \eqref{expre of sin theta}, it follows that ${\rm sin}\theta=-i\mu(1+\frac{\mu^2}{4})^{\frac{1}{2}}$, and thus
\begin{equation}
F_{0}(\mu,n,m):=\frac{\mu e^{b(\mu)|n-m|}}{{\rm sin}\theta}=f_0(\mu)e^{b(\mu)|n-m|},\quad \mu\in(0,2),\ %n,m\in\Z,
\end{equation}
where $f_0(\mu)=i\left(1+\frac{\mu^2}{4}\right)^{-\frac{1}{2}}$. Clearly, for any $n,m\in\Z$, $F_0(\mu,n,m)$ is continuously differentiable on $(0,2)$. Moreover, since $b(\mu)={\rm ln} \big(1+\frac{\mu^2}{2}-\mu(1+\frac{\mu^2}{4})^{\frac{1}{2}}\big)$, we have
$$\lim\limits_{\mu\rightarrow0^{+}}F_{0}(\mu,n,m)=i,\quad \lim\limits_{\mu\rightarrow2^{-}}F_{0}(\mu,n,m)=\frac{\sqrt2i}{2}(3-2\sqrt2)^{|n-m|}.$$
Additionally,
\begin{align}\label{differentiate on F0}
\frac{\partial F_0}{\partial\mu}(\mu,n,m)=f'_0(\mu)e^{b(\mu)|n-m|}+f_0(\mu)\frac{\partial}{\partial\mu}\left(e^{b(\mu)|n-m|}\right),\quad \mu\in(0,2).
\end{align}
We claim that
$$\sup\limits_{n,m\in\Z}\left\|\frac{\partial F_0}{\partial\mu}(\mu,n,m)\right\|_{L^{1}((0,2))}\lesssim 1.$$
Indeed, the first term in \eqref{differentiate on F0} is uniformly bounded because $f_0(\mu)$ is continuously differentiable on $(0,2)$ and $b(\mu)<0$. For the second term, notice that $b'(\mu)<0$ for any $\mu\in(0,2)$, then
\begin{equation}\label{integral ebnm}
\left\|\frac{\partial}{\partial\mu}\left(e^{b(\mu)|n-m|}\right)\right\|_{L^{1}((0,2))}\leq \int_{0}^{2}-b'(\mu)|n-m|e^{b(\mu)|n-m|}d\mu\leq2.
\end{equation}
By the Van der Corput Lemma, %the \eqref{decay estim of K02} follows from the following inequality
we conclude that
$$|K_{0,2}(t,n,m)|\lesssim|t|^{-\frac{1}{4}}\left(\big|\lim\limits_{\mu\rightarrow2^{-}}F_0(\mu,n,m)\big|+\sup\limits_{n,m\in\Z}\left\|\frac{\partial F_0}{\partial\mu}(\mu,n,m)\right\|_{L^{1}((0,2))}\right)\lesssim|t|^{-\frac{1}{4}},\quad t\neq0.$$
%where for the first inequality, we notice that $b(\mu)<0$ and the second can be concluded from
Therefore, combining \eqref{decay esti of K01+}, \eqref{decay estim of K02} and \eqref{kernel of 4order}, we obtain the desired estimate \eqref{eitH0 decay estimate}.
\end{proof}
\begin{comment}
\begin{remark}\label{remark of Theorem for free case}
{\rm From the proof process above, we can observe that the only difference between $e^{-it\Delta^2}$ and $e^{-it\Delta}$ lies in the power of $\mu$ in the exponential part, and this difference affects nothing else but the decay rates, ranging from $\frac{1}{4}$ to $\frac{1}{3}$. Therefore, based on this point, in Section \ref{Sec of proof}, it is sufficient to deal with the decay estimates for $e^{-itH}P_{ac}(H)$.}
\end{remark}
\vskip0.3cm
\end{comment}
Finally, we demonstrate that the decay rate $\frac{1}{4}$ in \eqref{eitH0 decay estimate} is sharp, that is, $\frac{1}{4}$ is the supremum of all $\alpha$ for which there exists a constant $C_{\alpha}$ such that
\begin{equation}\label{sharness meaning}
\left\|e^{-it\Delta^2}\phi\right\|_{\ell^{\infty}}\leq C_{\alpha}|t|^{-\alpha}\|\phi\|_{\ell^{1}},\quad t\neq0,
\end{equation}
holds for every sequence $\{\phi(n)\}_{n\in\Z}\in\ell^1(\Z)$.
\begin{theorem}\label{theorem of strichartz estimate}
{ Consider the free discrete bi-Schr\"{o}dinger inhomogeneous equation on the lattice $\Z$:
$$i(\partial_tu)(t,n)-(\Delta^2u)(t,n)+F(t,n)=0,$$
with the initial value $\{u(0,n)\}_{n\in\Z}\in\ell^2(\Z)$. Then
\begin{itemize}
 \item [(i)] The following Strichartz estimates hold
\begin{equation}\label{strichartz estimate}
\|\{u(t,n)\}\|_{L^{q}_t\ell^{r}}\leq C\big(\|\{u(0,n)\}\|_{\ell^2}+\|\{F(t,n)\}\|_{L^{\tilde{q}^{'}}_t\ell^{\tilde{r}^{'}}}\big),
\end{equation}
\end{itemize}
where $(\tilde{q},\tilde{r}),(q,r)\in\big\{(x,y)\neq(2,\infty):\frac{1}{x}+\frac{1}{4y}\leq\frac{1}{8},\ x,y\geq2\big\}$, $\tilde{q}^{'},\tilde{r}^{'}$ denote the dual exponents of $\tilde{q}$ and $\tilde{r}$, respectively and
$$\|\{u(t,n)\}\|_{L^{q}_t\ell^{r}}=\left(\int_{0}^{\infty}\left(\sum\limits_{n\in\Z}|u(t,n)|^{r}\right)^{\frac{q}{r}}dt\right)^{\frac{1}{q}};$$
 \begin{itemize}
 \item [(ii)] The decay estimate \eqref{eitH0 decay estimate} is sharp.
 \end{itemize}}
\end{theorem}
\begin{proof}
\textbf{(i)}~From the energy identity $\|\{u(t,n)\}\|_{\ell^2}=\|\{u(0,n)\}\|_{\ell^2}$ and the decay estimate \eqref{eitH0 decay estimate}, the Strichartz estimates \eqref{strichartz estimate} follow directly from \cite[Theorem 1.2]{KT98} by Keel and Tao.

\textbf{(ii)} To prove the sharpness of the decay estimate, we first establish the sharpness of the Strichartz estimates by constructing a Knapp counter-example. By duality, we have
\begin{equation}\label{equivalent relation}
\|\{e^{-it\Delta^2}u(0,n)\}\|_{L^{q}_{t}\ell^{r}}\leq C\|u(0,\cdot)\|_{\ell^2}\Leftrightarrow\left\|\phi(\cdot)\right\|_{\ell^2}\leq C\|\{F(t,n)\}\|_{L^{q'}_{t}\ell^{r'}}.
\end{equation}
where
$$\phi(n)=\int_{\R}^{}e^{-it\Delta^2}F(t,n)dt.$$
Based on the Fourier transform defined in  \eqref{fourier transform}, one obtains that
$$\mcaF\phi(x)=(2\pi)^{\frac{1}{2}}\hat{f}_{time}(-(2-2{\rm cos}x)^2,x),$$
where $\hat{f}_{time}(s,x)$ is the time Fourier transform of $f$, defined by $$\hat{f}_{time}(s,x)=(2\pi)^{-\frac{1}{2}}\int_{\R}^{}f(t,x)e^{-ist}dt,$$
and $f(t,x)=\mcaF F(t,n)=(2\pi)^{-\frac{1}{2}}\sum\limits_{n\in\Z}^{}e^{-inx}F(t,n)$. Therefore, the right side inequality of \eqref{equivalent relation} can be further expressed as follows:
\begin{equation}\label{equivent}
\left(\int_{-\pi}^{\pi}\big|\hat{f}_{time}\left(-(2-2{\rm cos}x)^2,x\right)\big|dx\right)^{\frac{1}{2}}\leq C'\|\{F(t,n)\}\|_{L^{q'}_{t}\ell^{r'}}.
\end{equation}
%where $\hat{f}_{time}(\tau,x)=\int_{0}^{\infty}f(t,x)e^{-i\tau t}dt$ is the time Fourier transform of $f$ and $f(t,x)=\mcaF F(t,n)$ and $\mcaF$ is the discrete Fourier transform defined in \eqref{fourier transform}.

For any $0<\varepsilon\ll1$, we choose
$$\hat{f}_{time}(s,x)=\chi(\varepsilon^{-4}s)\chi(\varepsilon^{-1}x),$$
where $\chi$ is the characteristic function of the interval $(-1,1)$. This yields that
$$F(t,n)=c\frac{{\rm sin}\left(\varepsilon^{4}t\right)}{t}\frac{{\rm sin}(\varepsilon n)}{n}.$$
%Then noticing that the
On one hand, using Taylor's expansion $(2-2{\rm cos}x)^2=O(x^4),\  x\rightarrow0,$ we find that
$$\left(\int_{-\pi}^{\pi}\big|\hat{f}_{time}\left(-(2-2{\rm cos}x)^2,x\right)\big|dx\right)^{\frac{1}{2}}\gtrsim\varepsilon^{\frac{1}{2}}.$$
On the other hand, observe that
\begin{align*}
\sum\limits_{n\in\Z}^{}\frac{|{\rm sin}(\varepsilon n)|^{r'}}{|n|^{r'}}&=\big(\sum\limits_{|n|\leq\frac{1}{\varepsilon}}^{}+\sum\limits_{|n|>\frac{1}{\varepsilon}}^{}\big)\frac{|{\rm sin}(\varepsilon(n))|^{r'}}{|n|^{r'}}\\
&\leq C''\varepsilon^{r'}\frac{1}{\varepsilon}+\sum\limits_{|n|>\frac{1}{\varepsilon}}^{}\frac{1}{|n|^{r'}}\lesssim \varepsilon^{r'-1},
\end{align*}
then it follows that $$\|\{F(t,n)\}\|_{L^{q'}_{t}\ell^{r'}}\lesssim\varepsilon^{\frac{1}{r}+\frac{4}{q}}.$$
Since $\varepsilon$ is arbitrary small, then $\frac{1}{2}\geq\frac{1}{r}+\frac{4}{q}$.

Then the decay rate in \eqref{eitH0 decay estimate} is also sharp. Indeed, if not, i.e., there exists an estimate of the form \eqref{eitH0 decay estimate} with $\alpha>\frac{1}{4}$. By \cite[Theorem 1.2]{KT98}, then this would imply Strichartz estimates in the range $\frac{1}{q}+\frac{\alpha}{r}\leq\frac{\alpha}{2}$. Since $\alpha>\frac{1}{4}$, then there exists $q,r\geq2$ satisfying
\begin{equation*}
\frac{1}{q}+\frac{\alpha}{r}\leq\frac{\alpha}{2}\ \ {\rm and}\ \ \frac{1}{q}+\frac{1}{4r}>\frac{1}{8}.
\end{equation*}
This contradicts the sharpness of the Strichartz estimates established above.
\end{proof}
\section{Proof of Theorem \ref{main-theorem}}\label{Sec of proof}
 This section is devoted to presenting a detailed proof of \eqref{eitH decay-estimate} for $e^{-itH}P_{ac}(H)$, from which \eqref{cos-sin decay-estimate} follows similarly.
To begin with, we recall the decomposition:
\begin{equation}\label{kernel of eitHPacH(4 section)1}
\left(e^{-itH}P_{ac}(H)\right)(n,m)=-\frac{2}{\pi i}\sum\limits_{j=0}^{3}(K^{+}_{j}-K^{-}_{j})(t,n,m),
\end{equation}
where $K^{\pm}_{j}(t,n,m)(j=0,1,2,3)$ are defined in \eqref{kernels of Ki}.
As demonstrated in Section \ref{Sec of decay for free}, the estimate for $K^{\pm}_0(t,n,m)$ has already been established. In what follows, we will focus on deriving the corresponding estimates for the kernels $K^{\pm}_{j}(t,n,m)$ with $j=1,2,3$.
\begin{theorem}\label{theorem of estimate for K123}
{ Under the assumptions in Theorem \ref{main-theorem}, let $K^{\pm}_{j}(t,n,m)(j=1,2,3)$ be defined as in  \eqref{kernels of Ki}. Then the following estimates hold:
\begin{equation*}
\sup\limits_{n,m\in\Z}\left|\left(K^{+}_{1}-K^{-}_{1}\right)(t,n,m)\right|\lesssim |t|^{-\frac{1}{4}},\quad t\neq0,
\end{equation*}
and
\begin{equation*}
\sup\limits_{n,m\in\Z}\left|K^{\pm}_{2}(t,n,m)\right|+\sup\limits_{n,m\in\Z}\left|K^{\pm}_{3}(t,n,m)\right|\lesssim |t|^{-\frac{1}{4}}, \quad t\neq0.
\end{equation*}
}
\end{theorem}
By combining Theorems \ref{D-E for free case} and \ref{theorem of estimate for K123}, we derive the \eqref{eitH decay-estimate}, thus completing the proof of Theorem \ref{main-theorem}.
\vskip0.2cm
To prove Theorem \ref{theorem of estimate for K123}, we will analyse the cases presented in Propositions \ref{Proposition of K1},\ \ref{Proposition of K2} and \ref{Proposition of K3}. Each case will be addressed in detail in the following three subsections, respectively.
\subsection{The estimates of kernels $(K_1^{+}-K^{-}_1)(t,n,m)$}\label{subsec of Case I}
\begin{proposition}\label{Proposition of K1}
{ Let $H=\Delta^2+V$ with $|V(n)|\lesssim \left<n\right>^{-\beta}$ for $\beta>15$. Suppose that %$H$ has no positive eigenvalue in $\rm{(}0,16\rm{)}$ and
$0$ is a regular point of $H$. Then
\begin{equation}\label{K1p-K1m}
\left|\left(K^{+}_{1}-K^{-}_{1}\right)(t,n,m)\right|\lesssim |t|^{-\frac{1}{4}},\quad t\neq0,
\end{equation}
uniformly in $ n,m\in\Z.$
}
\end{proposition}
In this subsection, we always assume that $0$ is a regular point of $H$ and $\beta>15$.
Before proof, we first make some preparations. Recall that the kernel of $K^{\pm}_{1}$ is given by:
\begin{equation}\label{kernel of K1}
K^{\pm}_{1}(t,n,m)=\int_{0}^{\mu_0}e^{-it\mu^4}\mu^3\left[R^{\pm}_0\left(\mu^4\right)v\left(M^{\pm}\left(\mu\right)\right)^{-1}vR^{\pm}_0\left(\mu^4\right)\right](n,m)d\mu.
\end{equation}
Using the expansion of $\left(M^{\pm}\left(\mu\right)\right)^{-1}$ from \eqref{asy expan on 0}, namely, $$\left(M^{\pm}\left(\mu\right)\right)^{-1}=S_0A_{01}S_0+\mu QA^{\pm}_{11}Q+\mu^2\left(QA^{\pm}_{21}Q+S_0A^{\pm}_{22}+A^{\pm}_{23}S_0\right)+\mu^3A^{\pm}_{31}+\Gamma_{4}(\mu),$$
and substituting it into \eqref{kernel of K1}, we obtain
\begin{align}
\left(K^{+}_{1}-K^{-}_{1}\right)(t,n,m)=\left(K_{11}+K^{+}_{12}-K^{-}_{12}\right)(t,n,m),
\end{align}
where
\begin{align}
K_{11}(t,n,m)&=\int_{0}^{\mu_0}e^{-it\mu^4}\mu^3\left(R^{+}_0\left(\mu^4\right)vS_0A_{01}S_0vR^{+}_0\left(\mu^4\right)-R^{-}_0\left(\mu^4\right)vS_0A_{01}S_0vR^{-}_0\left(\mu^4\right)\right)(n,m)d\mu,\label{kernel of K11}\\
K^{\pm}_{12}(t,n,m)&=\int_{0}^{\mu_0}e^{-it\mu^4}\mu^3\left[R^{\pm}_0\left(\mu^4\right)v\left(\left(M^{\pm}\left(\mu\right)\right)^{-1}-S_0A_{01}S_0\right)vR^{\pm}_0\left(\mu^4\right)\right](n,m)d\mu.\label{kernel of K12}
\end{align}
Hence, the estimate for $K^{+}_1-K^{-}_{1}$ further reduces to that of $K_{11}$ and $K^{\pm}_{12}$, respectively. Now  we establish the following crucial lemma.
\begin{lemma}\label{cancelation of caseI}
{ Let $Q,S_0$ be defined as in \eqref{definition of Q,S0,Qtuta}. Then, for any $f\in\ell^2(\Z)$, the following statements hold:
\begin{itemize}
\item [(1)] $\left[(R^{+}_0-R^{-}_0)\left(\mu^4\right)vS_0f\right](n)=\frac{i\theta^2_{+}}{4\mu^2{\rm sin}\theta_{+}}\sum\limits_{m\in\Z}^{}\int_{0}^{1}(1-\rho)F(-\theta_{+}|n-\rho m|)d\rho\cdot m^2(vS_{0}f)(m),$
    \vskip0.2cm
\item [(2)]
$
S_0\big(v\left(R^{+}_0-R^{-}_0\right)\left(\mu^4\right)f\big)=S_0\left(\frac{i\theta^2_{+}n^2v(n)}{4\mu^2{\rm sin}\theta_{+}}\sum\limits_{m\in\Z}^{}\int_{0}^{1}(1-\rho)F(-\theta_{+}|m-\rho n|)d\rho\cdot f(m)\right),
$
 \vskip0.2cm
\item [(3)]
$~\left(R^{\pm}_0\left(\mu^4\right)vWf\right)(n)=\sum\limits_{m\in\Z}^{}\int_{0}^{1}{\rm sign}(n-\rho m)\left(\frac{\theta_{\pm}e^{-i\theta_{\pm}|n-\rho m|}}{{4\mu^2\rm sin}\theta_{\pm}}-\frac{g(\mu)e^{b(\mu)|n-\rho m|}}{4\mu^2}\right)d\rho \cdot m(vWf)(m),$
\vskip0.2cm
\item [(4)]
$W\big(vR^{\pm}_0\left(\mu^4\right)f\big)=Wf^{\pm}$,
\end{itemize}
where $W=Q,S_0$ and
\begin{align}\label{gmu}
\begin{split}
F(s)&=e^{is}+e^{-is},\quad g(\mu)=-\frac{b(\mu)}{\mu\big(1+\frac{\mu^2}{4}\big)^{\frac{1}{2}}}, \\
f^{\pm}(n)&=nv(n)\sum\limits_{m\in\Z}^{}\int_{0}^{1}{\rm sign}(m-\rho n)\left(\frac{\theta_{\pm}e^{-i\theta_{\pm}|m-\rho n|}}{4\mu^2{\rm sin}\theta_{\pm}}-\frac{g(\mu)e^{b(\mu)|m-\rho n|}}{4\mu^2}\right)d\rho\cdot f(m).
\end{split}
\end{align}
}
\end{lemma}
\vskip0.2cm
\begin{remark}\label{remark on cancel lemma}
{\rm Roughly speaking, compared to the free kernel $R^{\pm}_0(\mu^4,n,m)$:
$$R^{\pm}_{0}(\mu^4,n,m)=\frac{-i}{4\mu^2}\left(\frac{e^{-i\theta_\pm|n-m|}}{{\rm sin}\theta_\pm}-\frac{e^{b(\mu)|n-m|}}{{\rm sin}\theta}\right).$$
%with regard to the singularity near $\mu=0$ of
The kernels considered in this lemma have less singularity near $\mu=0$. More precisely, the kernels in (1) and (2) contribute a factor of $\mu^2$, while those in (3) and (4) contribute a factor of $\mu$. In fact, noting that %by the definition of $\theta_+$, it satisfies that
$\mu^2=2(1-{\rm cos}\theta_+)$, we see that $\theta_+$ and $\mu$ are infinitesimals of the same order as $\mu\rightarrow0$, i.e., $\theta_+=O(\mu)$, which plays a key role in the subsequent decay estimates. %Consequently, .
}
\end{remark}
\begin{proof}[Proof of Lemma \ref{cancelation of caseI}]
\underline{\textbf{(1)~and~(2)}}
From the first equality of \eqref{kernel of R0 boundary} and the fact that $\theta_-=-\theta_+$, we have
\begin{equation}
\left(R^{+}_0-R^{-}_0\right)\left(\mu^4,n,m\right)=\frac{-i}{4\mu^2{\rm sin}\theta_{+}}F(-\theta_{+}|n-m|),
\end{equation}
where $F(s)=e^{is}+e^{-is}$.
Then
\begin{align}\label{kernel (1)}
\left[(R^{+}_0-R^{-}_0)\left(\mu^4\right)vS_0f\right](n)=\frac{-i}{4\mu^2{\rm sin}\theta_{+}}\sum\limits_{m\in\Z}^{}F\left(-\theta_{+}|n-m|\right)v(m)\left(S_0f\right)(m),
\end{align}
and
\begin{align}\label{kernel (2)}
S_0\left(v\left(R^{+}_0-R^{-}_0\right)\left(\mu^4\right)f\right)=S_0\left(v(n)\frac{-i}{4\mu^2{\rm sin}\theta_{+}}\sum\limits_{m\in\Z}^{}F\left(-\theta_{+}|n-m|\right)f(m)\right).
\end{align}
Notice that $F'(0)=0$, then by \cite[Lemma 2.5]{SWY22} and $F''(s)=-F(s)$, it follows that
\begin{align}
F(-\theta_+|n-m|)&=F(-\theta_+|n|)+\theta_+m{\rm sign}(n)F'(-\theta_+|n|)\nonumber\\
&\quad-\theta^2_+m^2\int_{0}^{1}(1-\rho)F(-\theta_+|n-\rho m|)d\rho,\label{ii Taylor 1}\\
F(-\theta_+|n-m|)&=F(-\theta_+|m|)+\theta_+n{\rm sign}(m)F'(-\theta_+|m|)\nonumber\\
&\quad-\theta^2_+n^2\int_{0}^{1}(1-\rho)F(-\theta_+|m-\rho n|)d\rho.\label{ii Taylor 2}
\end{align}
Taking \eqref{ii Taylor 1} into \eqref{kernel (1)} and \eqref{ii Taylor 2} into \eqref{kernel (2)}, and using the cancellation properties \eqref{cancel of s0}, we obtain the desired results (1) and (2).
\vskip0.3cm
\underline{\textbf{(3)~and~(4)}} As before, by \eqref{kernel of R0 boundary}, if we denote $F^{\pm}_1(s):=e^{\pm is}$ and $F_2(s):=e^{-s}$, then
$$R^{\pm}_{0}(\mu^4,n,m)=\frac{-i}{4\mu^2}\left(\frac{F^{\pm}_1\left(\mp\theta_{\pm}|n-m|\right)}{{\rm sin}\theta_\pm}-\frac{F_2(-b(\mu)|n-m|)}{{\rm sin}\theta}\right).$$
Applying \cite[Lemma 2.5]{SWY22} to $F^{\pm}_{1}$ and $F_2$, and observing that $(F^{\pm}_{1})'(s)=\pm iF^{\pm}_{1}(s)$, $F'_2(s)=-F_2(s)$, we have
\begin{align*}
\begin{split}
F^{\pm}_1\left(\mp\theta_{\pm}|n-m|\right)&=F^{\pm}_1\left(\mp\theta_{\pm}|n|\right)+i\theta_{\pm}m\int_{0}^{1}{\rm sign}(n-\rho m)F^{\pm}_{1}\left(\mp\theta_{\pm}|n-\rho m|\right)d\rho\\
&=F^{\pm}_1\left(\mp\theta_{\pm}|m|\right)+i\theta_{\pm}n\int_{0}^{1}{\rm sign}(m-\rho n)F^{\pm}_{1}\left(\mp\theta_{\pm}|m-\rho n|\right)d\rho,\\
F_2(-b(\mu)|n-m|)&=F_2(-b(\mu)|n|)-b(\mu)m\int_{0}^{1}{\rm sign}(n-\rho m)F_2(-b(\mu)|n-\rho m|)d\rho\\
&=F_2(-b(\mu)|m|)-b(\mu)n\int_{0}^{1}{\rm sign}(m-\rho n)F_2(-b(\mu)|m-\rho n|)d\rho.
\end{split}
\end{align*}
Following the same approach used in the proofs of \underline{\textbf{(1)~and~(2)}}, and utilizing the cancelation condition $Wv=0$, $\left<Wf,v\right>=0$ for $W=Q,S_0$, we then prove (3) and (4).
\end{proof}
Next we begin the proof of Proposition \ref{Proposition of K1}. First, we address the term $K_{11}$.
\begin{proposition}\label{claim of K11}
{ Under the assumptions in Propositions \ref{Proposition of K1}, let $K_{11}(t,n,m)$ be defined as in \eqref{kernel of K11}. Then
\begin{equation}\label{estimate of K11 in Claim}
\sup\limits_{n,m\in\Z}\left|K_{11}(t,n,m)\right|\lesssim |t|^{-\frac{1}{4}},\quad t\neq0.
\end{equation}
 }
\end{proposition}
\begin{proof}
To make best use of the cancelation properties of $S_0$ to eliminate the high singularity near $\mu=0$, we employ a trick by adding and subtracting a term. This allows us to rewrite:
\begin{align}
&R^{+}_0\left(\mu^4\right)vS_0A_{01}S_0vR^{+}_0\left(\mu^4\right)-R^{-}_0\left(\mu^4\right)vS_0A_{01}S_0vR^{-}_0\left(\mu^4\right)\nonumber\\
&=\left(R^{+}_0-R^{-}_0\right)\left(\mu^4\right)vS_0A_{01}S_0vR^{+}_0\left(\mu^4\right)+R^{-}_0\left(\mu^4\right)vS_0A_{01}S_0v\left(R^{+}_0-R^{-}_0\right)\left(\mu^4\right).\label{deformation}
\end{align}
Substituting \eqref{deformation} into \eqref{kernel of K11}, we reduce the estimate \eqref{estimate of K11 in Claim} to bounding the following two kernels:
$$K_{11}(t,n,m)=\left(\widetilde{K}_{11}+\widetilde{\widetilde{K}}_{11}\right)(t,n,m),$$
where
\begin{align*}
\widetilde{K}_{11}(t,n,m)&=\int_{0}^{\mu_0}e^{-it\mu^4}\mu^3\left(\left(R^{+}_0-R^{-}_0\right)\left(\mu^4\right)vS_0A_{01}S_0vR^{+}_0\left(\mu^4\right)\right)(n,m)d\mu,\\
\widetilde{\widetilde{K}}_{11}(t,n,m)&=\int_{0}^{\mu_0}e^{-it\mu^4}\mu^3\left(R^{-}_0\left(\mu^4\right)vS_0A_{01}S_0v\left(R^{+}_0-R^{-}_0\right)\left(\mu^4\right)\right)(n,m)d\mu.
\end{align*}
By symmetry, it suffices to deal with the term $\widetilde{K}_{11}(t,n,m)$.

From Lemma \ref{cancelation of caseI}, we have
\begin{align}\label{kernel of Ktuta11}
\widetilde{K}_{11}(t,n,m)&=\sum\limits_{m_1,m_2\in\Z}^{}\int_{[0,1]^2}(1-\rho _1){\rm sign}( M_2)\left(\Omega^{+}_{11}+\Omega^{-}_{11}\right)(t,N_1,M_2)d\rho_1 d\rho_2\nonumber\\
&\quad\times m^2_1m_2\left(vS_0A_{01}S_0v\right)(m_1,m_2),
\end{align}
where $N_1=n-\rho_1m_1$, $ M_2=m-\rho_2m_2$, and
\begin{align}\label{instrument of Omega11}
\begin{split}
\Omega^{\pm}_{11}(t, N_1,M_2)&=\int_{0}^{\mu_0}e^{-it\mu^4}\mu^3\Lambda^{\pm}_{11}(\mu,N_1,M_2)d\mu,\\
\Lambda^{\pm}_{11}(\mu,N_1,M_2)&=\frac{i}{16\mu^4}\left(\frac{\theta^3_{+}e^{-i\theta_{+}(|M_2|\pm |N_1|)}}{{\rm sin}^2\theta_{+}}-\frac{\theta^2_+e^{\mp i\theta_{+}| N_1|}}{{\rm sin}\theta_+}g(\mu)e^{b(\mu)|M_2|}\right),
\end{split}
\end{align}
with $g(\mu)$ defined in \eqref{gmu}.
In what follows, we will show that
\begin{equation}\label{Omega pm11}
|\Omega^{\pm}_{11}(t, N_1,M_2)|\lesssim |t|^{-\frac{1}{4}},\quad t\neq0,
\end{equation}
uniformly in $N_1,M_2$.
Once \eqref{Omega pm11} is established, then by the condition $\beta>15$, we have
$$|\widetilde{K}_{11}(t,n,m)|\lesssim|t|^{-\frac{1}{4}}\|\left<\cdot\right>^{2}v(\cdot)\|^2_{\ell^2}\|S_0A_{01}S_0\|_{\B (0,0)}\lesssim|t|^{-\frac{1}{4}},\quad t\neq0,$$
uniformly in $n,m\in\Z$, which proves \eqref{estimate of K11 in Claim}.

To establish \eqref{Omega pm11}, we focus on $\Omega^{+}_{11}$ for brevity, and the analysis for $\Omega^{-}_{11}$ is similar.
From \eqref{instrument of Omega11}, we have
\begin{equation}\label{omegap11 Omegap111 Omegap112}
\Omega^{+}_{11}(t, N_1,M_2)=\frac{i}{16}\left(\Omega^{+,1}_{11}-\Omega^{+,2}_{11}\right)(t, N_1,M_2),
\end{equation}
where
\begin{align}
\Omega^{+,1}_{11}(t, N_1,M_2)&=\int_{0}^{\mu_0}e^{-it\mu^4}e^{-i\theta_+(|N_1|+|M_2|)}\frac{\theta^3_+}{\mu{\rm sin}^2\theta_+}d\mu,\label{kernel of Omega p11}\\
\Omega^{+,2}_{11}(t, N_1,M_2)&=\int_{0}^{\mu_0}e^{-it\mu^4}e^{-i\theta_+|N_1|}\frac{\theta^2_+g(\mu)}{\mu{\rm sin}\theta_+}e^{b(\mu)|M_2|}d\mu.\label{kernel of Omega p12}
\end{align}

On one hand, one has
\begin{equation}\label{estimate for Omega p1 11}
\sup\limits_{N_1,M_2}\left|\Omega^{+,1}_{11}(t, N_1,M_2)\right|\lesssim|t|^{-\frac{1}{4}}.
\end{equation}
Indeed, applying the same variable substitution as in \eqref{varible substi1} to \eqref{kernel of Omega p11}, we obtain
\begin{equation}\label{new expre of Omegap1 11}
\Omega^{+,1}_{11}(t, N_1,M_2)=-\int_{r_0}^{0}e^{-it\left[(2-2{\rm cos}\theta_+)^2-\theta_{+}\left(-\frac{|N_1|+|M_2|}{t}\right)\right]}F_{11}(\theta_+)d\theta_+,\quad t\neq0,
\end{equation}
where $r_0\in (-\pi,0)$ satisfying ${\rm cos}r_0=1-\frac{\mu^2_0}{2}$ and
$$F_{11}(\theta_+)=\frac{\theta^3_+}{2(1-{\rm cos}\theta_+){\rm sin}\theta_+}.$$
Notice that $F_{11}(\theta_+)$ is continuously differentiable on $(r_0,0)$ and
$$\lim\limits_{\theta_+\rightarrow0}F_{11}(\theta_+)=1,\quad\lim\limits_{\theta_+\rightarrow0}F'_{11}(\theta_+)=0,$$
By Corollary \ref{corollary}, it follows that
$$|\eqref{new expre of Omegap1 11}|\lesssim \sup\limits_{s\in\R}\left|\int_{r_0}^{0}e^{-it\left[(2-2{\rm cos}\theta_+)^2-s\theta_+\right]}F_{11}(\theta_+)d\theta_+\right|\lesssim |t|^{-\frac{1}{4}}\left(1+\int_{r_0}^{0}|F'_{11}(\theta_+)|d\theta_+\right)\lesssim|t|^{-\frac{1}{4}}.$$
%$$\sup\limits_{N_1,M_2}\left|\Omega^{+,1}_{11}(t, N_1,M_2)\right|\lesssim|t|^{-\frac{1}{4}},$$
%where we used the fact that $\lim\limits_{\theta_+\rightarrow0}F^{(k)}_{11}(\theta_+)$ exist for $k=0,1$.
Thus, \eqref{estimate for Omega p1 11} is proved.

On the other hand,
\begin{equation}\label{estimate for Omega p2 11}
\sup\limits_{N_1,M_2}\left|\Omega^{+,2}_{11}(t, N_1,M_2)\right|\lesssim|t|^{-\frac{1}{4}}.
\end{equation}
Similarly, we apply the same variable substitution as in \eqref{varible substi1}, yielding
\begin{equation}\label{new expre of Omegap2 11}
\Omega^{+,2}_{11}(t, N_1,M_2)=-\int_{r_0}^{0}e^{-it\left[(2-2{\rm cos}\theta_+)^2-\theta_{+}\left(-\frac{|N_1|}{t}\right)\right]}\widetilde{F}_{11}(\theta_+,M_2)d\theta_+,
\end{equation}
where
$$\widetilde{F}_{11}(\theta_+,M_2)=\frac{\theta^2_+}{2(1-{\rm cos}\theta_+)}g(\mu(\theta_+))e^{b(\mu(\theta_+))|M_2|}:=f_{11}(\theta_+)e^{b(\mu(\theta_+))|M_2|}.$$
%$g(\mu)$ is defined in Lemma \ref{cancelation of caseI}.
We claim that $\lim\limits_{\theta_+\rightarrow0} \widetilde{F}_{11}(\theta_+,M_2)$ exists and
\begin{equation}\label{estimate of parti of Ftuta11}
\sup\limits_{M_2}\int_{r_0}^{0}\Big|\frac{\partial\widetilde{F}_{11} }{\partial\theta_+}(\theta_+,M_2)\Big|d\theta_+\lesssim 1.
\end{equation}
Then, by Corollary \eqref{corollary}, \eqref{estimate for Omega p2 11} follows from
% \sup\limits_{s\in\R}\left|\int_{r_0}^{0}e^{-it\left[(2-2{\rm cos}\theta_+)^2-s\theta_+\right]}\widetilde{F}_{11}(\theta_+)d\theta_+\right|
$$|\eqref{new expre of Omegap2 11}|\lesssim |t|^{-\frac{1}{4}}\left(\big|\lim\limits_{\theta_+\rightarrow0} \widetilde{F}_{11}(\theta_+,M_2)\big|+\int_{r_0}^{0}\Big|\frac{\partial\widetilde{F}_{11}}{\partial\theta_+}(\theta_+,M_2)\Big|d\theta_+\right)\lesssim |t|^{-\frac{1}{4}}.$$
Indeed, for any $M_2$, %it is clearly that $\widetilde{F}_{11}(\theta_+,M_2)$ is continuously differentiable on $(r_0,0)$. Moreover,
observe that $\mu(\theta_{+})\rightarrow0,b(\mu(\theta_{+}))\rightarrow0$ as $\theta_+\rightarrow0$, and
$$
\lim\limits_{\mu\rightarrow0}g(\mu)=1,\quad\lim\limits_{\mu\rightarrow0}g'(\mu)=0,$$
thus one can verify that $\lim\limits_{\theta_+\rightarrow0}f^{(k)}_{11}(\theta_+)$ exists for $k=0,1$. Consequently, $\lim\limits_{\theta_+\rightarrow0} \widetilde{F}_{11}(\theta_+,M_2)$ exists.
%$$\quad\lim\limits_{\theta_+\rightarrow0} \bar{F}_{11}(\theta_+)=1,\quad\lim\limits_{\theta_+\rightarrow0} \bar{F}'_{11}(\theta_+)=0,\quad\lim\limits_{\theta_+\rightarrow0}e^{b(\mu(\theta_+))|M_2|}=1,$$
%which implies that
%\begin{equation}\label{limits of F11tuta}
%\lim\limits_{\theta_+\rightarrow0}\tilde{F}^{(k)}_{11}(\theta_+,M_2)\ \ {\rm exist\  for}\  k=0,1.
%\end{equation}
A direct calculation yields that
$$\frac{\partial\widetilde{F}_{11}}{\partial\theta_+}\left(\theta_+,M_2\right)=\underbrace{f'_{11}(\theta_+)e^{b(\mu(\theta_+))|M_2|}}_{I_1}+f_{11}(\theta_+)\underbrace{b'(\mu(\theta_+))\mu'(\theta_+)|M_2|e^{b(\mu(\theta_+))|M_2|}}_{I_2},$$
where $I_1$ is uniformly bounded on $(r_0,0)$ since $b(\mu)<0$ for any $\mu\in(0,2)$, and the existence of $\lim\limits_{\theta_+\rightarrow0}f'_{11}(\theta_+)$. Moreover, $\|I_2\|_{L^{1}([r_0,0))}$ is controlled by 2 uniformly in $M_2$ from \eqref{integral ebnm}.%, noting that $b'(\mu(\theta_+))<0$ and $\mu'(\theta_+)<0$.
Therefore, \eqref{estimate of parti of Ftuta11} is established.

%\begin{equation}\label{estimate of differential for Ftuta11}
%\frac{\partial}{\partial\theta_+}\left(\tilde{F}_{11}(\theta_+,M_2)\right)\in L^{1}([r_0,0)),\quad{\rm uniformly\  in}\  M_2,
%\end{equation}
%which follows from the uniform boundedness of $\bar{F}'_{11}(\theta_+)e^{b(\mu(\theta_+))|M_2|}$ on $[r_0,0)$ and the fact
%\begin{equation}\label{estimate of differential for ebmutheta}
%\int_{r_0}^{0}\left|\frac{\partial}{\partial\theta_+}\left(e^{b(\mu(\theta_+))|M_2|}\right)d\theta_+\right|=\int_{r_0}^{0}b'(\mu(\theta_+))\mu'(\theta_+)|M_2|e^{b(\mu(\theta_+))|M_2|}d\theta_+\leq2,\ {\rm uniformly\ in}\ M_2,
%\end{equation}
%since that $b'(\mu)<0$ and $\mu'(\theta_+)<0$.
%Thus, by Lemma \ref{Von der}, \eqref{limits of F11tuta} and \eqref{estimate of differential for Ftuta11}, the \eqref{estimate for Omega p2 11} is obtained.
Combining \eqref{estimate for Omega p1 11},\eqref{estimate for Omega p2 11} and \eqref{omegap11 Omegap111 Omegap112}, \eqref{Omega pm11} holds for the $``+"$ case and we complete the proof of \eqref{estimate of K11 in Claim}.
\end{proof}
\vskip0.3cm
In the second part of this subsection, we deal with the $K^{\pm}_{12}(t,n,m)$ defined in \eqref{kernel of K12}. By \eqref{asy expan on 0}, it can be written as the following sum:
\begin{align}\label{kernel of K12(sum)}
K^{\pm}_{12}(t,n,m)=\sum\limits_{j=1}^{6}K^{\pm,j}_{12}(t,n,m),
\end{align}
where
\begin{align}\label{kernels of Kpmi12}
\begin{split}
K^{\pm,1}_{12}(t,n,m)&=\int_{0}^{\mu_0}e^{-it\mu^4}\mu^3\left[R^{\pm}_0\left(\mu^4\right)v\left(\mu QA^{\pm}_{11}Q\right)vR^{\pm}_0\left(\mu^4\right)\right](n,m)d\mu,\\
K^{\pm,2}_{12}(t,n,m)&=\int_{0}^{\mu_0}e^{-it\mu^4}\mu^3\left[R^{\pm}_0\left(\mu^4\right)v\left(\mu^2 QA^{\pm}_{21}Q\right)vR^{\pm}_0\left(\mu^4\right)\right](n,m)d\mu,\\
K^{\pm,3}_{12}(t,n,m)&=\int_{0}^{\mu_0}e^{-it\mu^4}\mu^3\left[R^{\pm}_0\left(\mu^4\right)v\left(\mu^2 S_0A^{\pm}_{22}\right)vR^{\pm}_0\left(\mu^4\right)\right](n,m)d\mu,\\
K^{\pm,4}_{12}(t,n,m)&=\int_{0}^{\mu_0}e^{-it\mu^4}\mu^3\left[R^{\pm}_0\left(\mu^4\right)v\left(\mu^2 A^{\pm}_{23}S_0\right)vR^{\pm}_0\left(\mu^4\right)\right](n,m)d\mu,\\
K^{\pm,5}_{12}(t,n,m)&=\int_{0}^{\mu_0}e^{-it\mu^4}\mu^3\left[R^{\pm}_0\left(\mu^4\right)v\left(\mu^3 A^{\pm}_{31}\right)vR^{\pm}_0\left(\mu^4\right)\right](n,m)d\mu,\\
K^{\pm,6}_{12}(t,n,m)&=\int_{0}^{\mu_0}e^{-it\mu^4}\mu^3\left[R^{\pm}_0\left(\mu^4\right)v\Gamma_4(\mu)vR^{\pm}_0\left(\mu^4\right)\right](n,m)d\mu.\\
\end{split}
\end{align}

  Based on Lemma \ref{cancelation of caseI},  the kernel $(R^{\pm}_0\left(\mu^4\right)vBvR^{\pm}_0\left(\mu^4\right))(n,m)$ contribute at least a factor of $\mu^3$ for $B=\mu QA^{\pm}_{11}Q,\ \mu^2QA^{\pm}_{21}Q,\ \mu^2S_0A^{\pm}_{22},\ \mu^2A^{\pm}_{23}S_0,\ \mu^3A^{\pm}_{31}$. This implies that one can follow a process similar to that used for $\widetilde{K}_{11}(t,n,m)$ to verify the same decay estimates for $K^{+,j}_{12}(t,n,m)$ with $j=1,2,3,4,5$. In fact, it is not difficult to derive the following corollary.
\begin{corollary}\label{estimate for Kpm12 2,3,4,5,6}
{ Under the assumptions in Proposition \ref{Proposition of K1}, let $K^{+,j}_{12}(t,n,m)$ be as in \eqref{kernels of Kpmi12}. Then
\begin{equation*}
\sup\limits_{n,m\in\Z}\left|K^{+,j}_{12}(t,n,m)\right|\lesssim |t|^{-\frac{1}{4}},\quad t\neq0,
\end{equation*}
holds for $j=1,2,3,4,5$.}
\end{corollary}
Finally, we focus on addressing $K^{+,6}_{12}(t,n,m)$ to complete the estimate for $K^{\pm}_{12}(t,n,m)$.
\begin{proposition}\label{claim of Kp612}
{ Under the assumptions in Propositions \ref{Proposition of K1}, let $K^{\pm,6}_{12}(t,n,m)$ be defined as in \eqref{kernels of Kpmi12}. Then
\begin{equation}\label{k126}
\sup\limits_{n,m\in\Z}\left|K^{+,6}_{12}(t,n,m)\right|\lesssim |t|^{-\frac{1}{4}},\quad t\neq0.
\end{equation}
%the estimates \eqref{estimate for Kpm12 1} hold for $K^{\pm,6}_{12}(t,n,m)$.
}
\end{proposition}
\begin{proof}
 We consider the $``+"$ case for instance. By \eqref{kernel of R0 boundary} and \eqref{kernels of Kpmi12}, it follows that
$$K^{+,6}_{12}(t,n,m)=-\frac{1}{16}\sum\limits_{j=1}^{4}\int_{0}^{\mu_0}e^{-it\mu^4}\Lambda^{+,j}_{12}(\mu,n,m)d\mu:=-\frac{1}{16}\sum\limits_{j=1}^{4}{\Omega}^{+,j}_{12}(t,n,m),$$
where $N_1=n-m_1$, $M_2=m-m_2$ and
\begin{align}\label{kernels of Omegaj12}
\begin{split}
\Lambda^{+,1}_{12}(\mu,n,m)&=\sum\limits_{m_1,m_2\in\Z}^{}e^{-i\theta_+(|N_1|+|M_2|)}\frac{\left(v\Gamma_4(\mu)v\right)(m_1,m_2)}{\mu{\rm sin}^2\theta_+},\\
\Lambda^{+,2}_{12}(\mu,n,m)&=\sum\limits_{m_1,m_2\in\Z}^{}e^{-i\theta_+|N_1|}\frac{\left(v\Gamma_4(\mu)v\right)(m_1,m_2)}{-\mu{\rm sin}\theta_+{\rm sin}\theta}e^{b(\mu)|M_2|},\\
\Lambda^{+,3}_{12}(\mu,n,m)&=\sum\limits_{m_1,m_2\in\Z}^{}e^{-i\theta_+|M_2|}\frac{\left(v\Gamma_4(\mu)v\right)(m_1,m_2)}{-\mu{\rm sin}\theta_+{\rm sin}\theta}e^{b(\mu)|N_1|},\\
\Lambda^{+,4}_{12}(\mu,n,m)&=\sum\limits_{m_1,m_2\in\Z}^{}\frac{\left(v\Gamma_4(\mu)v\right)(m_1,m_2)}{\mu{\rm sin}^2\theta}e^{b(\mu)(|N_1|+|M_2|)}.
\end{split}
\end{align}
Noting the symmetry between $\Lambda^{+,2}_{12}$ and $\Lambda^{+,3}_{12}$, it suffices to analyse the kernels $\Omega^{+,j}_{12}(t,n,m)$ for $j=1,2,4$.

On one hand, applying the variable substitution \eqref{varible substi1} to $\Omega^{+,j}_{12}(t,n,m)$ for $j=1,2$, we obtain that
\begin{align}\label{Omega121}
\begin{split}
\Omega^{+,1}_{12}(t,n,m)&=\int_{r_0}^{0}\sum\limits_{m_1,m_2\in\Z}^{}e^{-it(2-2{\rm cos}\theta_+)^2}e^{-i\theta_+(|N_1|+|M_2|)}\\
&\quad \times v(m_1)\varphi_1(\mu(\theta_+))(m_1,m_2)v(m_2)d\theta_+,
\end{split}
\end{align}
and
\begin{align}\label{Omega122}
\begin{split}
\Omega^{+,2}_{12}(t,n,m)&=\int_{r_0}^{0}\sum\limits_{m_1,m_2\in\Z}^{}e^{-it(2-2{\rm cos}\theta_+)^2}e^{-i\theta_+|N_1|}\\
&\quad\times v(m_1)\varphi_2(\mu(\theta_+))(m_1,m_2)v(m_2)e^{b(\mu(\theta_+))|M_2|}d\theta_+,
\end{split}
\end{align}
where
$$\varphi_{j}(\mu)=\frac{\Gamma_4(\mu)}{\mu^3}g_j(\mu)\ {\rm with}\ g_1(\mu)=\frac{1}{\sqrt{1-\frac{\mu^2}{4}}}\ {\rm and}\ g_2(\mu)=\frac{i}{\sqrt{1+\frac{\mu^2}{4}}}.$$
%In what follows, we will show that $$
Then, %if we can further demonstrate that $\lim\limits_{\theta_+\rightarrow0}\Phi_{j}(\mu\theta_+)$ exists and for $j=1,2$
\begin{align*}
\begin{split}
\left|\eqref{Omega121}\right|&\leq\sup\limits_{s\in\R}^{}\left|\int_{r_0}^{0}e^{-it\left[(2-2{\rm cos}\theta_+)^2-s\theta_+\right]}\Phi_1(\mu(\theta_+))d\theta_+\right|,\\
\left|\eqref{Omega122}\right|&\leq\sup\limits_{s\in\R}^{}\left|\int_{r_0}^{0}e^{-it\left[(2-2{\rm cos}\theta_+)^2-s\theta_+\right]}\Phi_2(\mu(\theta_+),m)d\theta_+\right|,
\end{split}
\end{align*}
where
\begin{align}
\begin{split}
\Phi_1(\mu)&=\sum\limits_{m_1\in\Z}^{}v(m_1)(\varphi_{1}(\mu)v)(m_1),\\
\Phi_2(\mu,m)&=\sum\limits_{m_1\in\Z}^{}v(m_1)\left(\varphi_{2}(\mu)\left(v(\cdot)e^{b(\mu)|\cdot-m|}\right)\right)(m_1).\\
%\varphi_{j}(\mu)&=\frac{\Gamma_4(\mu)}{\mu^3}g_j(\mu),\ g_1(\mu)=\frac{-1}{\sqrt{1-\frac{\mu^2}{4}}},\ g_2(\mu)=\frac{-i}{\sqrt{1+\frac{\mu^2}{4}}}.
\end{split}
\end{align}
%In what follows, we further demonstrate that both $\lim\limits_{\theta_+\rightarrow0}\Phi_{1}(\mu(\theta_+))$ and $\lim\limits_{\theta_+\rightarrow0}\Phi_{2}(\mu(\theta_+),m)$ exist and $\frac{\partial\Phi_j}{\partial\theta}\in L^{[r_0,0)}$ uniformly in $m\in\Z$. Then by Corollary \ref{corollary}, the
In what follows, we will show that
\begin{align}\label{limit of Phi12}
\lim\limits_{\theta_+\rightarrow0}\Phi_1(\mu(\theta_+))=\lim\limits_{\theta_+\rightarrow0}\Phi_{2}(\mu(\theta_+),m)=0,
\end{align}
and
\begin{align}\label{derivative of Phi12}
\|\left(\Phi_1(\mu(\theta_+))\right)'\|_{L^{1}([r_0,0))}+\left\|\frac{\partial\Phi_2}{\partial\theta_+}(\mu(\theta_+),m)\right\|_{L^{1}([r_0,0))}\lesssim 1,
\end{align}
uniformly in $m\in\Z$. Then, by Corollary \ref{corollary}, we can obtain that
\begin{equation}\label{estimate of Omega12,1,2}
\big|\Omega^{+,1}_{12}(t,n,m)\big|+\big|\Omega^{+,2}_{12}(t,n,m)\big|\lesssim |t|^{-\frac{1}{4}},\quad t\neq0.
\end{equation}
Indeed, noting that by virtue of \eqref{estimate of Gamma}, for $\mu\in(0,\mu_0]$,
\begin{equation}\label{property of varphi1,2}
\|\varphi_{j}(\mu)\|_{\B(0,0)}\lesssim\mu,\quad\left\|\varphi'_{j}(\mu)\right\|_{\B(0,0)}\lesssim1,\quad j=1,2,
\end{equation}
which implies that
\begin{align}
\begin{split}
\left|\Phi_{1}(\mu)\right|+\left|\Phi_{2}(\mu,m)\right|\lesssim\mu,\quad \left|\partial_{\mu}\Phi_{1}(\mu)\right|\lesssim1.
\end{split}
\end{align}
uniformly in $m\in \Z$. This proves \eqref{limit of Phi12}.
Since $\mu'(\theta_+)<0$, we have
$$\int_{r_0}^{0}\left|\left(\Phi_1(\mu(\theta_+))\right)'\right|d\theta_+ =\int_{r_0}^{0}\left|\left(\partial_{\mu}\Phi_1\right)(\mu(\theta_+))\mu'(\theta_+)\right|d\theta_+\lesssim1.$$
%Consequently, by Corollary \ref{corollary},
%$$|\eqref{Omega121}|\lesssim |t|^{-\frac{1}{4}}, \quad t\neq0.$$
Moreover, one can caculate that
\begin{align*}
\partial_{\mu}\left(\Phi_{2}(\mu,m)\right)&=\sum\limits_{m_1\in\Z}^{}v(m_1)\varphi'_2(\mu)\left(v(\cdot)e^{b(\mu)|\cdot-m|}\right)(m_1)\\
&\quad+\sum\limits_{m_1\in\Z}^{}v(m_1)\varphi_2(\mu)\left(v(\cdot)\partial_{\mu}\left(e^{b(\mu)|\cdot-m|}\right)\right)(m_1).
\end{align*}
Then by \eqref{property of varphi1,2},
\begin{align*}
\left|\partial_{\mu}\left(\Phi_{2}(\mu,m)\right)\right|\lesssim\|v\|^2_{\ell^{2}}+\|v\|_{\ell^{2}}\left\|v(\cdot)\partial_{\mu}\left(e^{b(\mu)|\cdot-m|}\right)\right\|_{\ell^1}.
\end{align*}
Hence, by \eqref{integral ebnm},
$$\int_{r_0}^{0}\left|\frac{\partial\Phi_2}{\partial\theta_+}\left(\mu(\theta_+),m\right)\right|d\theta_+\lesssim1+\sum\limits_{m_1\in\Z}^{}|v(m_1)|\int_{r_0}^{0}\left|\partial_{\mu}\left(e^{b(\mu)|m_1-m|}\right)\right|d\theta_+\lesssim1,$$
uniformly in $m\in \Z$.
Thus, \eqref{derivative of Phi12} is obtained. %by Lemma \ref{Von der}, we prove that \eqref{estimate for Kpm12 1} hold for $\widetilde{\Omega}^{+,j}_{12}(j=1,2)$, and so does $\widetilde{\Omega}^{+,3}_{12}$.

On the other hand, the kernel of ${\Omega}^{+,4}_{12}(t,n,m)$ is as follows:
$${\Omega}^{+,4}_{12}(t,n,m)=\int_{0}^{\mu_0}e^{-it\mu^4}\Phi_{4}(\mu,n,m)d\mu,$$
where
\begin{align*}
\Phi_{4}(\mu,n,m)&=\sum\limits_{m_1\in\Z}^{}v(m_1)e^{b(\mu)|n-m_1|}\left(\varphi_4(\mu)\left(v(\cdot)e^{b(\mu)|\cdot-m|}\right)\right)(m_1),\\
\varphi_4(\mu)&=\frac{\Gamma_4(\mu)}{\mu^3}g_4(\mu),\ g_4(\mu)=\frac{-1}{1+\frac{\mu^2}{4}}.
\end{align*}
By applying the Van der Corput Lemma directly, it follows that
\begin{equation}\label{omega124}
\big|{\Omega}^{+,4}_{12}(t,n,m)\big|\lesssim |t|^{-\frac{1}{4}}\left(|\Phi_{4}(\mu_0,n,m)|+\left\|\frac{\partial\Phi_4}{\partial\mu}(\mu,n,m)\right\|_{L^{1}((0,\mu_0])}\right)\lesssim |t|^{-\frac{1}{4}},\quad t\neq0,
\end{equation}
uniformly in $n,m\in\Z$, where the uniform boundedness of $|\Phi_{4}(\mu_0,n,m)|$ relies on the facts that $b(\mu)<0$ and $\varphi_4(\mu)$ also satisfies \eqref{property of varphi1,2}. The uniform estimate for the integral of partial derivative can be derived similarly to $\Phi_2(\mu,m)$. Therefore, combining \eqref{estimate of Omega12,1,2} and \eqref{omega124}, the desired result \eqref{k126} is obtained.
%Then follow the same argument about estimates for $\partial^{(k)}_{\mu}\Phi_4(\mu,n,m)$ for $k=0,1$ and apply the Von der corput Lemma directly, one can prove that the estimates \eqref{estimate for Kpm12 1} hold for $\widetilde{\Omega}^{+,4}_{12}$, and therefore, this completes the proof.
\end{proof}
In summary, combining Corollary \ref{estimate for Kpm12 2,3,4,5,6}, Proposition \ref{claim of Kp612} and Proposition \ref{claim of K11}, then Proposition \ref{Proposition of K1} is proved.
\subsection{The estimates of kernels $K_2^{\pm}(t,n,m)$}\label{subsection of case II}
\begin{proposition}\label{Proposition of K2}
{ Let $H=\Delta^2+V$ with $|V(n)|\lesssim \left<n\right>^{-\beta}$ for $\beta>2$. %Suppose that $H$ has no positive eigenvalue in $\rm{(}0,16\rm{)}$,
Then
\begin{equation}\label{estimate of Kpm2}
\sup\limits_{n,m\in\Z}\left|K^{\pm}_{2}(t,n,m)\right|\lesssim |t|^{-\frac{1}{4}},\quad t\neq0.%,\ {\rm uniformly\  in}\  n,m\in\Z.
\end{equation}
}
\end{proposition}
Recall from \eqref{kernels of Ki} that
\begin{equation}\label{kernel of K2(two parts)}
K^{\pm}_{2}(t,n,m)=\left(K^{\pm}_{21}-K^{\pm}_{22}\right)(t,n,m),
\end{equation}
where
\begin{align}\label{kernels of Kpm21,22}
\begin{split}
K^{\pm}_{21}(t,n,m)&=\int_{\mu_0}^{2-\mu_0}e^{-it\mu^4}\mu^3\left(R^{\pm}_{0}\left(\mu^4\right)VR^{\pm}_{0}\left(\mu^4\right)\right)(n,m)d\mu,\\
K^{\pm}_{22}(t,n,m)&=\int_{\mu_0}^{2-\mu_0}e^{-it\mu^4}\mu^3\left(R^{\pm}_{0}\left(\mu^4\right)VR^{\pm}_V\left(\mu^4\right)VR^{\pm}_{0}\left(\mu^4\right)\right)(n,m)d\mu.
\end{split}
\end{align}

Considering that there is no singularity on the interval $[\mu_0,2-\mu_0]$, it is convenient in this part to use the second form of $R^{\pm}_0\left(\mu^{4},n,m\right)$ given in \eqref{kernel of R0 boundary}, namely,
\begin{align}
R^{\pm}_0\left(\mu^{4},n,m\right)=e^{-i\theta_{\pm}|n-m|}\left(A^{\pm}(\mu)+B(\mu)e^{(b(\mu)+i\theta_{\pm})|n-m|}\right):=e^{-i\theta_{\pm}|n-m|}\widetilde{R}^{\pm}_{0}(\mu,n,m),\label{R0 in the R0tuta form}
\end{align}
where
\begin{equation}\label{Apm,Bmu}
A^{\pm}(\mu)=\frac{\pm i}{4\mu^3\sqrt{1-\frac{\mu^2}{4}}},\quad B(\mu)=\frac{-1}{4\mu^3\sqrt{1+\frac{\mu^2}{4}}}.
\end{equation}

It is straitforward to verify that $\widetilde{R}^{\pm}_{0}(\mu,n,m)$ satisfies the following property, which will be frequently used in this subsection. We summarize it below.%And the following lemma will be used often.
\begin{lemma}\label{property of R0tuta}
{
Let $0<a<b<2$, and let $\widetilde{R}^{\pm}_{0}(\mu,n,m)$ be defined as in \eqref{R0 in the R0tuta form}. Then there exists a constant $C(a,b)>0$ such that $$\sup\limits_{\mu\in[a,b]}\left|\widetilde{R}^{\pm}_{0}(\mu,n,m)\right|+\int_{a}^{b}\left|\partial_{\mu}\left(\widetilde{R}^{\pm}_{0}(\mu,n,m)\right)\right|d\mu\leq C(a,b),$$
uniformly in $n,m\in\Z$.}
\end{lemma}
\begin{comment}
\begin{proof}
Obviously, (i) can be concluded from the continuity of $A^{\pm}(\mu)$ and $B(\mu)$ on $[a,b]$ and $b(\mu)<0$. As for (ii), a direct calculation yields that
 $$\partial_{\mu}\left(\widetilde{R}^{\pm}_{0}(\mu,n,m)\right)=\left(A^{\pm}(\mu)\right)'+B'(\mu)e^{(b(\mu)+i\theta_{\pm})|n-m|}+B(\mu)\partial_{\mu}\left(e^{(b(\mu)+i\theta_{\pm})|n-m|}\right).$$
Since that $A^{\pm}(\mu),B(\mu)$ are continuously differentiable on $([a,b])$, it suffices to show that
$$\int_{a}^{b}\left|\partial_{\mu}\left(e^{(b(\mu)+i\theta_{\pm})|n-m|}\right)\right|d\mu\leq C(a,b).$$
Indeed, for any $\mu\in[a,b]$,
$$\left|(b'(\mu)+i\theta'_{\pm}(\mu))|n-m|e^{(b(\mu)+i\theta_{\pm})|n-m|}\right|\lesssim -b'(\mu)|n-m|e^{b(\mu)|n-m|}\in L^{1}([a,b]),$$
which is uniformly bounded based on \eqref{integral ebnm}, then the desired conclusion is obtained.
\end{proof}
\end{comment}
\vskip0.2cm
\begin{proposition}\label{claim of Kpm21}
{ Under the assumptions in Proposition \ref{Proposition of K2}, let $K^{\pm}_{21}(t,n,m)$ be defined as in \eqref{kernels of Kpm21,22}, then \eqref{estimate of Kpm2} holds for $K^{\pm}_{21}(t,n,m)$.
}
\end{proposition}
\begin{proof}
By \eqref{kernels of Kpm21,22} and \eqref{R0 in the R0tuta form}, one has
\begin{align*}
K^{\pm}_{21}(t,n,m)=\sum\limits_{m_1\in\Z}^{}\Omega^{\pm}_{21}(t,n,m,m_1)V(m_1),
\end{align*}
where $N_1=n-m_1,M_1=m-m_1$, and
\begin{align}\label{Omegapm21}
\begin{split}
\Omega^{\pm}_{21}(t,n,m,m_1)&=\int_{\mu_0}^{2-\mu_0}e^{-it\mu^4}e^{-i\theta_{\pm}(|N_1|+|M_1|)}F^{\pm}_{21}(\mu,n,m,m_1)d\mu,\\
F^{\pm}_{21}(\mu,n,m,m_1)&=\mu^3\widetilde{R}^{\pm}_{0}(\mu,n,m_1)\widetilde{R}^{\pm}_{0}(\mu,m_1,m).
\end{split}
\end{align}

We focus on $\Omega^{+}_{21}$, and $\Omega^{-}_{21}$ follows in a similar way. Applying the variable substitution \eqref{varible substi1} to \eqref{Omegapm21}, then for $t\neq0$, we obtain
\begin{equation}\label{new expre of Omegap 21}
\Omega^{+}_{21}(t,n,m,m_1)=-\int_{r_1}^{r_0}e^{-it\left[(2-2{\rm cos}\theta_+)^2-\theta_{+}\left(-\frac{|N_1|+|M_1|}{t}\right)\right]}F^{+}_{21}(\mu(\theta_+),n,m,m_1)d\theta_+,
\end{equation}
where $r_1\in(-\pi,0)$ satisfying ${\rm{cos}}r_1=1-\frac{(2-\mu_0)^2}{2}$. Therefore, by Lemma \ref{property of R0tuta}, we have
 $$\sup\limits_{\mu\in[\mu_0,2-\mu_0]}\big|F^{+}_{12}(\mu,n,m,m_1)\big|+\big\|(\partial_{\mu}F^{+}_{12})(\mu,n,m,m_1)\|_{L^{1}([\mu_0,2-\mu_0])} \lesssim 1,$$
 uniformly in $n,m,m_1\in\Z$. %As before, it is key to estimate the uniform boundedness of $\left\|\frac{\partial}{\partial\theta_+}F^{+}_{21}\right\|_{L^{1}([r_1,r_0])}.$ Since that $\mu'(\theta_+)=-\left(1-\frac{\mu^2}{4}\right)^{\frac{1}{2}}$, then the uniform boundedness follows from Lemma \ref{property of R0tuta}.
 Therefore by Corollary \ref{corollary}, it follows that
$$\left|\Omega^{+}_{21}(t,n,m,m_1)\right|\lesssim |t|^{-\frac{1}{4}},\ {\rm uniformly\ in}\ n,m,m_1\in\Z.$$
Thus,
$|K^{+}_{21}(t,n,m)|\lesssim|t|^{-\frac{1}{4}}$ is obtained.
\end{proof}

\begin{proposition}\label{claim of Kpm22}
{ Under the assumptions in Proposition \ref{Proposition of K2}, let $K^{\pm}_{22}(t,n,m)$ be defined as in \eqref{kernels of Kpm21,22}. Then \eqref{estimate of Kpm2} holds for $K^{\pm}_{22}(t,n,m)$.
}
\end{proposition}
\begin{proof}
As before, by \eqref{kernels of Kpm21,22} and \eqref{R0 in the R0tuta form}, one has
\begin{align}\label{kernels of Kpm22(new)}
K^{\pm}_{22}(t,n,m)=\int_{\mu_0}^{2-\mu_0}e^{-it\mu^4}\sum\limits_{m_1,m_2\in\Z}^{}e^{-i\theta_+(|N_1|+|M_2|)}\mu^3f^{\pm}_{22}(\mu,n,m,m_1,m_2)d\mu,
\end{align}
where $N_1=n-m_1,M_2=m-m_2$, and
\begin{align}\label{fpm22}
\begin{split}
f^{\pm}_{22}(\mu,n,m,m_1,m_2)=\widetilde{R}^{\pm}_{0}(\mu,n,m_1)\left(VR^{\pm}_V\left(\mu^4\right)V\right)(m_1,m_2)\widetilde{R}^{\pm}_{0}(\mu,m_2,m).
\end{split}
\end{align}

We take $K^{+}_{22}$ for instance. Applying the variable substitution \eqref{varible substi1} to \eqref{kernels of Kpm22(new)}, and using Corollary \ref{corollary}, we obtain
\begin{align}\label{estimate of Kp22}
\begin{split}
\left|K^{+}_{22}(t,n,m)\right|&\leq\sup\limits_{s\in\R}\left|\int_{r_1}^{r_0}e^{-it\left[(2-2{\rm cos}\theta_+)^2-s\theta_+\right]}F^{+}_{22}(\mu(\theta_+),n,m)d\theta_+\right|\\
&\lesssim|t|^{-\frac{1}{4}}\left(\left|F^{+}_{22}(\mu_0,n,m)\right|+\int_{r_1}^{r_0}\left|\frac{\partial F^{+}_{22}}{\partial\theta_+}(\mu(\theta_+),n,m)\right|d\theta_+\right)\lesssim|t|^{-\frac{1}{4}},
\end{split}
\end{align}
where
\begin{equation}\label{Fp22}
F^{+}_{22}(\mu,n,m)=\mu^3\sum\limits_{m_1,m_2\in\Z}^{}f^{+}_{22}(\mu,n,m,m_1,m_2):=\mu^3\widetilde{F}^{+}_{22}(\mu,n,m)
\end{equation}
and for the last inequality, we have used their uniform boundedness in advance. In what follows, we explain it in detail.

On one hand, for any $\mu\in[\mu_0,2-\mu_0]$, by Lemma \ref{property of R0tuta}~(i) and the continuity of $R^{\pm}_{V}(\mu^4)$ in Theorem \ref{LAP-theorem}, take $\frac{1}{2}<\varepsilon_1<\beta-\frac{1}{2}$, then exists a constant $C(\mu_0)>0$ such that
$$|F^{+}_{22}(\mu,n,m)|\lesssim \left\|V(\cdot)\left<\cdot\right>^{\varepsilon_1}\right\|^2_{\ell^2}\left\|R^{+}_{V}(\mu^4)\right\|_{\B(\varepsilon_1,-\varepsilon_1)}\leq C(\mu_0).$$

On the other hand, %for the second term in the second inequality of \eqref{estimate of Kp22}, firstly, by \eqref{Fp22},
a direct calculation yields that
\begin{align}\label{derivative of Fp22mu}
\begin{split}
(\partial_{\mu}\widetilde{F}^{+}_{22})(\mu,n,m)&=\sum\limits_{m_1,m_2\in\Z}^{}\left(\partial_{\mu}\widetilde{R}^{+}_{0}\right)(\mu,n,m_1)\left(VR^{+}_V\left(\mu^4\right)V\right)(m_1,m_2)\widetilde{R}^{+}_{0}(\mu,m_2,m)\\
&\quad+ \sum\limits_{m_1,m_2\in\Z}^{}\widetilde{R}^{+}_{0}(\mu,n,m_1)\left(V\partial_{\mu}\left(R^{+}_V\left(\mu^4\right)\right)V\right)(m_1,m_2)\widetilde{R}^{+}_{0}(\mu,m_2,m)\\
&\quad+\sum\limits_{m_1,m_2\in\Z}^{}\widetilde{R}^{+}_{0}(\mu,n,m_1)\left(VR^{+}_V\left(\mu^4\right)V\right)(m_1,m_2)\left(\partial_{\mu}\widetilde{R}^{+}_{0}\right)(\mu,m_2,m)\\
&:=L^{+}_1(\mu,n,m)+L^{+}_2(\mu,n,m)+L^{+}_3(\mu,n,m),
\end{split}
\end{align}
Then it suffices to verify that
\begin{equation}\label{integral of Li}
\int_{r_1}^{r_0}\left|L^{+}_{j}\left(\mu(\theta_+),n,m\right)\mu'(\theta_+)\right|d\theta_+\lesssim1,\quad j=1,2,3,
\end{equation}
uniformly in $n,m\in\Z$.
In fact, for $j=1$, take $\frac{1}{2}<\varepsilon_2<\beta-1$,
\begin{align*}
\left|L^{+}_{1}(\mu,n,m)\right|&\leq\left\|\left<\cdot\right>^{\varepsilon_2}V(\cdot)\left(\partial_{\mu}\widetilde{R}^{+}_{0}\right)(\mu,n,\cdot)\right\|_{\ell^2}\left\|R^{+}_{V}(\mu^4)\right\|_{\B(\varepsilon_2,-\varepsilon_2)}\left\|\left<\cdot\right>^{\varepsilon_2}V(\cdot)\widetilde{R}^{+}_{0}(\mu,\cdot,m)\right\|_{\ell^2}\\
&\lesssim\sum\limits_{m_1\in\Z}\left<m_1\right>^{\varepsilon_2}|V(m_1)|\left|\left(\partial_{\mu}\widetilde{R}^{+}_{0}\right)(\mu,n,m_1)\right|.
\end{align*}
By Lemma \ref{property of R0tuta}, it follows that
\begin{equation}
\int_{r_1}^{r_0}\left|L^{+}_{1}\left(\mu(\theta_+),n,m\right)\mu'(\theta_+)\right|d\theta_+\lesssim \sum\limits_{m_1\in\Z}\left<m_1\right>^{\varepsilon_2}|V(m_1)|\int_{\mu_0}^{2-\mu_0}\left|\left(\partial_{\mu}\widetilde{R}^{+}_{0}\right)(\mu,n,m_1)\right|d\mu\lesssim1.
\end{equation}
%where we used the fact that $\mu'(\theta_+)<0$.
By symmetry, the same argument applies to $L^{+}_{3}$.

For $j=2$, for any $\mu\in[\mu_0,2-\mu_0]$, taking $\frac{3}{2}<\varepsilon_3<\beta-\frac{1}{2}$, we utilize the continuity of $\partial_{\mu}\left(R^{+}_V\left(\mu^4\right)\right)$ in Theorem \ref{LAP-theorem} to obtain that
\begin{align*}
\left|L^{+}_{2}(\mu,n,m)\right|\leq\left\|\left<\cdot\right>^{\varepsilon_3}V(\cdot)\widetilde{R}^{+}_{0}(\mu,n,\cdot)\right\|_{\ell^2}\left\|\partial_{\mu}\left(R^{+}_{V}(\mu^4)\right)\right\|_{\B(\varepsilon_3,-\varepsilon_3)}\left\|\left<\cdot\right>^{\varepsilon_3}V(\cdot)\widetilde{R}^{+}_{0}(\mu,\cdot,m)\right\|_{\ell^2}\lesssim1.
\end{align*}
This implies that \eqref{integral of Li} holds for $j=2$. Thus, we complete the proof.
\end{proof}
Therefore, combining Propositions \ref{claim of Kpm21} and \ref{claim of Kpm22}, then Proposition \ref{Proposition of K2} is established.
\subsection{The estimates of kernels $K_3^{\pm}(t,n,m)$}\label{subsection of case III}
\begin{proposition}\label{Proposition of K3}
{ Let $H=\Delta^2+V$ with $|V(n)|\lesssim \left<n\right>^{-\beta}$ for $\beta>7$, and suppose that %$H$ has no positive eigenvalue in $\rm{(}0,16\rm{)}$ and
16 is a regular point of $H$. Then,
\begin{equation}\label{estimate of Kpm3}
\left|K^{\pm}_{3}(t,n,m)\right|\lesssim |t|^{-\frac{1}{4}},\quad t\neq0,\ {\rm uniformly\ in}\ n,m\in\Z.
\end{equation}
}
\end{proposition}
Under the assumptions of Proposition \ref{Proposition of K3}, we first recall that the kernel of $K^{\pm}_{3}(t,n,m)$ in \eqref{kernels of Ki} is given by:
\begin{equation*}%\label{kernel of K3}
K^{\pm}_{3}(t,n,m)=\int_{2-\mu_0}^{2}e^{-it\mu^4}\mu^3\left[R^{\pm}_0\left(\mu^4\right)v\left(M^{\pm}\left(\mu\right)\right)^{-1}vR^{\pm}_0\left(\mu^4\right)\right](n,m)d\mu,
\end{equation*}
and the asymptotic expansion \eqref{asy expan on 2} of $\left(M^{\pm}\left(\mu\right)\right)^{-1}$:
\begin{equation*}%\label{asy of Mmu 2(new)}
\left(M^{\pm}\left(\mu\right)\right)^{-1}=\widetilde{Q}B_{01}\widetilde{Q}+(2-\mu)^{\frac{1}{2}}B^{\pm}_{11} +\Gamma_{1}(2-\mu),\quad\mu\in[2-\mu_0,2).
\end{equation*}
Then, we further obtain:
$$K^{\pm}_{3}(t,n,m)=\sum\limits_{j=1}^{3}K^{\pm}_{3j}(t,n,m),$$
where
\begin{align}\label{kernels of Kpm3i}
\begin{split}
K^{\pm}_{31}(t,n,m)&=\int_{2-\mu_0}^{2}e^{-it\mu^4}\mu^3\left(R^{\pm}_0\left(\mu^4\right)v\widetilde{Q}B_{01}\widetilde{Q}vR^{\pm}_0\left(\mu^4\right)\right)(n,m)d\mu,\\
K^{\pm}_{32}(t,n,m)&=\int_{2-\mu_0}^{2}e^{-it\mu^4}\mu^3\left(R^{\pm}_0\left(\mu^4\right)v(2-\mu)^{\frac{1}{2}}B^{\pm}_{11} vR^{\pm}_0\left(\mu^4\right)\right)(n,m)d\mu,\\
K^{\pm}_{33}(t,n,m)&=\int_{2-\mu_0}^{2}e^{-it\mu^4}\mu^3\left(R^{\pm}_0\left(\mu^4\right)v\Gamma_{1}(2-\mu)vR^{\pm}_0\left(\mu^4\right)\right)(n,m)d\mu.
\end{split}
\end{align}
%In what follows, we only deal with the $K^{\pm}_{31}(t,n,m)$ and as for the cases $K^{\pm}_{3i}(i=2,3)$, one can deduce the estimates in a similar way in Claims \ref{claim of estimate for kpmi12,i=1} and \ref{claim of Kp612} of Subsection \ref{subsec of Case I}, and we omit the details for brevity.
Thus, it suffices to show that \eqref{estimate of Kpm3} holds for $K^{\pm}_{3j}(t,n,m)$ with $j=1,2,3$, respectively.

Compared to Subsection \ref{subsec of Case I}, a distinction lies in that it is not straightforward to use the cancelation condition \eqref{cancel of Qtuta} of $\widetilde{Q}$ to eliminate the singularity near $\mu=2$. Therefore, before proceeding with the proof, we first address this issue. %thing that we should consider is how to utilize the cancelation of $\tilde{Q}$ to kill some singularity near $\mu=2$.

Recalling the unitary operator $J$ defined in \eqref{J}, i.e., $(J\phi)(n)=(-1)^{n}\phi(n)$, one can observe that
\begin{equation*}
JR_{-\Delta}(z)J=-R_{-\Delta}(4-z),\quad z\in\C\setminus[0,4].
\end{equation*}
By the limiting absorption principle, it follows that
\begin{equation}\label{relation of Rminus Laplacian 0and4}
JR^{\pm}_{-\Delta}(\mu^2)J=-R^{\mp}_{-\Delta}(4-\mu^2),\quad \mu\in(0,2).
\end{equation}
\begin{comment}
which together with \eqref{R0lambda}, $J^2=1, \tilde{v}=Jv$ yields that
\begin{align}\label{R0vQB01QvR0 decomposition}
R^{\pm}_0\left(\mu^4\right)v\tilde{Q}B_{01}\tilde{Q}vR^{\pm}_0\left(\mu^4\right)=\frac{1}{4\mu^4}\sum\limits_{j=1}^{4}\Lambda^{\pm,j}_{31}(\mu),
\end{align}
where
\begin{align}\label{R0vQB01QvR0 decomposition parts}
\begin{split}
\Lambda^{\pm,1}_{31}(\mu)&=JR^{\mp}_{-\Delta}(4-\mu^2)\tilde{v}\tilde{Q}B_{01}\tilde{Q}\tilde{v}R^{\mp}_{-\Delta}(4-\mu^2)J,\\
\Lambda^{\pm,2}_{31}(\mu)&=JR^{\mp}_{-\Delta}(4-\mu^2)\tilde{v}\tilde{Q}B_{01}\tilde{Q}\tilde{v}JR_{-\Delta}\left(-\mu^2\right),\\
\Lambda^{\pm,3}_{31}(\mu)&=R_{-\Delta}\left(-\mu^2\right)J\tilde{v}\tilde{Q}B_{01}\tilde{Q}\tilde{v}R^{\mp}_{-\Delta}(4-\mu^2)J,\\
\Lambda^{\pm,4}_{31}(\mu)&=R_{-\Delta}\left(-\mu^2\right)J\tilde{v}\tilde{Q}B_{01}\tilde{Q}\tilde{v}JR_{-\Delta}\left(-\mu^2\right).
\end{split}
\end{align}
\end{comment}
%Similar to Lemma \ref{cancelation of caseI}, based on the cancelation condition \eqref{cancel of Qtuta} of $\tilde{Q}$,
With this, we can establish the following cancelation lemma.
\begin{lemma}\label{cancelation of caseIII}
{ Let $\widetilde{Q}$ be defined in \eqref{definition of Q,S0,Qtuta} and $\tilde{v}=Jv$. Then for any $f\in\ell^{2}(\Z)$, we have
\begin{itemize}
\item[(i)] $\left(R^{\mp}_{-\Delta}(4-\mu^2)\tilde{v}\widetilde{Q}f\right)(n)=\frac{\theta_{\mp}}{2{\rm sin}\theta_{\mp}}\sum\limits_{m\in\Z}\int_{0}^{1}{\rm sign}(n-\rho m)e^{-i\theta_{\mp}|n-\rho m|}d\rho\cdot m\left(\tilde{v}\widetilde{Q}f\right)(m)$,
    \vskip0.2cm
\item [(ii)] $\widetilde{Q}\big(\tilde{v}R^{\mp}_{-\Delta}(4-\mu^2)f\big)=\widetilde{Q}\left(\frac{n\tilde{v}(n)\theta_{\mp}}{2{\rm sin}\theta_{\mp}}\sum\limits_{m\in\Z}\int_{0}^{1}{\rm sign}(m-\rho n)e^{-i\theta_{\mp}|m-\rho n|}d\rho \cdot f(m)\right)$,
\end{itemize}
where $\theta_{\pm}$ satisfies ${\rm cos}\theta_{\pm}=\frac{\mu^2}{2}-1$.
%defined in \eqref{kernel of lapa boundary}.
}
\end{lemma}

\begin{proof}
By utilizing the property \eqref{cancel of Qtuta} of $\widetilde{Q}$ and following the proof procedure for {\bf{(3) and (4)}} in Lemma \ref{cancelation of caseI}, we obtain the desired result.
\end{proof}
\begin{remark}
{\rm From \eqref{kernel of R0 boundary}, we observe that the singularity $(2-\mu)^{-\frac{1}{2}}$ of the kernel $R^{\pm}_0(\mu^4,n,m)$ near $\mu=2$ originates from that of $R^{\pm}_{-\Delta}(\mu^2)$. This singularity, in turn, can be transferred to that of $R^{\mp}_{-\Delta}(4-\mu^2)$ via the unitary transform $J$. Recalling \eqref{kernel of lapa boundary}, the kernel of $R^{\mp}_{-\Delta}(4-\mu^2)$ is given by:
$$R^{\mp}_{-\Delta}(4-\mu^2, n,m)=\frac{-i}{2{\rm sin}\theta_{\mp}}e^{-i\theta_{\mp}|n-m|}.$$
%where ${\rm sin}\theta_{\mp}=\pm\frac{\mu}{2}(4-\mu^2)^{-\frac{1}{2}}$. %can be calculated as the form $$
Noting that $\theta_{\pm}=O((2-\mu)^{\frac{1}{2}})$ as $\mu\rightarrow2$, this implies that the kernels presented in this lemma can eliminate the singularity near $\mu=2$.
}
\end{remark}

Using Lemma \ref{cancelation of caseIII}, we now establish the estimate for the kernel $K^{\pm}_{31}(t,n,m)$.
\begin{proposition}\label{claim of Kpm31}
{ Under the assumptions in Proposition \ref{Proposition of K3}, let $K^{\pm}_{31}(t,n,m)$ be defined as in \eqref{kernels of Kpm3i}. One has
\begin{equation}\label{estimate of K31}
\left|K^{\pm}_{31}(t,n,m)\right|\lesssim |t|^{-\frac{1}{4}},\quad t\neq0,\ {\rm uniformly\ in}\ n,m\in\Z.
\end{equation}
%the estimate \eqref{estimate of Kpm3} holds for $K^{\pm}_{31}(t,n,m)$ defined in \eqref{kernels of Kpm3i}.
}
\end{proposition}
\begin{proof}
From \eqref{relation of Rminus Laplacian 0and4} and \eqref{R0mu4 and Rdeltamu2}, we obtain that
\begin{align}\label{R0vQB01QvR0 decomposition}
R^{\pm}_0\left(\mu^4\right)v\widetilde{Q}B_{01}\widetilde{Q}vR^{\pm}_0\left(\mu^4\right)=\frac{1}{4\mu^4}\sum\limits_{j=1}^{4}\Lambda^{\pm,j}_{31}(\mu),
\end{align}
where
\begin{align}\label{R0vQB01QvR0 decomposition parts}
\begin{split}
\Lambda^{\pm,1}_{31}(\mu)&=JR^{\mp}_{-\Delta}(4-\mu^2)\tilde{v}\widetilde{Q}B_{01}\widetilde{Q}\tilde{v}R^{\mp}_{-\Delta}(4-\mu^2)J,\\
\Lambda^{\pm,2}_{31}(\mu)&=JR^{\mp}_{-\Delta}(4-\mu^2)\tilde{v}\widetilde{Q}B_{01}\widetilde{Q}\tilde{v}JR_{-\Delta}\left(-\mu^2\right),\\
\Lambda^{\pm,3}_{31}(\mu)&=R_{-\Delta}\left(-\mu^2\right)J\tilde{v}\widetilde{Q}B_{01}\widetilde{Q}\tilde{v}R^{\mp}_{-\Delta}(4-\mu^2)J,\\
\Lambda^{\pm,4}_{31}(\mu)&=R_{-\Delta}\left(-\mu^2\right)J\tilde{v}\widetilde{Q}B_{01}\widetilde{Q}\tilde{v}JR_{-\Delta}\left(-\mu^2\right).
\end{split}
\end{align}
The kernel $K^{\pm}_{31}(t,n,m)$ in \eqref{kernels of Kpm3i} can be further expressed as follows:
\begin{align*}
K^{\pm}_{31}(t,n,m)=\sum\limits_{j=1}^{4}K^{\pm,j}_{31}(t,n,m),
\end{align*}
where
\begin{align*}
K^{\pm,j}_{31}(t,n,m)=\frac{1}{4}\int_{2-\mu_0}^{2}e^{-it\mu^4}\frac{\Lambda^{\pm,j}_{31}(\mu)}{\mu}(n,m)d\mu.
\end{align*}
\begin{comment}
where
\begin{align*}
K^{\pm,1}_{31}(t,n,m)&=\int_{2-\mu_0}^{2}e^{-it\mu^4}\big[JR^{\mp}_{-\Delta}(4-\mu^2)\tilde{v}\widetilde{Q}B_{01}\widetilde{Q}\tilde{v}R^{\mp}_{-\Delta}(4-\mu^2)J\big](n,m)\frac{d\mu}{\mu},\\
K^{\pm,2}_{31}(t,n,m)&=\int_{2-\mu_0}^{2}e^{-it\mu^4}\big[JR^{\mp}_{-\Delta}(4-\mu^2)\tilde{v}\widetilde{Q}B_{01}\widetilde{Q}\tilde{v}JR_{-\Delta}\left(-\mu^2\right)\big](n,m)\frac{d\mu}{\mu},\\
K^{\pm,3}_{31}(t,n,m)&=\int_{2-\mu_0}^{2}e^{-it\mu^4}\big[R_{-\Delta}\left(-\mu^2\right)J\tilde{v}\widetilde{Q}B_{01}\widetilde{Q}\tilde{v}R^{\mp}_{-\Delta}(4-\mu^2)J\big](n,m)\frac{d\mu}{\mu},\\
K^{\pm,4}_{31}(t,n,m)&=\int_{2-\mu_0}^{2}e^{-it\mu^4}\big[R_{-\Delta}\left(-\mu^2\right)J\tilde{v}\widetilde{Q}B_{01}\widetilde{Q}\tilde{v}JR_{-\Delta}\left(-\mu^2\right)\big](n,m)\frac{d\mu}{\mu}.
\end{align*}
\end{comment}

By symmetry, it suffices to prove that the estimates \eqref{estimate of K31} hold for  $K^{\pm,j}_{31}(t,n,m)$ with $j=1,2,4$. For illustration, we focus on the $``-"$ case.

\textbf{(i)}~By virtue of Lemma \ref{cancelation of caseIII}, it follows that
\begin{align*}
\begin{split}
K^{-,1}_{31}(t,n,m)&=\sum\limits_{m_1,m_2\in\Z}\int_{[0,1]^2}(-1)^{n+m}{\rm sign}(N_1){\rm sign}(M_2)\Omega^{-,1}_{31}(t,N_1,M_2)d\rho_1d\rho_2\\
&\quad\times m_1m_2\left(\tilde{v}\widetilde{Q}B_{01}\widetilde{Q}\tilde{v}\right)(m_1,m_2),\\
\end{split}
\end{align*}
where $N_1=n-\rho_1m_1$, $M_2=m-\rho_2m_2$, and
\begin{align}\label{kernel of Omegam31 1}
\begin{split}
\Omega^{-,1}_{31}(t,N_1,M_2)=\frac{1}{16}\int_{2-\mu_0}^{2}e^{-it\mu^4}e^{-i\theta_{+}(|N_1|+|M_2|)}\frac{\theta^2_+}{\mu{\rm sin}^2\theta_+}d\mu.
\end{split}
\end{align}
We apply the following variable substitution to \eqref{kernel of Omegam31 1}:
\begin{align}\label{varible substi2}
{\rm cos}\theta_{+}=\frac{\mu^2}{2}-1 \Longrightarrow\  \frac{d\mu}{d\theta_+}=-\frac{{\rm sin}\theta_+}{\mu},\quad \theta_+\rightarrow0\ {\rm as}\  \mu\rightarrow2,
\end{align}
 obtaining that
 \begin{align}\label{new expr of Omega31 1}
 \Omega^{-,1}_{31}(t,N_1,M_2)=\frac{-1}{16}\int_{r_2}^{0}e^{-it(2+2{\rm cos}\theta_+)^2}e^{-i\theta_+(|N_1|+|M_2|)}F_{31}(\theta_{+})d\theta_+,
 \end{align}
with $r_2={\rm arccos}\left(\frac{(2-\mu_0)^2}{2}-1\right)\in(-\pi,0)$ and
 $$F_{13}(\theta_+)=\frac{\theta^2_+}{(2+2{\rm cos}\theta_+){\rm sin}\theta_+}.$$
%$$\left|\eqref{kernel of Omegam31 1}\right|\lesssim\sup\limits_{s\in\R}\left|\int_{r_2}^{0}e^{-it\left[(2+2{\rm cos}\theta_+)^2-s\theta_+\right]}F_{31}(\theta_+)d\theta_+\right|$$
%where , and Then following the similar process in Claim \ref{claim of estimate for kpmi12,i=1} and by Lemma \ref{Von der}, the desired estimate \eqref{estimate of Kpm3} holds for $\Omega^{-,1}_{31}$ and so does $K^{-,1}_{31}$.
Noting that $\lim\limits_{\theta_+\rightarrow0}F^{(k)}_{31}(\theta_+)$ exists for $k=0,1$, it concludes from Corollary \ref{corollary} that
\begin{equation}\label{estimate for omega31,1}
\big|\eqref{new expr of Omega31 1}\big|\lesssim |t|^{-\frac{1}{4}},\quad t\neq0,\ {\rm uniformly\ in}\ N_1,M_2,
\end{equation}
which implies that \eqref{estimate of K31} holds for $K^{-,1}_{31}(t,n,m)$.
\vskip0.2cm
\textbf{(ii)}~Similarly, we have
\begin{align*}
\begin{split}
K^{-,2}_{31}(t,n,m)&=\sum\limits_{m_1,m_2\in\Z}\int_{0}^{1}(-1)^{n}{\rm sign}(N_1)\Omega^{-,2}_{31}(t,N_1,M_2)d\rho\\
&\quad\times (-1)^{m_2}m_1\left(\tilde{v}\widetilde{Q}B_{01}\widetilde{Q}\tilde{v}\right)(m_1,m_2),
\end{split}
\end{align*}
where $N_1=n-\rho m_1$, $M_2=m-m_2$, and
\begin{align}\label{kernel of Omegam31 2}
\begin{split}
\Omega^{-,2}_{31}(t,N_1,M_2)%&=\frac{1}{16}\int_{2-\mu_0}^{2}e^{-it\mu^4}e^{-i\theta_{+}|N_1|}e^{b(\mu)|M_2|}\frac{-i\theta_+}{\mu{\rm sin}\theta_+{\rm sin}\theta}d\mu\\
%\xlongequal[]{{\rm by}\ \eqref{varible substi2}}
={\frac{1}{16}\int_{r_2}^{0}e^{-it\left[(2+2{\rm cos}\theta_+)^2-\theta_+\left(-\frac{|N_1|}{t}\right)\right]}\widetilde{F}_{31}(\theta_+,M_2)d\theta_+}
\end{split}
\end{align}
with $$\widetilde{F}_{31}(\theta_+,M_2)=\frac{\theta_+f_{31}(\mu(\theta_+))}{2(1+{\rm cos}\theta_+)}e^{b(\mu(\theta_+))|M_2|}:=\widetilde{f}_{31}(\theta_+)e^{b(\mu(\theta_+))|M_2|},\quad f_{31}(\mu)=\frac{-1}{\mu\sqrt{1+\frac{\mu^2}{4}}}.$$
 By a method similar to used for \eqref{new expre of Omegap2 11}, one can obtain the desired estimate \eqref{estimate of K31} for $K^{-,2}_{31}(t,n,m)$. %and so does $K^{-,3}_{31}$.
\vskip0.2cm
\textbf{(iii)}~Finally, as for the $K^{-,4}_{31}$, we can calculate that
\begin{align*}
\begin{split}
K^{-,4}_{31}(t,n,m)=\sum\limits_{m_1,m_2\in\Z}\Omega^{-,4}_{31}(t,N_1,M_2)\left(J\tilde{v}\widetilde{Q}B_{01}\widetilde{Q}\tilde{v}J\right)(m_1,m_2),
\end{split}
\end{align*}
where $N_1=n-m_1$, $M_2=m-m_2$, and
\begin{align}\label{kernel of Omegam31 4}
\begin{split}
\Omega^{-,4}_{31}(t,N_1,M_2)=\frac{1}{16}\int_{2-\mu_0}^{2}e^{-it\mu^4}\tilde{g}(\mu)e^{b(\mu)(|N_1|+|M_2|)}d\mu,
\end{split}
\end{align}
with $\tilde{g}(\mu)=\frac{1}{\mu^3(1+\frac{\mu^2}{4})}$. Following a process similar to that for $K_{0,2}(t,n,m)$ defined in \eqref{decompose of K0+}, the estimate \eqref{estimate for omega31,1} holds for $\Omega^{-,4}_{31}$ and so does $K^{-,4}_{31}(t,n,m)$. This completes the proof.
\end{proof}
\begin{remark}
{\rm We note that both variable substitutions \eqref{varible substi1} and \eqref{varible substi2}  play important roles in handling the oscillatory integrals. However, they exhibit slight differences in addressing singularity. Specifically, the \eqref{varible substi1} does not alter the singularity near $\mu=0$, whereas \eqref{varible substi2} decreases a singularity of order $(2-\mu)^{-\frac{1}{2}}$. This implies that the kernel $K^{\pm}_{32}(t,n,m)$ has no singularity near $\mu=2$.
%as can be inferred from the variable substitution \eqref{varible substi1}.
%dealing with the terms $K^{\pm,j}_{31}(t,n,m)$ for $j=1,2,3$, which contain the singularity part $R^{\pm}_{-\Delta}(\mu^2)$ of $R^{\pm}_{0}(\mu^4)$, in essence, it only needs a power $(2-\mu)^{\frac{1}{2}}$ to eliminate their singularity such that we can use Corollary \ref{corollary} obtaining the desired decay estimate, since that the variable substitution can bring a power of $(2-\mu)^{\frac{1}{2}}$.
}
\end{remark}
Based on this observation, one can verify that the following proposition holds.
\begin{proposition}\label{claim of Kpm32}
{ Under the assumptions in Proposition \ref{Proposition of K3}, let $K^{\pm}_{32}(t,n,m)$ be defined as in \eqref{kernels of Kpm3i}. Then
\begin{equation*}
\left|K^{\pm}_{32}(t,n,m)\right|\lesssim |t|^{-\frac{1}{4}},\quad t\neq0,\ {\rm uniformly\ in}\ n,m\in\Z.
\end{equation*}
}
\end{proposition}
Finally, the estimate for $K^{\pm}_{33}(t,n,m)$ can be derived by following the proof of Proposition \ref{claim of Kp612}.
\begin{proposition}\label{claim of K33}
{ Under the assumptions in Propositions \ref{Proposition of K3}, let $K^{\pm}_{33}(t,n,m)$ be defined as in \eqref{kernels of Kpm3i}. Then
\begin{equation*}
\left|K^{\pm}_{33}(t,n,m)\right|\lesssim |t|^{-\frac{1}{4}},\quad t\neq0,\ {\rm uniformly\ in}\ n,m\in\Z.
\end{equation*}
%the estimates \eqref{estimate for Kpm12 1} hold for $K^{\pm,6}_{12}(t,n,m)$.
}
\end{proposition}
Therefore, combining Propositions \ref{claim of Kpm31}, \ref{claim of Kpm32} and \ref{claim of K33}, then  Proposition \ref{Proposition of K3} is established. Together with Propositions \ref{Proposition of K1} and  \ref{Proposition of K2}, we finish  the whole proof of  Theorem \ref{theorem of estimate for K123}.
%\section{Proof of Theorems \ref{LAP-theorem} and \ref{asy theor of Mpm mu}}\label{sec of proof of LAP and Asy}

\section{Proof of Theorem \ref{LAP-theorem}}\label{proof of LAP}
In this section, we are devoted to  completing the proof of Theorem \ref{LAP-theorem}, i.e., the limiting absorption principle for $\Delta^2+V$.
To the end, it suffices to prove the following Proposition \ref{LAP1}.
\begin{proposition}\label{LAP1}
{ Let $H=\Delta^2+V$ with $|V(n)|\lesssim \left<n\right>^{-\beta}$ for some $\beta>1$ and $\mcaI=(0,16)$. Given $\lambda\in \mcaI$, let $\mcaJ$ be the neighborhood of $\lambda$ defined in \eqref{mourre estim} below. For any relatively compact interval $I\subseteq\mcaJ\setminus\sigma_{p}(H)$, define $\widetilde{I}=\{z:\Re z\in I,\ 0<|\Im z|\leq1\}$. Then, for any $j\in\left\{0,\cdots,[\beta]-1\right\}$ and $j+\frac{1}{2}<s\leq[\beta]$, the following statements hold:}
\begin{itemize}
{
\item [(i)] \begin{equation}\label{jth reso estim}
    \sup\limits_{z\in\widetilde{I}}\left\|R^{(j)}_V(z)\right\|_{\B(s,-s)}<\infty.
    \end{equation}
    \vskip0.1cm
\item[(ii)] $R^{(j)}_V(z)$ is uniformly continuous on $\widetilde{I}$ in the norm topology of $\B(s,-s)$.
\vskip0.1cm
\item[(iii)] For $\mu\in I$, the norm limits
\begin{equation*}
     \frac{d^{j}}{d\mu^j}(R^{\pm}_V(\mu))=\lim\limits_{\varepsilon\downarrow0}R^{(j)}_V(\mu\pm i\varepsilon)
    \end{equation*}
     exist in $\B(s,-s)$ and are uniformly norm continuous on $I$.
}
\end{itemize}
\end{proposition}
Before presenting the proof, we  outline our main steps. Firstly, based on the theory developed by Jensen, Mourre and Perry in \cite{JMP84}(see also Theorem \ref{Resol-Smoo}), we aim to identify a suitable conjugate operator $A$. This operator will enable us to establish estimates for the derivatives of the resolvent $R_V(z)$ in the space $\mscH^{A}_{s}$ (the Besov space associated with $A$, as defined in \cite[Section 3.1]{BS99}). Subsequently, we will attempt to replace the space $\mscH^{A}_{s}$ with $\ell^{2,s}$, thereby obtaining the desired results.

We now introduce the conjugate operator $A$ considered here. Define the position operator $\mcaN$ as:
$$(\mcaN\phi)(n):=n\phi(n),\quad n\in\Z,\quad \forall\ \phi\in\mcaD(\mcaN)=\Big\{\phi\in\ell^2(\Z):\sum\limits_{n\in\Z}|n|^2|\phi(n)|^2<\infty\Big\},$$
and the difference operator $\mcaP$ on $\ell^2(\Z)$ by:
$$(\mcaP\phi)(n):=\phi(n+1)-\phi(n),\quad\forall\ \phi\in\ell^2(\Z).$$
It immediately follows that the dual operator $\mcaP^{*}$ of $\mcaP$ is given by:
$$(\mcaP^{*}\phi)(n):=\phi(n-1)-\phi(n),\quad\forall\ \phi\in\ell^2(\Z).$$
Let us consider the self-adjoint operator $A$ on $\ell^2(\Z)$ satisfying
\begin{equation}\label{A}
iA=\mcaN\mcaP-\mcaP^{*}\mcaN.
\end{equation}

To apply Theorem \ref{Resol-Smoo} to our specific case, it suffices to verify two conditions: the regularity of $H$ with respect to $A$~and the  Mourre estimate of the form \eqref{Mourre}. The first condition is verified in Lemma \ref{regularity of H}, while for the second, we derive the following estimate.
\begin{lemma}\label{mourre esti lemma}
{ Let $H=\Delta^2+V$, where $|V(n)|\lesssim \left<n\right>^{-\beta}$ with $\beta>1$ and let $A$ be defined as in \eqref{A}. Then, for any $\lambda\in\mcaI$, there exist constants $\alpha>0$, $\delta>0$ and a compact operator $K$ on $\ell^2(\Z)$, such that
\begin{equation}\label{mourre estim}
E_{H}(\mcaJ)ad^{1}_{iA}(H)E_{H}(\mcaJ)\geq\alpha E_{H}(\mcaJ)+K,\quad \mcaJ=(\lambda-\delta,\lambda+\delta),
\end{equation}
where $E_{H}(\mcaJ)$ represents the spectral projection of $H$ onto the interval $\mcaJ$ and $ad^{1}_{A}(H)$ is defined in \eqref{adk}.}
\end{lemma}
We delay the proof of this lemma to the end of this section. Now, combining this lemma and Lemma \ref{regularity of H}, one can apply Theorem \ref{Resol-Smoo} to $H$ to obtain the following estimates.
\begin{lemma}\label{LAP lemma}
{Let $H=\Delta^2+V$ with $|V(n)|\lesssim \left<n\right>^{-\beta}$ for some $\beta>1$ and let $A$ be defined as in \eqref{A}. Given $\lambda\in\mcaI$ and $\mcaJ$ is defined in \eqref{mourre estim}. Then, for any relatively compact interval $I\subseteq\mcaJ\setminus\sigma_{p}(H)$, any $j\in\left\{0,\cdots,[\beta]-1\right\}$ and $s>j+\frac{1}{2}$, one has
\begin{itemize}
\item [(i)]\begin{equation}\label{jth esti of Lem}
\sup\limits_{{\Re}z\in I,{\Im}z\neq0}\left\|\left<A\right>^{-s}R^{(j)}_V(z)\left<A\right>^{-s}\right\|<\infty.
\end{equation}
\item[(ii)] Denote $\widetilde{I}=\{z:\Re z\in I,\ 0<|\Im z|\leq1\}$, then $\left<A\right>^{-s}R^{(j)}_V(z)\left<A\right>^{-s}$ is H\"{o}lder continuous on $\widetilde{I}$ with the exponent $\delta(s,j)$ defined in \eqref{delta}.
    \vskip0.1cm
\item [(iii)] Let $\mu\in I$. The norm limits
$$\lim\limits_{\varepsilon\downarrow0}\left<A\right>^{-s}R^{(j)}_V(\mu\pm i\varepsilon)\left<A\right>^{-s}$$
exist and equal
$$\frac{d^{j}}{d\mu^j}(\left<A\right>^{-s}R^{\pm}_V(\mu)\left<A\right>^{-s}),$$
where
$$\left<A\right>^{-s}R^{\pm}_V(\mu)\left<A\right>^{-s}:=\lim\limits_{\varepsilon\downarrow0}\left<A\right>^{-s}R_V(\mu\pm i\varepsilon)\left<A\right>^{-s}.$$
The norm limits are H\"{o}lder continuous with exponent $\delta(s,n)$ given by \eqref{delta}.
\end{itemize}
}
\end{lemma}

With Lemma \ref{LAP lemma} established, we now proceed to prove Proposition \ref{LAP1}.
\begin{proof}[Proof of Proposition \ref{LAP1}]
%We will present the proof specially for the case when $j=0$, and the other cases can be handled analogously.

First of all, by virtue of the resolvent identity
\begin{equation}\label{reso iden}
R_V(z)=R_V(z_0)+(z-z_0)(R_V(z_0))^2+(z-z_0)^2R_V(z_0)R_V(z)R_V(z_0),%\quad \forall\ z,z_0\in\rho(H),
\end{equation}
For any $\lambda\in\mcaI$, any relatively compact interval $I\subseteq\mcaJ\setminus\sigma_{p}(H)$ and $z\in\widetilde{I}$. Taking $z_0=-i$ in the \eqref{reso iden}, we obtain that
 \begin{equation}\label{-i}
 R_V(z)=R_V(-i)+(z+i)(R_V(-i))^2+(z+i)^2R_V(-i)R_V(z)R_V(-i).
 \end{equation}
 \vskip0.3cm
 \underline{$\bm{Case\ j=0.}$}~\ (i)~Based on \eqref{-i}, to obtain the \eqref{jth reso estim}, it suffices to show that for any $\frac{1}{2}<s\leq[\beta]$, one has
 \begin{equation}\label{R(-i)R(z)R(-i)}
\sup\limits_{z\in\widetilde{I}}\left\|R_V(-i)R_V(z)R_V(-i)\right\|_{\B(s,-s)}<\infty.
 \end{equation}
 In view of the estimate \eqref{jth esti of Lem} and noting that
 \begin{equation}\label{recom}
 \left<\mcaN\right>^{-s}R_V(-i)R_V(z)R_V(-i)\left<\mcaN\right>^{-s}=\underbrace{\left<\mcaN\right>^{-s}R_V(-i)\left<A\right>^{s}}\underbrace{\left<A\right>^{-s}R_V(z)\left<A\right>^{-s}}\underbrace{\left<A\right>^{s}R_V(-i)\left<\mcaN\right>^{-s}},
 \end{equation}
 to establish \eqref{R(-i)R(z)R(-i)}, by duality, it is enough to prove that
 \begin{equation*}
 \left<A\right>^{s}R_V(\pm i)\left<\mcaN\right>^{-s}\in\B(0,0),\quad \frac{1}{2}<s\leq[\beta].
 \end{equation*}
 In fact, this result holds for $0\leq s\leq[\beta]$. To see this, note that it's trivial for $s=0$. Next, we will demonstrate that
 \begin{equation}\label{beta}
 \left<A\right>^{[\beta]}R_V(\pm i)\left<\mcaN\right>^{-[\beta]}\in\B(0,0),
 \end{equation}
 and then, by complex interpolation, the desired result follows. Furthermore, the proof of \eqref{beta} can be reduced to verifying that
 \begin{equation}\label{ell}
 A^{\ell}R_V(\pm i)\left<\mcaN\right>^{-[\beta]}\in\B(0,0),\quad\forall\ 1\leq\ell\leq[\beta],\ \ell\in N^+.
 \end{equation}

 Indeed, for any $1\leq\ell\leq[\beta]$, we use the formula
 \begin{equation}\label{commu of A and R}
 ad^1_{A}(R_V(\pm i))=R_V(\pm i)ad^{1}_{A}(H)R_V(\pm i),
 \end{equation}
where $ad^1_{A}(\cdot)$ is defined in \eqref{adk}.
%With the goal to put the power of $A$ and $\left<\mcaN\right>^{-[\beta]}$ together in mind,
With the goal of combining the powers of $A$ and $\left<\mcaN\right>^{-[\beta]}$, we repeatedly apply the formula
$$AR_V(\pm i)=ad^1_{A}(R_V(\pm i))+R_V(\pm i)A.$$
 This allows us to express $A^{\ell}R_V(\pm i)\left<\mcaN\right>^{-[\beta]}$ as a finite sum of operators of the form
 $$B_{k}A^{k}\left<\mcaN\right>^{-[\beta]},\quad 0\leq k\leq\ell,$$
 where $B_{k}\in\B(0,0)$, and if it contains such term $ad^{q}_{A}(H)$, then $q$ is at most $\ell$. Since $k\leq[\beta]$, it follows that $A^{k}\left<\mcaN\right>^{-[\beta]}\in\B(0,0)$, which proves the \eqref{ell}. Therefore, the desired result (i) is established.
\vskip0.2cm
 Furthermore, (ii) and (iii) follow directly from the corresponding results in Lemma \ref{LAP lemma} and the relations \eqref{-i} and \eqref{recom}.%, correspondingly. Thus, this completes the proof.
 \vskip0.3cm
 \underline{$\bm{ Case\ j\geq1.}$}\ For $j\geq1$, likewise, the key step is to prove (i). Since $R^{(j)}_V(z)=C_j(R_V(z))^{j}$, where $C_j$ is a constant depending on $j$, we can use \eqref{-i} and the commutative property $R_V(z)R_{V}(-i)=R_{V}(-i)R_V(z)$ to express $R^{(j)}_V(z)$ as follows:
 \begin{align*}
 R^{(j)}_V(z)=\sum\limits_{k_1+k_2+k_3=j}C(k_1,k_2,k_3,j)(z+i)^{k_2+2k_3}(R_V(-i))^{k_1+2k_2+2k_3}R^{(k_3)}_V(z).
 \end{align*}
For each term in the sum above, following in a similar approach to the case $j=0$, we can also establish the \eqref{jth reso estim}. This completes the proof.
\end{proof}
%\vskip0.4cm
Finally, we give the proof of Lemma \ref{mourre esti lemma}.
\begin{proof}[Proof of Lemma \ref{mourre esti lemma}]
For convenience, in this proof, we denote $H_0:=\Delta^2$ and replace the notation $ad^{1}_{A}(\cdot)$ with $[\cdot,A]$.

For any $\lambda\in\mcaI=(0,16)$, to obtain \eqref{mourre estim}, the key step is to prove that it holds for $H_0$ with $K=0$. Specifically, we need to show that there exist constants
$\alpha>0$ and $\delta>0$, such that
\begin{equation}\label{H0}
E_{H_0}(\mcaJ)[H_0,iA]E_{H_0}(\mcaJ)\geq\alpha E_{H_0}(\mcaJ),\quad \mcaJ=(\lambda-\delta,\lambda+\delta).
\end{equation}

Once \eqref{H0} is established, after some deformation treatment, we have
\begin{align*}
&E_{H}(\mcaJ)[H,iA]E_{H}(\mcaJ)=E_{H_0}(\mcaJ)[H_0,iA]E_{H_0}(\mcaJ)+\\
&\underbrace{E_{H_0}(\mcaJ)[H_0,iA](E_{H}(\mcaJ)-E_{H_0}(\mcaJ))+(E_{H}(\mcaJ)-E_{H_0}(\mcaJ))[H_0,iA]E_{H}(\mcaJ)+E_{H}(\mcaJ)[V,iA]E_{H}(\mcaJ)}_{K_1}\\
&\geq \alpha E_{H_0}(\mcaJ)+K_1=\alpha E_{H}(\mcaJ)+\underbrace{\alpha( E_{H}(\mcaJ)-E_{H_0}(\mcaJ))+K_1}_{K},
\end{align*}
where the compactness $K$ follows from the fact that both $V$ and $[V,iA]$ are bounded compact operators under the assumption $|V(n)|\lesssim\left<n\right>^{-\beta}$ with $\beta>1$. This establishes \eqref{mourre estim}.

In what follows, we focus on proving \eqref{H0}. Indeed, for any $\lambda\in(0,16)$, take $0<\delta<\frac{1}{2}\min(\lambda, 16-\lambda):=\delta_0=\delta_0(\lambda)$. By Lemma \ref{regularity of H},
$$[H_0,iA]=2H_0(4-\sqrt {H_{0}}),\quad 0\leq H_0\leq16.$$
Define $g(x)=2x(4-\sqrt x)$ for $x\in[0,16]$. Then, $C(\lambda):=\min\limits_{x\in \mcaJ_1}g(x)>0$, where $\mcaJ_1=[\lambda-\delta_0,\lambda+\delta_0]$. Using functional calculus, we obtain
 \begin{equation}\label{J1}
 E_{H_0}(\mcaJ_1)[H_0,iA]E_{H_0}(\mcaJ_1)\geq C(\lambda) E_{H_0}(\mcaJ_1):=\alpha E_{H_0}(\mcaJ_1).
 \end{equation}
 Now, take $\mcaJ=(\lambda-\delta,\lambda+\delta)\subseteq\mcaJ_1$, and multiply both sides of \eqref{J1} by $E_{H_0}(\mcaJ)$. This yields the desired inequality \eqref{H0}.
\end{proof}
\section{Proof of Theorem \ref{asy theor of Mpm mu}}\label{proof of asy}
This section is dedicated to presenting the proof of asymptotic expansions of $\left(M^{\pm}\left(\mu^4\right)\right)^{-1}$.
To begin with, we come to  characterize the regular condition given  in Definition \ref{defin of regular point 1}. 

Recall that $U(n)={\rm sign} (V(n))$, $v(n)=\sqrt{|V(n)|}$. Define
$$T_{0}=U+vG_{0}v,\quad\widetilde{T}_{0}=U+v\widetilde{G}_{0}v,$$
where $G_0$ and $\widetilde{G}_0$ are integral operators with the following kernels, respectively:
\begin{align}
G_0(n,m)&=\frac{1}{12}\left(|n-m|^3-|n-m|\right),\label{kernel of G0}\\
\widetilde{G}_0(n,m)&=\frac{(-1)^{|n-m|}}{32\sqrt2}\left(2\sqrt2 |n-m|-\left(2\sqrt2-3\right)^{|n-m|}\right).\label{kernel of G0tuta}
\end{align}
Additionally, recall that $S_0$ and $\widetilde{Q}$ are the orthogonal projections onto the following spaces:
$$S_0\ell^2(\Z)=\left({\rm span}\{v,v_1\}\right)^{\bot},\quad \widetilde{Q}\ell^2(\Z)=\left({\rm span}\{\tilde{v}\}\right)^{\bot},\ v_1(n)=nv(n),\ \tilde{v}(n)=(-1)^{n}v(n).$$
Denote
\begin{align}\label{S,Stuta}
\begin{split}
S:&={\rm Ker}S_0T_0S_0\mspace{-5mu}\mid_{S_0\ell^2(\Z)}=\big\{f\in S_0\ell^{2}(\Z):S_0T_0f=0\big\},\\
\widetilde{S}:&={\rm Ker}\widetilde{Q}\widetilde{T}_0\widetilde{Q}\mspace{-5mu}\mid_{\widetilde{Q}\ell^2(\Z)}=\big\{f\in \widetilde{Q}\ell^{2}(\Z):\widetilde{Q}\widetilde{T}_0f=0\big\}.
\end{split}
\end{align}
\begin{lemma}\label{charac of regular}
{ Let $H=\Delta^2+V$ on $\Z$ and $|V(n)|\lesssim \left<n\right>^{-\beta}$ with $\beta>7$, then
\begin{itemize}
\item [(i)]$f\in S\Longleftrightarrow \exists\ \phi\in W_{\frac{3}{2}}(\Z)$ such that $H\phi=0$. Moreover, $f=Uv\phi$ and $\phi(n)=-(G_0vf)(n)+c_1n+c_2$,
    where
    \begin{equation}\label{exp of c1,c2}
    c_1=\frac{\left<T_0f,v'\right>}{\|v'\|^2_{\ell^2}},\quad c_2=\frac{\left<T_0f,v\right>}{\|V\|_{\ell^1}}-\frac{\left<v_1,v\right>}{\|V\|_{\ell^1}}c_1,\quad v'=v_1-\frac{\left<v_1,v\right>}{\|V\|_{\ell^1}}v.
    \end{equation}
\item [(ii)] $f\in \widetilde{S}\Longleftrightarrow \exists\ \phi\in W_{\frac{1}{2}}(\Z)$ such that $H\phi=16\phi$. Moreover, $f=Uv\phi$ and $\phi(n)=-(\widetilde{G}_{0}vf)(n)+(-1)^{n}c$, where
    \begin{equation}\label{exp of c}
    c=\frac{\left<\widetilde{T}_0f,\tilde{v}\right>}{\|V\|_{\ell^1}}.
    \end{equation}
\end{itemize}
}
\end{lemma}
\begin{remark}\label{remar of charac of regular}
{\rm Under the assumption of Theorem \ref{asy theor of Mpm mu}, as a consequence of Lemma \ref{charac of regular}, it follows that
\begin{align*}
    \begin{split}
0\ {\rm is\  a\  regular\  point\  of}\  H\ &\Leftrightarrow\  S=\{0\}\Leftrightarrow S_0T_0S_0\ {\rm is\  invertible\  in}\  S_0\ell^2(\Z).\\
 16\ {\rm is\  a\  regular\  point\  of} \ H\ &\Leftrightarrow\ \widetilde{S}=\{0\}\Leftrightarrow\widetilde{Q}\widetilde{T}_{0}\widetilde{Q}\ {\rm is\  invertible\  in}\ \widetilde{Q}\ell^2(\Z).
 \end{split}
 \end{align*}
}
\end{remark}
\begin{proof}[Proof of Lemma \ref{charac of regular}]
{\textbf{\underline{(i)}}}~``$\bm{\Longrightarrow}$"  Let $f\in S$. Then $f\in S_0\ell^2(\Z)$ and $S_0T_0f=0$. Denote by $P_0$ the orthogonal projection onto span$\{v,v_1\}$. Then $S_0=I-P_0$, and it follows that
\begin{equation}\label{Uf of S}
Uf=-vG_0vf+P_0T_0f.
\end{equation}
Let
\begin{equation}\label{v'}
v'=v_1-\frac{\left<v_1,v\right>}{\|V\|_{\ell^1}}v,
\end{equation}
so that $\{v',v\}$ forms an orthogonal basis for span$\{v,v_1\}$. In this case, we have
\begin{equation}\label{P0T0f 1}
P_0T_0f=\frac{\left<P_0T_0f,v\right>}{\|V\|_{\ell^1}}v+\frac{\left<P_0T_0f,v'\right>}{\|v'\|^2_{\ell^2}}v'=\frac{\left<T_0f,v\right>}{\|V\|_{\ell^1}}v+\frac{\left<T_0f,v'\right>}{\|v'\|^2_{\ell^2}}v'.
\end{equation}
Substituting \eqref{v'} into the second equality of \eqref{P0T0f 1}, we further obtain that
\begin{align}\label{P0T0f2}
\begin{split}
P_0T_0f&=\frac{\left<T_0f,v\right>}{\|V\|_{\ell^1}}v+\frac{\left<T_0f,v'\right>}{\|v'\|^2_{\ell^2}}\left(v_1-\frac{\left<v_1,v\right>}{\|V\|_{\ell^1}}v\right)\\
&=\frac{\left<T_0f,v'\right>}{\|v'\|^2_{\ell^2}}v_1+\left(\frac{\left<T_0f,v\right>}{\|V\|_{\ell^1}}-\frac{\left<T_0f,v'\right>\left<v_1,v\right>}{\|v'\|^2_{\ell^2}\|V\|_{\ell^1}}\right)v\\
&:=c_1v_1+c_2v.
\end{split}
\end{align}
Multiplying both sides of \eqref{Uf of S} by $U$ and substituting $P_0T_0f$ with \eqref{P0T0f2}, then
 $$f=-UvG_0vf+U(c_1v_1+c_2v)=Uv(-G_0vf+c_1n+c_2):=Uv\phi.$$

 Firstly, $\phi=-G_0vf+c_1n+c_2\in W_{\frac{3}{2}}(\Z)$. Considering that $|c_1n+c_2|\lesssim1+|n|\in W_{\frac{3}{2}}(\Z)$, it suffices to verify that $G_0vf\in W_{\frac{3}{2}}(\Z)$. Indeed, by \eqref{kernel of G0}, $\left<f,v\right>=0$ and $\left<f,v_1\right>=0$, it follows that
 \begin{align*}
 12(G_0vf)(n)&=\sum\limits_{m\in\Z}^{}\left(|n-m|^3-|n-m|\right)v(m)f(m)\\
 &=\sum\limits_{m\in\Z}^{}\left(|n-m|^3-|n-m|-n^2|n|+3|n|nm\right)v(m)f(m)\\
 &:=\sum\limits_{m\in\Z}^{}K(n,m)v(m)f(m).
 \end{align*}
 We decompose $K(n,m)$ into three parts:
 \begin{align*}
 K(n,m)&=|n-m|(n^2-2nm+m^2-1)-n^2|n|+3|n|nm\\
 &=\left(n^2(|n-m|-|n|)+|n|nm\right)-2n(|n-m|-|n|)m+(m^2-1)|n-m|\\
 &:=K_1(n,m)-K_2(n,m)+K_3(n,m).
 \end{align*}
 For $K_1(n,m)$, if $n\neq m$, then
 \begin{align*}
 |K_1(n,m)|&=\left|\frac{\left(n^2(|n-m|-|n|)+|n|nm\right)(|n-m|+|n|)}{|n-m|+|n|}\right|\\
 &=\left|\frac{n^2m^2+|n|nm(|n-m|-|n|)}{|n-m|+|n|}\right|\leq 2|n|m^2.
 \end{align*}
 Since $K_1(n,n)=0$, we always have $|K_1(n,m)|\leq 2|n|m^2$.
 As for $K_2(n,m),K_3(n,m)$, by the triangle inequality, it yields that
 $$|K_2(n,m)|\leq 2|n|m^2,\quad |K_3(n,m)|\leq (1+|n|)|m|^3.$$
 In summary, one obtains that $|K(n,m)|\lesssim (1+|n|)|m|^3 $. Thus, in view that $\beta>7$,
  $$|(G_0vf)(n)|\lesssim \sum\limits_{m\in\Z}^{}|K(n,m)||v(m)f(m)|\lesssim \left<n\right>\sum\limits_{m\in\Z}^{}\left<m\right>^3|v(m)||f(m)|\lesssim \left<n\right>\in W_{\frac{3}{2}}(\Z). $$
  Consequently, we conclude that $\phi\in W_{\frac{3}{2}}(\Z)$.

  Next, we show that $H\phi=0$. Notice that $\Delta^2G_0vf=vf$ and $vf=vUv\phi=V\phi$, it yields that
 $$H\phi=(\Delta^2+V)\phi=-\Delta^2G_0vf+V\phi=-vf+vf=0.$$
\vskip0.3cm
``$\bm{\Longleftarrow}$" Suppose that $\phi\in W_{\frac{3}{2}}(\Z)$ and satisfies $H\phi=0$. %, i.e., $(\Delta^2+V)\phi=0$.
 Let $f=Uv\phi$. We will show that $f\in S$ and that $\phi(n)=-(G_0vf)(n)+c_1n+c_2$, where $c_1,c_2$ are defined in \eqref{exp of c1,c2}.

\textbf{On one hand}, $f\in S_0\ell^2(\Z)$, i.e., for $k=0,1$, it can be verified that
$$\left<f,v_k\right>=\sum_{n\in\Z}(Uv\phi)(n)n^kv(n)=\sum_{n\in\Z}n^kV(n)\phi(n)=0.$$
In fact, take $\eta(x)\in C^{\infty}_0(\R)$ such that $\eta(x)=1$ for $|x|\leq1$ and $\eta(x)=0$ for $|x|>2$. For $k=0,1$ and any $\delta>0$, define
%$$\left<f,v_k\right>=\sum_{n\in\Z}(Uv\phi)(n)n^kv(n)=\sum_{n\in\Z}n^kV(n)\phi(n).$$
%Note that
%$$(\Delta^2\phi)(n)=\phi(n+2)-4\phi(n+1)+6\phi(n)-4\phi(n-1)+\phi(n-2),$$
$$F(\delta)=\sum_{n\in\Z}n^kV(n)\phi(n)\eta(\delta n).$$
For one thing, under the assumptions on $V$ and $\phi\in W_{\frac{3}{2}}(\Z)$, it follows from Lebesgue's dominated convergence theorem that
 $$\left<f,v_k\right>=\lim_{\delta\rightarrow0}F(\delta).$$
 \begin{comment}
\begin{align*}
\left<f,v\right>&=\sum\limits_{n\in\Z}^{}V(n)\phi(n)=-\sum\limits_{n\in\Z}^{}(\Delta^2\phi)(n)=0,\\
\left<f,v_1\right>&=\sum\limits_{n\in\Z}^{}nV(n)\phi(n)=-\sum\limits_{n\in\Z}^{}n(\Delta^2\phi)(n)=0.
\end{align*}
\end{comment}
For another, for any $\delta>0$, using the relation $V(n)\phi(n)=-(\Delta^2\phi)(n)$ and $\eta\in C^{\infty}_{0}(\R)$, we have
$$F(\delta)=-\sum_{n\in\Z}(\Delta^2\phi)(n)n^k\eta(\delta n)=-\sum_{n\in\Z}\phi(n)(\Delta^2G_{\delta,k})(n),$$
where $G_{\delta,k}(x)=x^k\eta(\delta x)$. Next we prove that %Fix $\frac{3}{2}<s<\frac{5}{2}$,
 for any $0<\delta<\frac{1}{3}$, $s>0$,
 \begin{equation}\label{estimate of Delta2G}
 \|\left<\cdot\right>^s(\Delta^2G_{\delta,k})(\cdot)\|_{\ell^2(\Z)}\leq C(k,s,\eta)\delta^{\frac{7}{2}-k-s},\quad k=0,1,
 \end{equation}
 where $C(k,s,\eta)$ is a constant depending on $k,s$ and $\eta$. Once this estimate is established, taking $\frac{3}{2}<s<\frac{5}{2}$, utilizing $\phi\in W_{\frac{3}{2}}(\Z)$ and H\"{o}lder's inequality, we obtain
 $$|F(\delta)|\leq C(k,s,\eta)\delta^{\frac{7}{2}-k-s}\|\left<\cdot\right>^{-s}\phi(\cdot)\|_{\ell^2(\Z)},\quad k=0,1.$$
 This implies $\lim\limits_{\delta\rightarrow0}F(\delta)=0$, which proves that $f\in S_0\ell^2(\Z)$.
 To derive \eqref{estimate of Delta2G}, we first apply the differential mean value theorem to get
 $$(\Delta^2G_{\delta,k})(n)=G^{(4)}_{\delta,k}(n-1+\Theta),$$
for some $\Theta\in[0,4]$. By Leibniz's derivative rule and the definition of $G_{\delta,k}$, one has
\begin{equation}\label{derive of G}
\left|(\Delta^2G_{\delta,k})(n)\right|=\left|G^{(4)}_{\delta,k}(n-1+\Theta)\right|\leq C_k\delta^{4-k}\sum\limits_{\ell=0}^{k}\left|\eta^{(4-\ell)}(\delta(n-1+\Theta))\right|.
\end{equation}
 Note that supp$(\eta^{(\ell)})\subseteq\{x:1\leq|x|\leq2\}$ for any $\ell\in\N^{+}$, then for any $s>0$ and $0<\delta<\frac{1}{3}$, the following estimate holds:
\begin{align*}
\|\left<\cdot\right>^s\eta^{(\ell)}(\delta(\cdot-1+\Theta))\|^2_{\ell^2(\Z)}\leq C(s,\eta)\sum\limits_{|n|\leq\frac{3}{\delta}}^{}|n|^{2s}\leq C'(s,\eta)\delta^{-2s-1},
\end{align*}
which gives the desired \eqref{estimate of Delta2G} by combining \eqref{derive of G} with the triangle inequality.

\textbf{On the other hand}, we first show that $\phi(n)=-(G_0vf)(n)+c_1n+c_2$, from which it follows that
$$S_0T_0f=S_0(U+vG_0v)f=S_0v\phi+S_0vG_0vf=S_0v\phi+S_0v(-\phi +c_1n+c_2)=S_0v\phi-S_0v\phi=0.$$
To see this, since $f\in S_0\ell^2(\Z)$ and according to ``$\bm{\Longrightarrow}$", $G_0vf\in W_{\frac{3}{2}}(\Z)$, then $\tilde{\phi}:=\phi+G_0vf\in W_{\frac{3}{2}}(\Z)$ and $\Delta^2\tilde{\phi}=H\phi=0$,
%$$\Delta^2\tilde{\phi}=\Delta^{2}(\phi+G_0vf)=\Delta^2\phi+vf=\Delta^2\phi+V\phi=H\phi=0,$$
 which indicates that $\tilde{\phi}=\tilde{c}_{1}n+\tilde{c}_{2}$ for some constants $\tilde{c}_1$ and $\tilde{c}_2$. Next we determine that $\tilde{c}_1=c_1,\tilde{c}_2=c_2$. Indeed, since
$$0=H\phi=(\Delta^2+V)\phi=-vf-VG_0vf+V(\tilde{c}_{1}n+\tilde{c}_{2})=U(-vT_0f+\tilde{c}_1nv^2+\tilde{c}_2v^2),$$
then $\tilde{c}_1v_1+\tilde{c}_2v=T_0f$. Based on this, we further have
\begin{align}
\tilde{c_1}\left<v_1,v\right>+\tilde{c}_2\left<v,v\right>&=\left<T_0f,v\right> \label{eq1},\\
\tilde{c_1}\left<v_1,v_1\right>+\tilde{c}_2\left<v,v_1\right>&=\left<T_0f,v_1\right>,\label{eq2}
\end{align}
and combine that $v_1=v^{'}+\frac{\left<v_1,v\right>}{\|V\|_{\ell^1}}v$ and $\big<v^{'},v\big>=0$, it follows that
$$ \tilde{c}_1=\frac{\left<T_0f,v'\right>}{\|v'\|^2_{\ell^2}},\quad \tilde{c}_2=\frac{\left<T_0f,v\right>}{\|V\|_{\ell^1}}-\frac{\left<v_1,v\right>}{\|V\|_{\ell^1}}\tilde{c}_1.$$
Therefore, $f\in S$ and (i) is derived.
\vskip0.3cm
{\textbf{\underline{(ii)}}}~``$\bm{\Longrightarrow}$" Assume that $f\in \widetilde{S}$. Then $f\in \widetilde{Q}\ell^2(\Z)$ and $\widetilde{Q}\widetilde{T}_{0}f=0$. Recall that $\widetilde{P}=\left\|V\right\|^{-1}_{\ell^1}\left<\cdot,\tilde{v}\right>\tilde{v}$ and thus $\widetilde{Q}=I-\widetilde{P}$. Then
\begin{equation}\label{Uftuta}
Uf=-v\widetilde{G}_0vf+\widetilde{P}\widetilde{T}_0f=-v\widetilde{G}_0vf+\left\|V\right\|^{-1}_{\ell^1}\left<\widetilde{T}_0f,\tilde{v}\right>\tilde{v}:=-v\widetilde{G}_0vf+c\tilde{v}.
\end{equation}
Multiplying $U$ from both sides of \eqref{Uftuta}, one obtains that
$$f=-Uv\widetilde{G}_0vf+cU\tilde{v}=Uv(-\widetilde{G}_0vf+Jc):=Uv\phi,$$
where $(Jc)(n)=(-1)^{n}c$. %By virtue of \eqref{kernel of G0tuta}, we can directly calculate that $(\Delta^2-16)\widetilde{G}_0=\widetilde{G}_0(\Delta^2-16)=I$ and $(\Delta^2-16)(Jc)=0$. Therefore,
%$$(H-16)\phi=(\Delta^2-16+V)\phi=(\Delta^2-16)(-\widetilde{G}_0vf+Jc)+V\phi=-vf+vf=0.$$

Firstly, we prove that $\phi=-\widetilde{G}_0vf+Jc\in W_{\frac{1}{2}}(\Z)$. It is enough to show that $\widetilde{G}_0vf\in W_{\frac{1}{2}}(\Z)$. By \eqref{kernel of G0tuta},
\begin{align*}
32\sqrt2(\widetilde{G}_0vf)(n)&=\sum\limits_{m\in\Z}^{}\left(2\sqrt2 |n-m|-\left(2\sqrt2-3\right)^{|n-m|}\right)(-1)^{|n-m|}v(m)f(m)\\
&=\sum\limits_{m\in\Z}^{}(-1)^{n}\left(2\sqrt2 |n-m|-\left(2\sqrt2-3\right)^{|n-m|}\right)\tilde{v}(m)f(m)\\
&=\sum\limits_{m\in\Z}^{}(-1)^{n}\left(2\sqrt2 (|n-m|-|n|)-\left(2\sqrt2-3\right)^{|n-m|}\right)\tilde{v}(m)f(m)
\end{align*}
where we used the facts that ${(-1)^{|n-m|}=(-1)^{n+m}}$ in the second equality and $\left<f,\tilde{v}\right>$=0 in the third equality, respectively. Since $0<3-2\sqrt2<1$ and by the triangle equality, we have
\begin{align*}
\left|(\widetilde{G}_0vf)(n)\right|\lesssim \sum\limits_{m\in\Z}(1+|m|)|\tilde{v}(m)f(m)|\lesssim1\in W_{\frac{1}{2}}(\Z).
\end{align*}
Hence, $\phi\in W_{\frac{1}{2}}(\Z)$. Moreover, note that $(\Delta^2-16)\widetilde{G}_0vf=vf,(\Delta^2-16)(Jc)=0$ and $vf=V\phi$, then
$$(H-16)\phi=(\Delta^2-16+V)\phi=(\Delta^2-16)(-\widetilde{G}_0vf+Jc)+V\phi=-vf+vf=0.$$
\vskip0.3cm
``$\bm{\Longleftarrow}$" Given that $\phi\in W_{\frac{1}{2}}(\Z)$ and satisfies $H\phi=16\phi$. %i.e., $(\Delta^2-16+V)\phi=0$.
Let $f=Uv\phi$, then $f\in \widetilde{S}$ and $\phi(n)=-(\widetilde{G}_0vf)(n)+(-1)^nc$, where $c$ is defined in \eqref{exp of c}. Indeed, let $\eta$ be as in \textbf{(i)}. For any $\delta>0$, define
$$\widetilde{F}(\delta)=\sum\limits_{n\in\Z}(JV\phi)(n)\eta(\delta n).$$
%then $\left<f,\tilde{v}\right>=\lim\limits_{\delta\rightarrow0}\widetilde{F}(\delta)$.
Noting that $V(n)\phi(n)=-[(\Delta^2-16)\phi](n)$ and $J\Delta J=-\Delta-4$, %and $\eta\in C^{\infty}_0(\R)$, for any $\delta>0$, one has
we can apply the same method as in part \textbf{(i)} to obtain that
$$\left<f,\tilde{v}\right>=\lim\limits_{\delta\rightarrow0}\widetilde{F}(\delta)=-\lim\limits_{\delta\rightarrow0}\sum\limits (J\phi)(n)[(\Delta^2+8\Delta)(\eta(\delta\cdot))](n)=0.$$
%$$\widetilde{F}(\delta)=-\sum\limits_{n\in\Z}[(\Delta^2+8\Delta)J\phi]\eta(\delta n)=-\sum\limits\phi(n)[(\Delta^2+8\Delta)(\eta(\delta\cdot))](n).$$
%Applying the same method in (i) ``$\Longleftarrow$", similarly, we can obtain that $\lim\limits_{\delta\rightarrow0}\widetilde{F}(\delta)=0$ and this proves that $\left<f,\tilde{v}\right>=0$.
Finally, it is key to show that $\phi(n)=-(\widetilde{G}_0vf)(n)+(-1)^nc$. Once this is established, then
\begin{comment}
$$\left<f,\tilde{v}\right>=\sum\limits_{n\in\Z}^{}(JV\phi)(n)=-\sum\limits_{n\in\Z}^{}(J(\Delta^2-16)\phi)(n)=-\sum\limits_{n\in\Z}^{}(\Delta^2+8\Delta)(J\phi)(n)=0,$$
where for the third equality, we used the fact that $J\Delta J=-\Delta-4$ and $J^2=1$.
\end{comment}
$$\widetilde{Q}\widetilde{T}_{0}f=\widetilde{Q}(U+v\widetilde{G}_{0}v)f=\widetilde{Q}v\phi+\widetilde{Q}\tilde{v}J\widetilde{G}_0vf=\widetilde{Q}v\phi+\widetilde{Q}\tilde{v}J(-\phi+Jc)=0.$$
Therefore, $f\in \widetilde{S}$ and (ii) is proved. To see this, let
$\tilde{\phi}=\phi+\widetilde{G}_0vf$. By a similar argument as in \textbf{(i)}, we have $\tilde{\phi}\in W_{
\frac{1}{2}}(\Z)$ and $(\Delta^2-16)\tilde{\phi}=0$, which is equivalent to $(\Delta^2+8\Delta)J\tilde{\phi}=0$. This implies that $J\tilde{\phi}=\tilde{c}$ for some constant $\tilde{c}$. Moreover, using the condition $H\phi=16\phi$, one can obtain that $\tilde{c}\tilde{v}=\widetilde{T}_0f$. Thus, $\tilde{c}=\frac{\left<\widetilde{T}_0f,\tilde{v}\right>}{\|V\|_{\ell^1}}$.
\end{proof}
To establish Theorem \ref{asy theor of Mpm mu}, we will frequently utilize the following lemma.
\begin{lemma}\label{lemm of expa}
{ \cite[Lemma 2.1]{JN01} Let $\mscH$ be a complex Hilbert space. Let $A$ be a closed operator and $S$ a projection. Suppose $A+S$ has a bounded inverse. Then $A$ has a bounded inverse if and only if
$$a\equiv S-S(A+S)^{-1}S$$
has a bounded inverse in $S\mscH$, and in this case
$$A^{-1}=(A+S)^{-1}+(A+S)^{-1}Sa^{-1}S(A+S)^{-1}.$$
}
\end{lemma}
\vskip0.2cm
\begin{proof}[\bf{Proof of Theorem \ref{asy theor of Mpm mu}}]
\textbf{(i)} Suppose that $0$ is a regular point of $H$ and $\beta>15$. Then by Remark \ref{remar of charac of regular}, $S_0T_0S_0$ is invertible in $S_0\ell^2(\Z)$.
%For brevity, we only focus on proving the \eqref{asy expan on 0} and the \eqref{asy expan on 2} follows a similar way. So in what follows, we always assume that zero is a regular point of $H$ and $\beta>15$.

Firstly, taking $N=3$ in the formula \eqref{Puiseux expan of R0 0}, namely, as $s>\frac{15}{2}$, we have
\begin{equation}\label{Pui seux of R0 take 3}
R^{\pm}_0\left(\mu^4\right)=\mu^{-3}G^{\pm}_{-3}+\mu^{-1}G^{\pm}_{-1}+G^{\pm}_0+\mu G^{\pm}_1+\mu^2G^{\pm}_{2}+\mu^3G^{\pm}_{3}+\Gamma_4(\mu)\ \ {\rm in}\ \B(s,-s),\ \mu \rightarrow0.
\end{equation}
Since $\beta>15$ and $M^{\pm}(\mu)=U+vR^{\pm}_0(\mu^4)v$, we obtain the following relation on $\ell^2(\Z)$ as $\mu\rightarrow0$,
 \begin{align*}
 M^{\pm}(\mu)=\mu^{-3}vG^{\pm}_{-3}v+\mu^{-1}vG^{\pm}_{-1}v+(U+vG^{\pm}_0v)
 +\mu vG^{\pm}_1v+\mu^2vG^{\pm}_{2}v+\mu^3vG^{\pm}_{3}v+\Gamma_4(\mu).
 \end{align*}
 Noticing that $vG^{\pm}_{-3}v=a^{\pm}P$ with $a^{\pm}=\frac{-1\pm i}{4}\|V\|_{\ell^1}$, we extract the factor $a^{\pm}\mu^{-3}$, then it can be further written as
\begin{equation}\label{M pm}
M^{\pm}\left(\mu\right)=\frac{a^{\pm}}{\mu^3}\widetilde{M}^{\pm}(\mu),
\end{equation}
where
\begin{equation}\label{M tuta}
\widetilde{M}^{\pm}(\mu)=P+\frac{1}{a^{\pm}}\mu^2vG^{\pm}_{-1}v+\frac{1}{a^{\pm}}\mu^3T_0+\frac{1}{a^{\pm}}\mu^4vG^{\pm}_{1}v+\frac{1}{a^{\pm}}\mu^5vG^{\pm}_{2}v+\frac{1}{a^{\pm}}\mu^6vG^{\pm}_{3}v+\Gamma_7(\mu).
\end{equation}
Then as $\mu\rightarrow0$, the invertibility of $M^{\pm}\left(\mu\right)$ on $\ell^{2}(\Z)$ reduces to that of $\widetilde{M}^{\pm}(\mu)$, and in this case, they satisfy the following relation:
\begin{align}\label{Relat of M and M tuta} \left(M^{\pm}\left(\mu\right)\right)^{-1}=\frac{\mu^3}{a^{\pm}}\left(\widetilde{M}^{\pm}(\mu)\right)^{-1}.
\end{align}

\underline{\textbf{Step 1:}}\  By Lemma \ref{lemm of expa}, $\widetilde{M}^{\pm}(\mu)$ is invertible on $\ell^{2}(\Z)\Leftrightarrow M^{\pm}_1(\mu):=Q-Q\left(\widetilde{M}^{\pm}(\mu)+Q\right)^{-1}Q$ is invertible on $Q\ell^{2}(\Z)$ and in this case, one has
\begin{align}\label{Relat of M tuta and M1}
\left(\widetilde{M}^{\pm}(\mu)\right)^{-1}=\left(\widetilde{M}^{\pm}(\mu)+Q\right)^{-1}\left[I+Q\left(M^{\pm}_1(\mu)\right)^{-1}Q\left(\widetilde{M}^{\pm}(\mu)+Q\right)^{-1}\right].
\end{align}
By Von Neumann expansion, a direct calculation yields that
\begin{equation}\label{M tuta plus Q}
 \widetilde{M}^{\pm}(\mu)+Q=I-\sum\limits_{k=1}^{5}\mu^{k+1}B^{\pm}_{k}+\Gamma_7(\mu),\ \mu\rightarrow 0,
\end{equation}
where
\begin{itemize}
\item $B^{\pm}_{1}=\frac{1}{a^{\pm}}vG^{\pm}_{-1}v$,\quad$B^{\pm}_{2}=\frac{1}{a^{\pm}}T_0$,\quad $B^{\pm}_{3}=\frac{1}{a^{\pm}}vG^{\pm}_{1}v-\left(\frac{1}{a^{\pm}}vG^{\pm}_{-1}v\right)^2$,
\vskip0.2cm
\item $B^{\pm}_{4}=-\left(\frac{1}{a^{\pm}}\right)^2\left(vG^{\pm}_{-1}vT_0+T_0vG^{\pm}_{-1}v\right)+\frac{1}{a^{\pm}}vG^{\pm}_{2}v$,
\vskip0.2cm
\item
$B^{\pm}_{5}=\frac{1}{a^{\pm}}vG^{\pm}_{3}v-\left(\frac{1}{a^{\pm}}\right)^2\left(vG^{\pm}_{-1}v\cdot vG^{\pm}_{1}v+T^2_0+vG^{\pm}_{1}v \cdot vG^{\pm}_{-1}v-\frac{1}{a^{\pm}}\left(vG^{\pm}_{-1}v\right)^3\right)$.
\begin{equation}\label{expr of Bkpm}
\end{equation}
\end{itemize}
Thus,
$$M^{\pm}_1(\mu)=Q-Q\left(\widetilde{M}^{\pm}(\mu)+Q\right)^{-1}Q=\frac{a^{\pm}_{-1}\mu^2}{a^{\pm}}\widetilde{M_1}^{\pm}(\mu):=\frac{1}{b^{\pm}}\mu^2\widetilde{M_1}^{\pm}(\mu),$$
where $a^{\pm}_{-1}=\frac{1\pm i}{4}$ and
\begin{equation}\label{M1 tuta}
\widetilde{M_1}^{\pm}(\mu)=QvG_{-1}vQ+b^{\pm}\sum\limits_{k=2}^{5}\mu^{k-1}QB^{\pm}_{k}Q+\Gamma_5(\mu),\ G_{-1}= \frac{1}{a^{\pm}_{-1}}G^{\pm}_{-1},\ \mu\rightarrow0.
\end{equation}
Furthermore, the invertibility of $M^{\pm}_1(\mu)$ on $Q\ell^{2}(\Z)$ can be reduced to that of $\widetilde{M_1}^{\pm}(\mu)$, and if so, then
\begin{equation}\label{Relat of M1 and M1 tuta}
\left(M^{\pm}_1(\mu)\right)^{-1}=\frac{b^{\pm}}{\mu^2}\left(\widetilde{M_1}^{\pm}(\mu)\right)^{-1}.
\end{equation}
However, $QvG_{-1}vQ$ is not invertible on $Q\ell^{2}(\Z)$. In fact, denote by
$${\rm Ker}QvG_{-1}vQ:=\{f\in Q\ell^{2}(\Z):QvG_{-1}vQf=0\}$$
 the kernel of $QvG_{-1}vQ$ on $Q\ell^{2}(\Z)$. Then we have the following claim.
\vskip0.2cm
\underline{\textbf{Claim:}} \ ${\rm Ker}QvG_{-1}vQ=S_0\ell^{2}(\Z)$ and $QvG_{-1}vQ+S_0$ is invertible on $Q\ell^{2}(\Z)$. We denote by $D_0:=\left(QvG_{-1}vQ+S_0\right)^{-1}$ its inverse.

Indeed, for any $f\in Q\ell^{2}(\Z)$, then $\left<f,v\right>=0$. By virtue of the expression $G_{-1}(n,m)=\frac{1}{8}-\frac{1}{2}|n-m|^2$, a direct calculation yields that
$$QvG_{-1}vQf=\left<f,v_1\right>Q(v_1).$$
Since $Q(v_1)\not\equiv0$ (otherwise $V\equiv0$), it implies that
$$g\in{\rm Ker}QvG_{-1}vQ\Leftrightarrow g\in Q\ell^{2}(\Z)\ {\rm and}\ QvG_{-1}vQg=0\Leftrightarrow\left<g,v\right>=0\ {\rm and}\ \left<g,v_1\right>=0\Leftrightarrow g\in S_0\ell^{2}(\Z).$$

To establish the invertibility of $QvG_{-1}vQ+S_0$, it suffices to show that it is both injective and surjective. For brevity, let $G:=QvG_{-1}vQ$. On one hand, assume that $\phi\in Q\ell^2(\Z)$ satisfies $(G+S_0)\phi=0$. Then $G\phi=-S_0\phi$. By the self-adjointness of $G$ and the fact that ${\rm Ker}G=S_0\ell^2(\Z)$, we have
$$\left<G\phi,G\phi\right>=\left<G\phi,-S_0\phi\right>=\left<\phi,-GS_0\phi\right>=0\ \Longrightarrow G\phi=0.$$
Consequently, $\phi=S_0\phi=-G\phi=0$.
On the other hand, for any $\varphi\in Q\ell^2(\Z)$, note that ${\rm Ran}G$ is closed, so $Q\ell^2(\Z)={\rm Ran}G\bigoplus{\rm Ker}G$. Thus $\varphi=\varphi_1+\varphi_2$, where $\varphi_1\in {\rm Ran}G$ and $\varphi_2\in {\rm Ker}G$. It follows that $$G\varphi=G\varphi_1=(G+S_0)\phi_1\in {\rm Ran}(G+S_0),$$ i.e., ${\rm Ran}G\subseteq{\rm Ran}(G+S_0)$. Moreover, ${\rm Ker}G\subseteq{\rm Ran}(G+S_0)$ is trivial. Hence, $Q\ell^2(\Z)={\rm Ran}(G+S_0)$. This proves the claim.

Therefore, based on this claim, we can continue the {\textbf {Step 2}} below by applying the Lemma \ref{lemm of expa} to $\widetilde{M_1}^{\pm}(\mu)$.
\vskip0.2cm
\underline{\textbf{Step 2:}} \ $\widetilde{M_1}^{\pm}(\mu)$ is invertible on $Q\ell^{2}(\Z)\Leftrightarrow M^{\pm}_{2}(\mu):=S_0-S_0\left(\widetilde{M_1}^{\pm}(\mu)+S_0\right)^{-1}S_0$ is invertible on $S_0\ell^{2}(\Z)$. In this case,
\begin{equation}\label{Relat of M1 tuta and M2}
\left(\widetilde{M_1}^{\pm}(\mu)\right)^{-1}=\left(\widetilde{M_1}^{\pm}(\mu)+S_0\right)^{-1}\left(I+S_0\left(M^{\pm}_{2}(\mu)\right)^{-1}S_0\left(\widetilde{M_1}^{\pm}(\mu)+S_0\right)^{-1}\right).
\end{equation}
By \eqref{M1 tuta}, Von-Neumann expansion and the relation $QD_0=D_0Q=D_0$, a direct calculation yields that
\begin{equation}\label{inverse of M1 tuta plus S0}
\left(\widetilde{M_1}^{\pm}(\mu)+S_0\right)^{-1}=D_0-\sum\limits_{k=1}^{4}\mu^{k}\widetilde{B}_{k}^{\pm}+D_0\Gamma_5(\mu)D_0,\ \mu\rightarrow0,
\end{equation}
with
\begin{itemize}
\item $\widetilde{B}_{1}^{\pm}=b^{\pm}D_0B^{\pm}_2D_0$,\quad$\widetilde{B}_{2}^{\pm}=b^{\pm}D_0B^{\pm}_3D_0-\left(b^{\pm}\right)^2\left(D_0B^{\pm}_2\right)^2D_0$,
\vskip0.2cm
\item
$\widetilde{B}_{3}^{\pm}=b^{\pm}D_0B^{\pm}_4D_0-\left(b^{\pm}\right)^2\left(D_0B^{\pm}_2D_0B^{\pm}_3D_0+D_0B^{\pm}_3D_0B^{\pm}_2D_0-b^{\pm}\left(D_0B^{\pm}_2\right)^3D_0\right),$
\vskip0.2cm
\item $\widetilde{B}_{4}^{\pm}=b^{\pm}D_0B^{\pm}_5D_0-\left(b^{\pm}\right)^2\left(D_0B^{\pm}_2D_0B^{\pm}_4D_0+\left(D_0B^{\pm}_3\right)^2D_0+D_0B^{\pm}_4D_0B^{\pm}_2D_0\right)$

    $\quad\ \ +\left(b^{\pm}\right)^3\left(\left(D_0B^{\pm}_2\right)^2D_0B^{\pm}_3D_0+D_0B^{\pm}_2D_0B^{\pm}_3D_0B^{\pm}_2D_0+D_0B^{\pm}_3\left(D_0B^{\pm}_2\right)^2D_0\right)$

$\quad\ \ -\left(b^{\pm}\right)^4\left(D_0B^{\pm}_2\right)^4D_0$.
\begin{equation}\label{expre of Bpm tuta}
\end{equation}
\end{itemize}
And then
\begin{equation}\label{M2}
M^{\pm}_2(\mu)=\frac{\mu}{a^{\pm}_{-1}}\left(S_0T_0S_0+a^{\pm}_{-1}\sum\limits_{k=2}^{4}\mu^{k-1}S_0\widetilde{B}_{k}^{\pm}S_0+S_0\Gamma_4(\mu)S_0\right)
\end{equation}
where we used the fact that $S_0D_0=D_0S_0=S_0$. By assumption, since $S_0T_0S_0$ is invertible on $S_0\ell^{2}(\Z)$, and we denote by $\widetilde{D}_{0}:=\left(S_0T_0S_0\right)^{-1}$, then based $S_0\widetilde{D}_{0}=\widetilde{D}_{0}=\widetilde{D}_{0}S_0$ and Von-Neumann expansion, it follows that
\begin{equation}\label{inverse of M2}
\left(M^{\pm}_2(\mu)\right)^{-1}=\frac{a^{\pm}_{-1}}{\mu}\widetilde{D}_{0}+C^{\pm}_0+C^{\pm}_{1}\mu+C^{\pm}_2\mu^2+\Gamma_3(\mu),\ \mu\rightarrow0
\end{equation}
with
\begin{itemize}
\item
$C^{\pm}_0=-\left(a^{\pm}_{-1}\right)^2\widetilde{D}_{0}\widetilde{B}_{2}^{\pm}\widetilde{D}_{0}$,\quad $C^{\pm}_{1}=-\left(a^{\pm}_{-1}\right)^2\left(\widetilde{D}_{0}\widetilde{B}_{3}^{\pm}\widetilde{D}_{0}-a^{\pm}_{-1}\left(\widetilde{D}_{0}\widetilde{B}_{2}^{\pm}\right)^2\widetilde{D}_{0}\right)$,
\vskip0.2cm
\item $C^{\pm}_2=-\left(a^{\pm}_{-1}\right)^2\widetilde{D}_{0}\widetilde{B}_{4}^{\pm}\widetilde{D}_{0}+\left(a^{\pm}_{-1}\right)^3\left(\widetilde{D}_{0}\widetilde{B}_{2}^{\pm}\widetilde{D}_{0}\widetilde{B}_{3}^{\pm}\widetilde{D}_{0}+\widetilde{D}_{0}\widetilde{B}_{3}^{\pm}\widetilde{D}_{0}\widetilde{B}_{2}^{\pm}\widetilde{D}_{0}\right)$

    $\quad\ \ -\left(a^{\pm}_{-1}\right)^4\left(\widetilde{D}_{0}\widetilde{B}_{2}^{\pm}\right)^3\widetilde{D}_{0}$.
\end{itemize}
\begin{equation}\label{expr of Cpm}
\end{equation}
Combining \eqref{inverse of M2},\eqref{Relat of M1 tuta and M2},\eqref{Relat of M1 and M1 tuta},\eqref{Relat of M tuta and M1} and \eqref{Relat of M and M tuta}, the desired \eqref{asy expan on 0} is obtained.
\vskip0.3cm
\textbf{(ii)} Assume that $16$ is a regular point of $H$ and $\beta>7$. Then by Remark \ref{remar of charac of regular}, $\widetilde{Q}\widetilde{T}_{0}\widetilde{Q}$ is invertible on $ \widetilde{Q}\ell^2(\Z)$. Take $N=1$ in \eqref{Puiseux expan of R0 2}, then as $s>\frac{7}{2}$, we have
\begin{equation}
R^{\pm}_{0}\left((2-\mu)^4\right)=\mu^{-\frac{1}{2}}\widetilde{G}^{\pm}_{-1}+\widetilde{G}^{\pm}_0+\mu^{\frac{1}{2}}\widetilde{G}^{\pm}_1+\Gamma_1(\mu),\ \mu\rightarrow0\ {\rm in}\ \B(s,-s),
\end{equation}
Since $\beta>7$, similarly, one can obtain that
\begin{equation}
M^{\pm}\left(2-\mu\right)=U+\tilde{v}JR^{\pm}_0((2-\mu)^4)J\tilde{v}=\mu^{-\frac{1}{2}}\tilde{v}\hat{G}^{\pm}_{-1}\tilde{v}+\widetilde{T}_0+\mu^{\frac{1}{2}}\tilde{v}\hat{G}^{\pm}_{1}\tilde{v}+\Gamma_{1}(\mu),\ 
\mu\rightarrow0,
\end{equation}
where 
$$\hat{G}^{\pm}_{j}=J\widetilde{G}^{\pm}_{j}J,\quad \widetilde{T}_0=U+\tilde{v}\hat{G}^{\pm}_{0}\tilde{v},\quad j=-1,0,1.$$
Noting that $\tilde{v}\hat{G}_{-1}\tilde{v}=d^{\pm}\widetilde{P}$ with $d^{\pm}=\frac{\pm i}{32}\|V\|_{\ell^1}$, we further obtain that 
\begin{equation*}
M^{\pm}\left(2-\mu\right)=\frac{d^{\pm}}{\mu^{\frac{1}{2}}}\left(\widetilde{P}+\frac{\mu^{\frac{1}{2}}}{d^{\pm}}\widetilde{T}_0+\frac{\mu}{d^{\pm}} v\widetilde{G}^{\pm}_{1}v+\Gamma_{\frac{3}{2}}(\mu)\right):=\frac{d^{\pm}}{\mu^{\frac{1}{2}}}\widetilde{M}^{\pm}(\mu),\ \mu\rightarrow0.
\end{equation*}
Then the invertibility of $M^{\pm}\left(2-\mu\right)$ on $\ell^{2}(\Z)$ reduces to that of $\widetilde{M}^{\pm}(\mu)$, and in this case, they satisfy the following relation
\begin{align}\label{Relat of M and M tuta 2} \left(M^{\pm}\left(2-\mu\right)\right)^{-1}=\frac{\mu^{\frac{1}{2}}}{d^{\pm}}\left(\widetilde{M}^{\pm}(\mu)\right)^{-1}.
\end{align}
Apply Lemma \ref{lemm of expa} to $\widetilde{M}^{\pm}(\mu)$, then
$$\widetilde{M}^{\pm}(\mu)\ {\rm is\  invertible\  in\
  }\ell^{2}(\Z)\Leftrightarrow M^{\pm}_1(\mu):=\widetilde{Q}-\widetilde{Q}\left(\widetilde{M}^{\pm}(\mu)+\widetilde{Q}\right)^{-1}\widetilde{Q}\ {\rm is\  invertible\  in\
  }\widetilde{Q}\ell^{2}(\Z).$$ In this case, one has
\begin{align}\label{Relat of M tuta and M1 2}
\left(\widetilde{M}^{\pm}(\mu)\right)^{-1}=\left(\widetilde{M}^{\pm}(\mu)+\widetilde{Q}\right)^{-1}\left[I+\widetilde{Q}\left(M^{\pm}_1(\mu)\right)^{-1}\widetilde{Q}\left(\widetilde{M}^{\pm}(\mu)+\widetilde{Q}\right)^{-1}\right].
\end{align}
By Von-Neumann expansion, it yields that
\begin{equation}
\left(\widetilde{M}^{\pm}(\mu)+\widetilde{Q}\right)^{-1}=I-\sum\limits_{k=1}^{2}\mu^{\frac{k}{2}}D^{\pm}_{k}+\Gamma_{\frac{3}{2}}(\mu),\quad\mu\rightarrow0,
\end{equation}
with
$$D^{\pm}_{1}=\frac{1}{d^{\pm}}\widetilde{T}_0,\quad D^{\pm}_{2}=\frac{1}{d^{\pm}}v\widetilde{G}^{\pm}_1v-\left(\frac{1}{d^{\pm}}\widetilde{T}_0\right)^2.$$
Then
\begin{equation}
 M^{\pm}_1(\mu)=\frac{\mu^{\frac{1}{2}}}{d^{\pm}}\left(\widetilde{Q}\widetilde{T}_0\widetilde{Q}+d^{\pm}\mu^{\frac{1}{2}}\widetilde{Q}D^{\pm}_{2}\widetilde{Q}+\Gamma_1(\mu)\right).
\end{equation}
Since $\widetilde{Q}\widetilde{T}_0\widetilde{Q}$ is invertible on $\widetilde{Q}\ell^{2}(\Z)$, we denote by $E_0:=\left(\widetilde{Q}\widetilde{T}_0\widetilde{Q}\right)^{-1}$. Then $E_0\widetilde{Q}=E_0=\widetilde{Q}E_0$ and by Von-Neumann expansion, one has
\begin{equation}\label{inverse of M1 2}
\left( M^{\pm}_1(\mu)\right)^{-1}=\frac{d^{\pm}}{\mu^{\frac{1}{2}}}E_0-\left(d^{\pm}\right)^2E_0D^{\pm}_2E_0+\Gamma_{\frac{1}{2}}(\mu),\quad\mu\rightarrow0.
\end{equation}
Combining the \eqref{inverse of M1 2},\eqref{Relat of M tuta and M1 2} and \eqref{Relat of M and M tuta 2}, we obtain the \eqref{asy expan on 2}. This completes the proof of Theorem \ref{asy theor of Mpm mu}.
\end{proof}
\appendix
\section{Commutator estimates and Mourre Theory}\label{section of Appendix}
This appendix is divided into two parts. First, we review the main results of \cite{JMP84}, which focus on commutator estimates for a self-adjoint operator with respect to a suitable conjugate operator. These estimates establish the smoothness of the resolvent as a function of the energy between suitable spaces. Second, we collect a set of sufficient conditions related to the regularity of bounded self-adjoint operators with respect to conjugate operators.
%Firstly, we summarize the main result from \cite{JMP84}, which deals with multiple commutator estimates for a self-adjoint operator $T$ and a suitable conjugate operator $A$. This result plays a key role in the proof of Proposition \ref{LAP1}. Secondly, we focus on the case where $T$ is bounded and collect some relevant materials concerning the regularity of $T$ with respect to $A$~(as defined in Definition \ref{def of regularity} below) for application in our specific context.

To begin, we introduce some notations and definitions for clarity and convenience.  Let $(X,\left<\cdot,\cdot\right>)$ denote a separable complex Hilbert space and $T$ be a self-adjoint operator defined on $X$ with domain $\mcaD(T)$.
\begin{itemize}
\item
Define
$$X_{+2}:=\left(\mcaD(T),\left<\cdot,\cdot\right>_{+2}\right), \quad \left<\varphi,\phi\right>_{+2}:=\left<\varphi,\phi\right>+\left<T\varphi,T\phi\right>,\quad\forall\ \varphi,\phi\in\mcaD(T),$$
\end{itemize}
and let $X_{-2}$ be the dual space of $X_{+2}$.
\begin{itemize}
\item Let $A$ be a self-adjoint operator on $X$. The sesquilinear form $[T,A]$ on $\mcaD(T)\cap\mcaD(A)$ is defined as
   \begin{equation}\label{Sesqui-form}
   [T,A](\varphi,\phi):=\left<(TA-AT)\varphi,\phi\right>,\quad \forall\ \varphi,\phi\in\mcaD(T)\cap\mcaD(A).
\end{equation}
\end{itemize}
\begin{definition}\label{condi-veri}
{\rm Let $T$ be as above and $n\geq1$ be an integer. A self-adjoint operator $A$ on $X$ is said to be conjugate to $T$ at the point $E\in\R$ and $T$ is said to be $n$-smooth with respect to $A$, if the following conditions (a)$\sim$(e) are satisfied:
\begin{itemize}
\item [(a)] $\mcaD(A)\cap\mcaD(T)$ is a core for $T$.
\item [(b)]$e^{i\theta A}$ maps $\mcaD(T)$ into $\mcaD(T)$ and for each $\phi\in\mcaD(T)$,
$$\sup_{|\theta|\leq1}\|Te^{i\theta A}\phi\|<\infty.$$
\item [({\rm$c_n$})] The form $i[T,A]$ defined on $\mcaD(T)\cap\mcaD(A)$ is bounded from below and closable. The self-adjoint operator associated with its closure is denoted by $iB_1$. Assume $\mcaD(T)\subseteq\mcaD(B_1)$. If $n>1$, assume for $j=2,\cdots,n$ that the form $i[iB_{j-1},A]$, defined on $\mcaD(T)\cap\mcaD(A)$, is bounded from below and closable. The associated self-adjoint operator is denoted by $iB_{j}$, and it is assumed that $\mcaD(T)\subseteq\mcaD(B_j)$.
\item [({\rm$d_n$})] The form $[B_n, A]$, defined on $\mcaD(T)\cap\mcaD(A)$, extends to a bounded operator from $X_{+2}$ to $X_{-2}$.
\item [(e)] There exist $\alpha>0, \delta>0$ and a compact operator $K$ on $X$ such that
\begin{equation}\label{Mourre}
E_{T}(\mcaJ)iB_1E_{T}(\mcaJ)\geq\alpha E_{T}(\mcaJ)+E_{T}(\mcaJ)KE_{T}(\mcaJ),
\end{equation}
where $\mcaJ=(E-\delta,E+\delta)$ is called the interval of conjugacy.
\end{itemize}
}
\end{definition}
\begin{theorem}\label{Resol-Smoo}
(\cite[Theorem 2.2]{JMP84}) Let $T$ be as above and $n\geq1$ an integer. Let $A$ be a conjugate operator to $T$ at $E\in\R$. Assume that $T$ is $n$-smooth with respect to $A$. Let $\mcaJ$ be the interval of conjugacy and $I\subseteq \mcaJ\cap\sigma_c(T)$ a relatively compact interval. Let $s>n-\frac{1}{2}$.
 \begin{itemize}
\item [(i)]For ${\Re}z\in I,{\Im}z\neq0$, one has
$$\|\left<A\right>^{-s}(T-z)^{-n}\left<A\right>^{-s}\|\leq c.$$
\item [(ii)] For ${\Re}z,{\Re}z'\in I,\ 0<|{\Im}z|\leq1,\ 0<|{\Im}z'|\leq1$, there exists a constant $C$ independent of $z,z'$, such that
    $$\|\left<A\right>^{-s}((T-z)^{-n}-(T-z')^{-n})\left<A\right>^{-s}\|\leq C|z-z'|^{\delta_1},$$
    where
    \begin{equation}\label{delta}
    \delta_1=\delta_1(s,n)=\frac{1}{1+\frac{sn}{s-n+\frac{1}{2}}}.
    \end{equation}
\item [(iii)] Let $\lambda\in I$. The norm limits
$$\lim\limits_{\varepsilon\downarrow0}\left<A\right>^{-s}(T-\lambda\pm i\varepsilon)^{-n}\left<A\right>^{-s}$$
exist and equal
$$\left(\frac{d}{d\lambda}\right)^{n-1}(\left<A\right>^{-s}(T-\lambda\pm i0)^{-1}\left<A\right>^{-s}),$$
where
$$\left<A\right>^{-s}(T-\lambda\pm i0)^{-1}\left<A\right>^{-s}=\lim\limits_{\varepsilon\downarrow0}\left<A\right>^{-s}(T-\lambda\pm i\varepsilon)^{-1}\left<A\right>^{-s}.$$
The norm limits are H\"{o}lder continuous with exponent $\delta_1(s,n)$ given above.

\end{itemize}
\end{theorem}
\begin{remark}
 {\rm We remark that the interval $I$ in Theorem \ref{Resol-Smoo} can not only be restricted in $\mcaJ\cap\sigma_c(T)$ but can actually be taken as $\mcaJ\setminus\sigma_{p}(T).$ }
\end{remark}

In verifying the conditions of this theorem, particularly conditions $(c_n)$ and $(d_n)$, it is often more convenient to examine the regularity of $T$ with respect to a suitable conjugate operator. To this end, we will revisit this concept for the case where $T$ is a bounded self-adjoint operator and present some sufficient conditions to judge this regularity. %collect some sufficient conditions under . This will facilitate the discussion for the context of this paper.
%To apply Theorem \ref{Resol-Smoo}, as outlined in Definition \ref{condi-veri}, the verification of conditions $(c_{n})\sim(d_n)$ can be quite lengthy. Therefore, for the purpose of this paper, we will assume that $T$ is bounded and proceed to gather some sufficient conditions related to this matter.

Let $T$ be a bounded operator. For each integer $k$, we denote $ad^{k}_{A}(T)$ as the sesquilinear form on $\mcaD(A^k)$ defined iteratively as follows:
%For convenience, we denote by $ad^{1}_{A}(T):=[T,A]=TA-AT$ be the sesquilinear form defined in \eqref{Sesqui-form} above. And for each integer $k\in\N$, let $ad^{0}_{A}(T)=T$, we define $ad^{k}_{A}(T)$ on $\mcaD(A^k)$ by the following iteration
\begin{align}\label{adk}
\begin{split}
ad^{0}_A(T)&=T,\\
ad^{1}_{A}(T)&=[T,A]=TA-AT,\\
ad^{k}_{A}(T)&=ad^{1}_{A}\left(ad^{k-1}_{A}(T)\right)=\sum\limits_{i,j>0,i+j=k}^{}\frac{k!}{i!j!}(-1)^{i}A^{i}TA ^{j}.
\end{split}
\end{align}
\begin{definition}\label{def of regularity}
{\rm
Given an integer $k\in\N^{+}$, we say that $T$ is of $C^{k}(A)$, denoted by $T\in C^{k}(A)$, if the sesquilinear form $ad^{k}_{A}(T)$ admits a continuous extension to $X$. We identify this extension with its associated bounded operator in $X$ and denote it by the same symbol.}
\end{definition}
\begin{remark}
{\rm This property is often referred to as the regularity of $T$ with respect to $A$ in many contexts. Specifically, $T\in C^{k}(A)$ holds if and only if the vector-valued function $f(t)\phi$ on $\R$ has the usual $C^{k}(\R)$ regularity for every $\phi\in X$, where $f$ is defined as follows:
\begin{align*}
f:~~\R\longrightarrow\B(X),\ t\longmapsto f(t)=e^{itA}Te^{itA}.
\end{align*}}
\end{remark}
Moreover, this property satisfies the following algebraic structure.
\begin{lemma}\label{Regularity lemma}
{  For any $k\in\N^{+}$, let $T_1,T_2$ be bounded self-adjoint operators on $X$ such that $T_1,T_2\in C^{k}(A)$. Then, $T_1+T_2\in C^{k}(A)$ and $ad^{k}_{A}(T_1+T_2)=ad^{k}_{A}(T_1)+ad^{k}_{A}(T_2)$.
}
\end{lemma}
\begin{proof}
The result follows from the case $k=1$ established in \cite[Section 2]{GGM04}, combined with an inductive argument.
\end{proof}
As an application, in particular, we consider $X=\ell^2(\Z)$, $T=H=\Delta^2+V$, where $|V(n)|\lesssim \left<n\right>^{-\beta}$ for some $\beta>0$, and let $A$ be defined as in \eqref{A}. We then establish the following regularity property of $H$ with respect to $A$:
\begin{lemma}\label{regularity of H}
{ Let $H=\Delta^2+V$, where $|V(n)|\lesssim \left<n\right>^{-\beta}$ with $\beta>1$ and let $A$ be defined as in \eqref{A}. Then, $H\in C^{[\beta]}(A)$.}
\end{lemma}
\begin{proof}
First, we note that \cite[Lemma 4.1]{BS99} establishes that $ad^{1}_{iA}(-\Delta)=-\Delta(4+\Delta)$. Based on this, we claim that $\Delta^2\in\C^{\infty}(A)$. To verify this, a direct calculation yields
$$ad^{1}_{iA}(\Delta^2)=(-\Delta)[ad^{1}_{iA}(-\Delta)]+[ad^{1}_{iA}(-\Delta)](-\Delta)=2\Delta^2(4+\Delta).$$
Thus, $\Delta^2\in C^1(A)$. By repeating a similar decomposition process, one can find that for any $k\in\N^{+}$, $ad^{k}_{iA}(\Delta^2)$ is a polynomial about $-\Delta$ of degree $2+k$. Consequently, $\Delta^2\in C^{k}(A)$ for all $k$, i.e., $\Delta^2\in C^{\infty}(A)$.

As for the potential $V$, \cite[Proposition 5.1]{BS99} proves that $V\in C^{k}(A)$ for some positive integer $k$ if $V(n)$ satisfies the following decay condition:
$$V(n)\rightarrow0\ {\rm and}\  |(\mcaP^{k}V)(n)|=O(|n|^{-k}),\quad|n|\rightarrow\infty.$$
Under our assumption on $V$, we conclude that $V\in C^{[\beta]}(A)$. Combining this with the result for $\Delta^2$ and applying Lemma \ref{Regularity lemma}, the proof is complete.
\end{proof}
%\section{The sharpness of the decay estimates for free discrete bi-Schr\"{o}dinger equation}
\normalem

\end{document}